\documentclass[10pt,a4paper,reqno]{amsart}
\usepackage{amssymb,amsmath,amsthm,amstext,amsfonts}

\usepackage{mathrsfs}

\usepackage{enumitem}

\usepackage{hyperref}
\RequirePackage[dvipsnames]{xcolor} 
\definecolor{halfgray}{gray}{0.55}
\definecolor{webgreen}{rgb}{0,0.5,0}
\definecolor{webbrown}{rgb}{.6,0,0} \hypersetup{%
	colorlinks=true, linktocpage=true, pdfstartpage=3,
	pdfstartview=FitV,%
	breaklinks=true, pdfpagemode=UseNone, pageanchor=true,
	pdfpagemode=UseOutlines,%
	plainpages=false, bookmarksnumbered, bookmarksopen=true,
	bookmarksopenlevel=1,%
	hypertexnames=true,
	pdfhighlight=/O,
	urlcolor=webbrown, linkcolor=RoyalBlue,
	citecolor=webgreen, 
	pdftitle={},%
	pdfauthor={},%
	pdfsubject={2020 Mathematics Subject Classification},%
	pdfkeywords={},%
	pdfcreator={pdfLaTeX},%
	pdfproducer={LaTeX with hyperref}%
}

\newtheorem{theorem}{Theorem}[section]
\newtheorem{lemma}[theorem]{Lemma}
\newtheorem{corollary}[theorem]{Corollary}
\newtheorem{proposition}[theorem]{Proposition}
\theoremstyle{definition}
\newtheorem{definition}[theorem]{Definition}
\newtheorem{remark}[theorem]{Remark}
\numberwithin{equation}{section}
\numberwithin{figure}{section}

\newcommand{\R}{\mathbb R}
\newcommand{\C}{\mathbb C}
\newcommand{\GL}{\operatorname{GL}}
\newcommand{\Gr}{\operatorname{Gr}}
\newcommand{\Id}{\operatorname{Id}}
\newcommand{\Span}{\operatorname{span}}
\newcommand{\Fix}{\operatorname{Fix}}

\newcommand{\bol}{\operatorname{bol}}
\newcommand{\supp}{\operatorname{supp}}
\newcommand{\dist}{\operatorname{dist}}
\newcommand{\norm}[1]{\left\lVert #1\right\rVert}
\newcommand{\abs}[1]{\left|#1\right|}
\newcommand{\cF}{\mathcal F}
\newcommand{\cH}{\mathcal H}

\usepackage{orcidlink}

\title[Lyapunov exponents and dominated splittings]{From a Gap in the Lyapunov Spectrum to Dominated Splittings}

\author[Lucas Backes]{Lucas Backes \orcidlink{0000-0003-3275-1311}}
\address{\noindent Departamento de Matem\'atica, Universidade Federal do Rio Grande do Sul, Av. Bento Gon\c{c}alves 9500, CEP 91509-900, Porto Alegre, RS, Brazil.}
\email{lucas.backes@ufrgs.br}

\date{}
\subjclass[2020]{37H15, 37D30}
\keywords{linear cocycles, dominated splitting, gap in the Lyapunov spectrum, holonomies, fiber bunching}

\begin{document}
\begin{abstract}
Let $A$ be an $\alpha$-H\"older cocycle over a transitive two-sided subshift of finite type, and assume that $A$ is strongly bunched. We prove that if there exist $k\in\{1,\ldots,d-1\}$ and $c>0$ such that
\[
 \lambda_k(p)-\lambda_{k+1}(p)\ge c
\]
for every periodic point $p$, then $A$ admits a dominated splitting of index $k$.
Thus, under strong bunching, a uniform asymptotic gap in the Lyapunov spectrum of periodic measures implies a uniform geometric splitting along all orbit segments. This extends the periodic-gap criterion of Kassel--Potrie from locally constant cocycles to H\"older cocycles in arbitrary dimension, and extends the two-dimensional fiber-bunched result of Velozo to higher dimensions under the stronger bunching assumption.
\end{abstract}
\maketitle

\section{Introduction}\label{sec:introduction}
Let $f:X\to X$ be a homeomorphism of a compact metric space and let
$A:X\to\GL(d,\R)$ be a linear cocycle. We say that $A$
admits a \emph{dominated splitting of index $k$} if there is a continuous
$A$-invariant decomposition
\[
 \R^d=E(x)\oplus F(x)
 \; \text{ with }\;
 \dim E(x)=k,
\]
such that the ``greatest expansion'' of $A(x)$ along $F(x)$ is smaller than the ``greatest contraction'' of $A(x)$ along $E(x)$ by a factor that becomes exponentially small as time evolves. More precisely, there are constants $C>0$ and $\tau\in(0,1)$ such that
\[
 \frac{\|A^n(x)|_{F(x)}\|}
      {m(A^n(x)|_{E(x)})}
 \le C\tau^n
\]
for every $x\in X$ and every $n\ge1$, where $m(\cdot)$ denotes the conorm. Bochi and Gourmelon \cite{BochiGourmelon} provided an alternative characterization of this property showing that it is equivalent to the existence of a uniform exponential separation between the $k$-th and $(k+1)$-st singular values of $A^n(x)$. This uniform separation has an immediate asymptotic consequence. Indeed, if $A$ admits a dominated splitting of index $k$, then there exists $c>0$ such that
\[
 \lambda_k(\mu)-\lambda_{k+1}(\mu)\ge c
\]
for every ergodic invariant probability measure $\mu$ where $\lambda_j(\mu)$ denotes the $j$-th Lyapunov exponent of $A$ with respect to $\mu$ (see \cite[Fact 2.5]{KasselPotrie}). Thus domination, which is a uniform geometric property along all orbit segments, forces a uniform gap in the corresponding Lyapunov exponents, which are asymptotic quantities defined almost everywhere.

The present paper is motivated by the converse question: under what conditions
does a uniform gap between Lyapunov exponents imply the existence of a dominated
splitting? In general, such a converse is false. Indeed, Velozo
\cite{Velozo} constructed a H\"older $SL(2,\R)$-cocycle over a subshift of
finite type whose Lyapunov exponents are uniformly separated on a set of full
probability, but which is not uniformly hyperbolic. Since, in dimension two,
uniform hyperbolicity is equivalent to the existence of a dominated splitting
of index one, this example shows that a uniform Lyapunov gap alone is not enough
to imply domination. Therefore, some additional hypothesis is needed in order
to obtain a positive result.

As in Velozo's setting, we take the base dynamics to be a subshift of finite
type and impose an additional condition linking the fiber dynamics to the
hyperbolicity of the base. Our main result shows that \emph{strong bunching} is
sufficient for this purpose in arbitrary dimension. More precisely, let $A$ be an $\alpha$-H\"older cocycle over a transitive
two-sided subshift of finite type, and assume that $A$ is strongly bunched in
the sense of Bochi--Garibaldi \cite{BochiGaribaldi} (see Section \ref{sec:fiber-bunching}). We prove that if there exist $k\in\{1,\ldots,d-1\}$ and $c>0$ such that
\[ \lambda_k(p)-\lambda_{k+1}(p)\ge c\]
for every periodic point $p$, then $A$ admits a dominated splitting of index
$k$.
Thus, the theorem turns asymptotic information about the Lyapunov spectrum associated to periodic measures into a uniform geometric property valid along every orbit segment. This is the central point of our result. As a simple consequence we get that, if $A$ is strongly bunched and has Lyapunov spectrum bounded away from zero in the sense that $\lambda_k(p)\geq c>0>-c\geq \lambda_{k+1}(p)$
for every periodic point $p$, then $A$ is uniformly hyperbolic.

The relation between gaps in the Lyapunov spectra and dominated
splittings has already been studied in the literature.
In \emph{dimension two}, Velozo proved that a fiber-bunched $SL(2,\R)$-cocycle over a
transitive subshift of finite type or a transitive Anosov diffeomorphism is
uniformly hyperbolic whenever its Lyapunov exponents have a uniform gap
\cite{Velozo}. In arbitrary dimension, Kassel and Potrie
\cite{KasselPotrie} proved that, for \emph{locally constant} cocycles over subshifts
of finite type, a uniform gap between the $k$-th and $(k+1)$-st Lyapunov
exponents of periodic measures forces a dominated splitting of index $k$.
For \emph{typical} fiber-bunched cocycles over subshifts of finite type, Park
\cite{ParkProximal} proved that a uniform gap between two consecutive
Lyapunov exponents at all periodic points likewise implies a dominated
splitting of the corresponding index.
DeWitt and Gogolev \cite{DeWittGogolev} obtained another periodic-data
criterion: they proved that a finite-dimensional cocycle over a hyperbolic
system admits the dominated splitting indicated by its periodic data when the
periodic data are constant, and also when they are sufficiently close to
constant. More recently, in \cite{BackesDominated} we have proved an analogous result for quasi-compact operator cocycles on Banach spaces, allowing
infinite-dimensional and possibly non-invertible fiber dynamics.

Note that our result differs from the results of DeWitt--Gogolev and also from our previous work
in that we do not assume that the periodic spectra are constant, or close to
a fixed spectrum. We only assume a uniform gap at one prescribed index.
Our theorem may therefore be viewed as extending the periodic-gap criterion
of Kassel--Potrie from locally constant cocycles to H\"older cocycles in
arbitrary dimension under the strong bunching assumption, without imposing
the typicality hypothesis appearing in Park's result. It also extends the
two-dimensional fiber-bunched result of Velozo to higher-dimensional
cocycles, although with the stronger bunching condition required by our
argument. In a different direction, Yemini \cite{Yemini} proved that if a
gap between two consecutive Lyapunov exponents persists uniformly in a
$C^0$-neighborhood of a cocycle, then the cocycle admits a dominated
splitting of the corresponding index.

\subsection{Why are we interested in dominated splittings?}
Uniform hyperbolicity is one of the most basic notions in the theory of dynamical systems. It provides a robust geometric structure and strong control of the
asymptotic behavior of systems. At the same time, uniform hyperbolicity is
too restrictive to describe many dynamical phenomena that arise naturally.
This led to the development of weaker forms of hyperbolic behavior which still
retain part of the uniform geometry of the hyperbolic setting. Among them,
dominated splittings have become especially important. The notion was developed by Ma\~n\'e in connection with the $C^1$ Stability Conjecture
\cite{ManeStability}, and has since played a fundamental role in the study of
nonuniformly hyperbolic and partially hyperbolic systems.

Dominated splittings are a basic tool in smooth
dynamics because they provide a robust form of separation between different
directions of the tangent dynamics. For the derivative cocycle of a
diffeomorphism, such splittings persist under $C^1$ perturbations and form the
geometric framework behind many notions of partial hyperbolicity. They occur
naturally in the study of robust transitivity, structural stability, generic
dynamics, and the organization of nonhyperbolic invariant sets. On surfaces,
for example, Pujals and Sambarino obtained a detailed description of the
dynamics of invariant sets carrying a dominated splitting
\cite{PujalsSambarino}. More generally, domination is strong enough to yield
uniform geometric information, while remaining flexible enough to occur far
beyond the uniformly hyperbolic setting. For this reason, criteria that allow
one to detect domination from asymptotic or periodic information are useful
both for linear cocycles and for smooth dynamical systems.

\section{Statements}\label{sec:definitions}

In this section, we fix the setting and notation and state the main theorem.

\subsection{Subshifts of finite type}

Let $Q=(q_{ij})_{1\le i,j\le \ell}$ be an $\ell\times\ell$ matrix with
$q_{ij}\in\{0,1\}$. The \emph{subshift of finite type} associated to the matrix $Q$ is the subset of the bi-infinite sequences $\{1,\dots, \ell\}^{\mathbb{Z}}$ satisfying
\[
 \widehat\Sigma=
 \{(x_n)_{n\in\mathbb Z}:q_{x_nx_{n+1}}=1\text{ for every }n\in\mathbb Z\}.
\]
We require that each row and column of $Q$ contains at least one nonzero entry. We let $ \hat f:\widehat\Sigma\to\widehat\Sigma$ be the left-shift map defined by $\hat f(x_n)_{n\in\mathbb Z}=(x_{n+1})_{n\in\mathbb Z}$ and assume that $\hat f$ is topologically transitive.

We will also use the one-sided spaces
\[
 \Sigma^u=\{(x_n)_{n\ge0}:q_{x_nx_{n+1}}=1\} \; \text{ and }\;
 \Sigma^s=\{(x_n)_{n\le0}:q_{x_nx_{n+1}}=1\}.
\]
Let $P^u:\widehat\Sigma\to\Sigma^u$ and
$P^s:\widehat\Sigma\to\Sigma^s$ be the natural projections and $f_{s}$ and $f_{u}$ denote the right and left shifts on $\Sigma^{s}$ and $\Sigma^{u}$, respectively. When only the left one-sided shift is used we write
\[
 \Sigma:=\Sigma^u,\quad P:=P^u:\widehat\Sigma\to\Sigma \; \text{ and }\; f:\Sigma\to\Sigma
\]
for the corresponding left shift.  Thus hatted symbols always refer to the
two-sided system and unhatted symbols to the one-sided left shift.

We define the \emph{local stable and unstable sets} of $\hat x\in\widehat\Sigma$ by
\[
 W^s_{\rm loc}(\hat x)=
 \{\hat y\in\widehat\Sigma:x_n=y_n\text{ for every }n\ge0\}
\]
and
\[
 W^u_{\rm loc}(\hat x)=
 \{\hat y\in\widehat\Sigma:x_n=y_n\text{ for every }n\le0\}.
\]
Equivalently, for $x\in\Sigma$ and $x^-\in\Sigma^s$ we write
\[
 W^s_{\rm loc}(x):=P^{-1}(x)\; \text{ and }\;
 W^u_{\rm loc}(x^-):=(P^s)^{-1}(x^-).
\]
Thus the fibers of $P$ are precisely the local stable sets.  For
$0<\theta<1$ we define the metric
\[
 d_\theta(\hat x,\hat y)=\theta^{N(\hat x,\hat y)} \; \text{ where }
 N(\hat x,\hat y)=\max\{N\ge0:x_n=y_n\text{ for all }|n|<N\}.
\]
If $\hat x$ and $\hat y$ have the same zero-coordinate, we write
\begin{equation}\label{eq:bracket-conv}
 [\hat x,\hat y]\in W^u_{\rm loc}(\hat x)\cap W^s_{\rm loc}(\hat y)
\end{equation}
for the point having the past of $\hat x$ and the future of $\hat y$.

Finally, for $m \in \mathbb{Z}$ and $a_{0},\dots,a_{k} \in \{1,\dots,\ell\}$, we define the cylinder notation
\[
[m;a_{0},\dots,a_{k}] = \{\hat{x} \in \widehat{\Sigma}: x_{m+i} = a_i, \, 0 \leq i \leq k\}.
\]

\subsection{Cocycles and Lyapunov exponents}

A $d$-dimensional \emph{linear cocycle over $\hat f$} is simply a map
$\hat A:\widehat\Sigma\to\GL(d,\R)$.  Its iterates are given by
\begin{equation*}\label{eq:cocycle-iterates}
 \hat A^n(\hat x)=
 \begin{cases}
  \hat A(\hat f^{n-1}\hat x)\cdots
  \hat A(\hat f\hat x)\hat A(\hat x),&n>0,\\
  \Id,&n=0,\\
  \bigl(\hat A^{-n}(\hat f^n\hat x)\bigr)^{-1},&n<0.
 \end{cases}
\end{equation*}

Let $\hat\mu$ be an ergodic $\hat f$-invariant probability measure on
$\widehat\Sigma$. Assuming that $\log^+\|\hat A\|$ and $\log^+\|\hat A^{-1}\|$ belong to $L^1(\hat\mu)$, Oseledets' theorem (see \cite{VianaLectures}) gives an integer $r_{\hat\mu}\in\{1,\ldots,d\}$ and distinct real numbers
\[
 \chi_1(\hat\mu)>\cdots>\chi_{r_{\hat\mu}}(\hat\mu),
\]
called the \emph{Lyapunov exponents}, and, for $\hat\mu$-almost every
$\hat x$, a decomposition
\[
 \R^d
 =
 E^1_{\hat x}\oplus\cdots\oplus E^{r_{\hat\mu}}_{\hat x},
\]
called the \emph{Oseledets splitting}. The subspaces
$E^i_{\hat x}$ depend measurably on $\hat x$ and satisfy
\[
 \hat A(\hat x)E^i_{\hat x}
 =
 E^i_{\hat f(\hat x)}.
\]
Moreover, for every nonzero $v\in E^i_{\hat x}$,
\[
 \lim_{n\to\pm\infty}
 \frac{1}{n}
 \log\|\hat A^n(\hat x)v\|
 =
 \chi_i(\hat\mu).
\]
In what follows, it will be convenient to write the \emph{Lyapunov spectrum} according to multiplicity as
\[
 \lambda_1(\hat\mu)\ge\cdots\ge\lambda_d(\hat\mu),
\]
where each $\chi_i(\hat\mu)$ is repeated
$\dim E^i_{\hat x}$ times.

We recall that, if $\hat p\in\widehat\Sigma$ is a periodic point of period $n$, then
\[
 \hat\mu_{\hat p}
 =
 \frac{1}{n}
 \sum_{j=0}^{n-1}
 \delta_{\hat f^j(\hat p)}
\]
is an ergodic $\hat f$-invariant probability measure supported on the
orbit of $\hat p$. Then, denoting by $\rho_1,\ldots,\rho_d$ the eigenvalues of
$\hat A^n(\hat p)$, repeated according to their algebraic multiplicities
and ordered so that $|\rho_1|\ge\cdots\ge|\rho_d|$, we have that the Lyapunov spectrum of $\hat\mu_{\hat p}$ is given simply by
\begin{equation*}\label{eq:periodic-exponents}
 \lambda_i(\hat p)
 :=
 \lambda_i(\hat\mu_{\hat p})
 =
 \frac{1}{n}\log|\rho_i|
 \; \text{ for }
 1\le i\le d.
\end{equation*}

\subsection{Fiber bunching and strong bunching}\label{sec:fiber-bunching}
Given an invertible linear map $L$, we write $m(L)=\norm{L^{-1}}^{-1}$ for the \emph{conorm} of $L$ and, following \cite{BochiGaribaldi}, we define its  \emph{bolicity} as
\[
 \bol(L)=\norm L\,\norm{L^{-1}}.
\]
Note that $\bol(L)$ measures the lack of conformality of $L$.
Then, we say that an $\alpha$-H\"older cocycle $\hat A$ is \emph{$\alpha$-fiber-bunched} if there exists $N\ge1$ such that
\begin{equation*}\label{eq:fiber-bunching}
 \bol(\hat A^N(\hat x))\,\theta^{N\alpha}<1
 \;\text{ for every }\hat x\in\widehat\Sigma.
\end{equation*}
Moreover, we will also use a notion of strong bunching which comes from \cite{BochiGaribaldi}, specialized to the
standard metric $d_\theta$ on a subshift of finite type.

\begin{definition}\label{def:strong-bunching}
An $\alpha$-H\"older cocycle
$\hat A:\widehat\Sigma\to\GL(d,\R)$ is \emph{strongly bunched} if
\begin{enumerate}[label=\textup{(\roman*)}]
\item when $d=2$, $\hat A$ is fiber-bunched;
\item when $d\ge3$, there exists $N\ge1$ such that
\begin{equation}\label{eq:strong-bunching}
 \bol(\hat A^N(\hat x))\,\theta^{N\alpha/3}<1
 \;\text{ for every }\hat x\in\widehat\Sigma.
\end{equation}
\end{enumerate}
\end{definition}

\begin{remark}
We observe that our notion of strong bunching is the eventual form of the strong bunching condition of Bochi--Garibaldi, specialized to the symbolic metric used here. Their notation
uses $\theta$ for the H\"older exponent, whereas here the H\"older exponent is
denoted by $\alpha$ and $\theta\in(0,1)$ denotes the parameter in the metric
$d_\theta$. For this metric, the stable and unstable contraction factors are
both $\theta$, so the corresponding hyperbolicity exponents in the notation of
\cite{BochiGaribaldi} may be taken to be
\[
 \lambda_s=\lambda_u=\Lambda_u=-\log\theta.
\]
It then follows from the explicit choice in
\cite[Lemma~3.6 and formula~(A.24)]{BochiGaribaldi} that we may take
\[
 \eta_0=\frac{\alpha}{3}.
\]
Thus the $(\eta_0,\alpha)$-bunching condition of Bochi--Garibaldi becomes \eqref{eq:strong-bunching}, after passing to an iterate. The denominator $3$ is therefore the one supplied by the cited estimate. 
\end{remark}

\subsection{Dominated splittings}
We say that the cocycle $\hat A$ has a \emph{dominated splitting of index $k$} if there is a continuous $\hat A$-invariant decomposition
\[
 \R^d=E(\hat x)\oplus F(\hat x) \; \text{ with }\; \dim E(\hat x)=k,
\]
and constants $C>0$ and $0<\tau<1$ such that
\[
 \frac{\norm{\hat A^n(\hat x)|_{F(\hat x)}}}
 {m(\hat A^n(\hat x)|_{E(\hat x)})}
 \le C\tau^n
 \; \text{ for every }\hat x\in\widehat\Sigma \text{ and }n\ge0.
\]

\subsection{Main result} The following is the main result of this paper.

\begin{theorem}\label{thm:main}
Let $\hat f:\widehat\Sigma\to\widehat\Sigma$ be a transitive two-sided
subshift of finite type and let
$\hat A:\widehat\Sigma\to\GL(d,\R)$ be an $\alpha$-H\"older cocycle which
is strongly bunched in the sense of Definition~\ref{def:strong-bunching}.  Fix
$k\in\{1,\ldots,d-1\}$.  Assume that there exists $c>0$ such that
\begin{equation}\label{eq:periodic-gap}
 \lambda_k(\hat p)-\lambda_{k+1}(\hat p)\ge c
\end{equation}
for every periodic point $\hat p$.  Then $\hat A$ admits a dominated splitting
of index $k$.
\end{theorem}

For an explanation of the role of the strong bunching hypothesis in our proof, we refer the reader to Remark \ref{rem:role-strong-bunching} below.

\begin{remark}
    It was observed in \cite[Example 3.8]{KasselPotrie} that the uniformity given by the constant $c$ in \eqref{eq:periodic-gap} is really necessary to get a dominated splitting even in the locally constant setting. More precisely, they constructed a locally constant cocycle $\hat A:\widehat \Sigma \to \GL(3,\R)$ over the full shift on  two symbols satisfying $\lambda_2(\hat p)>\lambda_3(\hat p)$ for every periodic point $\hat p$ which does not admit a dominated splitting.
\end{remark}

As a simple consequence of our main result we have the following.
\begin{corollary}\label{cor:corollary-hyperbolic}
    Let $\hat f$, $\hat A$ and $k$ be as in Theorem \ref{thm:main}. Moreover, suppose there exists $c>0$ such that
    \begin{equation}\label{eq:periodic-gap-corollary}
 \lambda_k(\hat p)\geq c>0>-c\geq \lambda_{k+1}(\hat p)
\end{equation}
for every periodic point $\hat p$.  Then $\hat A$ is uniformly hyperbolic.
\end{corollary}

\begin{proof}[Proof of Corollary \ref{cor:corollary-hyperbolic}] We start observing that, by \eqref{eq:periodic-gap-corollary}, for every periodic point $\hat p$ we have
    \begin{equation*}
        \lambda_k(\hat p) - \lambda_{k+1}(\hat p) \ge 2c > 0.
    \end{equation*}
    Thus, Theorem \ref{thm:main} guarantees that the cocycle $\hat A$ admits a dominated splitting of index $k$, which we denote by $\mathbb R^d= E(\hat x) \oplus F(\hat x)$ where $\dim E(\hat x) = k$. 

    Define  $a_n(\hat x) = \log \|\hat A^n(\hat x)|_{F(\hat x)}\|$. Since the dominated splitting is continuous, we have that $(a_n(\hat x))_{n \in \mathbb{N}}$ is a sequence of continuous and subadditive functions. Thus, by Kingman's subadditive theorem combined with the fact that the splitting above is dominated we have that
    \[\lim_{n\to +\infty} \frac{a_n(\hat x)}{n}=\inf_n\frac{1}{n}\int a_n(\hat x)d\hat \mu (\hat x)=\lambda_{k+1}(\hat \mu) \]
    for $\hat \mu$-almost every $\hat x\in \widehat \Sigma$ and every ergodic $\hat f$-invariant measure $\hat \mu$. Moreover, since the Lyapunov exponents of ergodic measures can be approximated by Lyapunov exponents on periodic points \cite[Theorem 1.4]{Kal}, it follows that 
    \[\lambda_{k+1}(\hat \mu)\leq -c \]
    for every ergodic measure $\hat \mu$.
Thus, by \cite[Theorem A.3]{Mor} we get that
\[\lim_{n\to +\infty} \sup_{\hat x}\frac{a_n(\hat x)}{n}=\sup_{\hat \mu}\inf_n\frac{1}{n}\int a_n(\hat x)d\hat \mu (\hat x)\leq -c \]
where the supremum is taken over all ergodic $\hat{f}$-invariant  measures. Consequently, given $\varepsilon\in (0,c)$ there exists $C>0$ such that 
\[\|\hat A^n(\hat x)|_{F(\hat x)}\|\leq Ce^{-(c-\varepsilon)n} \;\text{ for every } \hat x\in \widehat\Sigma \text{ and }n\in \mathbb N.\]
Applying a similar reasoning to $b_n(\hat x) = \log \left\|\left(\hat A^n(\hat x)|_{E(\hat x)}\right)^{-1}\right\|$ we conclude that for every $\varepsilon\in (0,c)$ there exist $C>0$ such that
\[ \left\|\left(\hat A^n(\hat x)|_{E(\hat x)}\right)^{-1}\right\| \leq   Ce^{(c-\varepsilon)n}  \;\text{ for every } \hat x\in \widehat\Sigma \text{ and }n\in \mathbb N,\]
which implies that 
\[ m\left(\hat A^n(\hat x)|_{E(\hat x)}\right)
       \ge C^{-1}e^{(c-\varepsilon)n}
       \;\text{ for every } \hat x\in \widehat\Sigma \text{ and }n\in \mathbb N. \]
Consequently, $\hat{ A}$ is uniformly hyperbolic as claimed.
\end{proof}

\begin{remark}
An analogous version of Theorem \ref{thm:main} holds when the base dynamics $\hat f:\widehat\Sigma\to\widehat\Sigma$ is replaced, for instance, by a transitive Anosov diffeomorphism. The proof proceeds via standard techniques using Markov partitions, so we omit the details.
\end{remark}

\begin{remark}\label{rem:mixing-reduction}
We will sometimes reduce from a transitive subshift of finite type to a mixing
one. Recall that a transitive subshift of finite type admits a finite cyclic
decomposition $\widehat\Sigma =\widehat\Sigma_0\sqcup\cdots\sqcup\widehat\Sigma_{r-1}$ where $\hat f(\widehat\Sigma_i)=\widehat\Sigma_{i+1\pmod r}$
and $\hat f^r|_{\widehat\Sigma_i}$ is mixing for every $i$.
Consider the cocycle $\hat A^r$ over $\hat f^r$ restricted to one
cyclic component. Its Lyapunov exponents are $r$ times the corresponding
Lyapunov exponents of $\hat A$. Hence the periodic-gap hypothesis is
preserved, with the constant $c$ replaced by $rc$. Moreover, if
$\hat A^r$ admits a dominated splitting of index $k$ on one cyclic
component, then this splitting can be transported by the cocycle to the other
components and gives a dominated splitting of index $k$ for $\hat A$ over
$\widehat\Sigma$. Thus, whenever convenient, we may pass to a mixing component
and to the corresponding iterate.
\end{remark}

\section{Preliminaries I: general constructions}\label{sec:prelim}
In this section, we gather some useful preliminary results and constructions that are going to be used in the proof of our main result.
We keep the notation of Section \ref{sec:definitions} throughout the paper.
Thus $\hat x,\hat y,\ldots$ always denote points of $\widehat\Sigma$, while $x,y,\ldots$ are reserved for
points of the one-sided shift $\Sigma$.  We regard $\hat A$ as acting on
the trivial bundle $\widehat\Sigma\times\R^d$ and denote the fiber over
$\hat x$ by $E_{\hat x}$ when the base point needs to be displayed.

\subsection{Characterization of dominated splittings}
We start recalling a characterization of dominated splittings due to Bochi and
Gourmelon \cite{BochiGourmelon}, which will play an important role in our
proof. Their characterization is formulated in terms of singular values.
Given an invertible matrix $L\in\GL(d,\R)$, let $\sigma_1(L)\ge\cdots\ge\sigma_d(L)>0$ denote its singular values.

\begin{theorem}\label{thm:BG-domination}
The cocycle $\hat A$ has a dominated splitting of index $k$ if and only if
there are $C>0$ and $0<\tau<1$ such that
\[
 \frac{\sigma_{k+1}(\hat A^n(\hat x))}
      {\sigma_k(\hat A^n(\hat x))}
 \le C\tau^n \; \text{ for every }
\hat x\in\widehat\Sigma \text{ and } n\ge0.
\]
\end{theorem}

\subsection{Grassmannians, Exterior powers and hyperplane sections}\label{sec: grassmanian}

\subsubsection{Grassmannians} We use the following Grassmannian notation throughout the paper.  If $V$ is a
$d$-dimensional real vector space, $\Gr(q,V)$ denotes the Grassmannian of
$q$-dimensional linear subspaces of $V$, and we abbreviate
\[
 \Gr(q,d):=\Gr(q,\mathbb R^d).
\]
If $E\to X$ is a vector bundle, then
\[
 \Gr_q(E):=\bigsqcup_{x\in X}\Gr(q,E_x)
\]
denotes the corresponding Grassmannian bundle.  Thus expressions such as
$\Gr(q,E_{\hat p})$ refer to a single fiber, whereas $\Gr_q(E)$ refers to
the whole bundle.

\subsubsection{Exterior powers}
For $1\le q\le d-1$, recall that the $q$-th exterior power
$\Lambda^q\mathbb R^d$ is the vector space generated by formal wedges
$v_1\wedge\cdots\wedge v_q$, which is multilinear in the vectors and alternating in the
sense that the wedge vanishes whenever two entries coincide.  It has dimension
$\binom dq$.  Every linear map $L:\mathbb R^d\to\mathbb R^d$ induces a linear
map $\Lambda^qL:\Lambda^q\mathbb R^d\longrightarrow\Lambda^q\mathbb R^d$ given by
\[
 (\Lambda^qL)(v_1\wedge\cdots\wedge v_q)
 =Lv_1\wedge\cdots\wedge Lv_q.
\]
With the Euclidean structure induced from $\mathbb R^d$, the singular values of
$\Lambda^qL$ are precisely the products
$\sigma_{i_1}(L)\cdots\sigma_{i_q}(L)$ with
$1\le i_1<\cdots<i_q\le d$.  In particular, its two largest singular values are
\[
 \sigma_1(\Lambda^qL)=\sigma_1(L)\cdots\sigma_q(L)
\]
and
\[
 \sigma_2(\Lambda^qL)=
 \sigma_1(L)\cdots\sigma_{q-1}(L)\sigma_{q+1}(L),
\]
so
\begin{equation*}\label{eq:wedge-gap}
 \frac{\sigma_2(\Lambda^qL)}{\sigma_1(\Lambda^qL)}
 =\frac{\sigma_{q+1}(L)}{\sigma_q(L)}.
\end{equation*}

A nonzero vector $v_1\wedge\cdots\wedge v_q\in\Lambda^q\R^d$ is said to be
\emph{decomposable}.  More generally, for any finite-dimensional real vector
space $V$ the \emph{Pl\"ucker embedding} $\iota_{q,V}:\Gr(q,V)\longrightarrow\mathbb P(\Lambda^qV)$ is given by
\begin{equation*}\label{eq:plucker-embedding}
 \iota_{q,V}(\Span\{v_1,\ldots,v_q\})
 =[v_1\wedge\cdots\wedge v_q]
\end{equation*}
where $\mathbb P(\Lambda^qV)$ denotes the projective space of one-dimensional
linear subspaces of $\Lambda^qV$, and $[v_1\wedge\cdots\wedge v_q]$ denotes
the line spanned by the nonzero decomposable vector
$v_1\wedge\cdots\wedge v_q$.
When the ambient vector space is clear we simply write $\iota_q$. In
particular, $\iota_q:\Gr(q,d)\to\mathbb P(\Lambda^q\R^d)$ in the standard
fiber.  Its image is exactly the set of projective lines generated by nonzero
decomposable $q$-vectors.  We occasionally identify a Grassmannian with its
Pl\"ucker image, but whenever a measure on the Grassmannian is involved we
will keep track of the inverse image under $\iota_q$ explicitly.

\subsubsection{Hyperplane sections} If $S$ is a $(d-q)$-plane, the set
\[
 \cH(S)=\{V\in\Gr(q,d):V\cap S\ne\{0\}\}
\]
is the corresponding \emph{hyperplane section}.  Equivalently, for a suitable nonzero covector
$\ell\in(\Lambda^q\R^d)^*$ one has
\[
 \cH(S)=\iota_q^{-1}\bigl(\mathbb P(\ker\ell)\bigr).
\]
Note that $\ell$ is a nonzero linear functional such that \[\ell(v_1\wedge\cdots\wedge v_q)=0\]
whenever
\[
\operatorname{span}\{v_1,\ldots,v_q\}\cap S\ne{0}.
\]

\subsubsection{Distances}
We fix the standard metric $d_{\Gr}$ induced by orthogonal projections on every Grassmannian which appears below. We will use in the sequel the following estimate
\begin{equation}\label{eq:prelim-grass-lip}
 d_{\Gr}(LV,LW)\le C_d\bol(L)d_{\Gr}(V,W),
\end{equation}
valid for every $L\in\GL(d,\R)$ and $V,W\in\Gr(q,d)$.
 Moreover, we use the corresponding
projective metric on $\mathbb P(\Lambda^q\R^d)$ and denote Hausdorff distance
by $d_H$. Since all these spaces are compact, replacing these metrics by any
other standard equivalent metrics only changes the constants in the estimates.

 When probabilities on $\Gr(q,d)$ are considered, $W_1$ denotes the \emph{Wasserstein-$1$ distance} associated to $d_{\Gr}$ given by
\[
 W_1(\nu_1,\nu_2)
 =\inf_{\pi\in\Pi(\nu_1,\nu_2)}
   \int d_{\Gr}(V,W)\,d\pi(V,W)
\]
where $\Pi(\nu_1,\nu_2)$ denotes the set of \emph{couplings} of
$\nu_1$ and $\nu_2$: these are the Borel probability measures $\pi$ on
$\Gr(q,d)\times\Gr(q,d)$ whose first marginal is $\nu_1$ and whose second
marginal is $\nu_2$.  Equivalently, for every Borel set $E\subset\Gr(q,d)$,
\[
 \pi(E\times\Gr(q,d))=\nu_1(E)\; \text{ and }\;
 \pi(\Gr(q,d)\times E)=\nu_2(E).
\]
Thus $W_1(\nu_1,\nu_2)$ is the least possible average distance
$d_{\Gr}(V,W)$ among all joint distributions of $V$ and $W$ having the prescribed
marginals $\nu_1$ and $\nu_2$.

\subsection{Finite extensions}\label{sec:finite extensions}

At two points in the proof we will pass to a finite extension of the two-sided
system. We recall here what this means and record the properties that will be
used later.

Let $F$ be a finite set and let
\[
 \beta:\widehat\Sigma\longrightarrow\operatorname{Sym}(F)
\]
be a locally constant map, where $\operatorname{Sym}(F)$ denotes the group of
permutations of $F$. After replacing the base, if necessary, by a
topologically conjugate higher-block presentation, we may assume that
$\beta(\hat x)$ depends only on the current symbol. The corresponding
\emph{finite extension} is the skew product
\[
  \hat f_{\mathrm{ext}}(\hat x,a)
 :=
 \bigl(\hat f\hat x,\beta(\hat x)a\bigr)
\]
acting on $\widehat\Sigma_{\mathrm{ext}}:= \widehat\Sigma\times F$.
The natural projection $\pi_{\mathrm{ext}}:
 \widehat\Sigma_{\mathrm{ext}}\longrightarrow\widehat\Sigma$ given by $\pi_{\mathrm{ext}}(\hat x,a)=\hat x$
satisfies
\[
 \pi_{\mathrm{ext}}\circ\hat f_{\mathrm{ext}}
 =
 \hat f\circ\pi_{\mathrm{ext}}.
\]
Thus every point of the base has finitely many lifts. The additional
coordinate records which element of a finite family is being followed along
the orbit. In the applications below, the elements of $F$ label finitely many
subspaces or components, and $\beta(\hat x)$ records the permutation of these
labels induced when one passes from $\hat x$ to $\hat f\hat x$.

As a simple example, consider the full two-shift and take
$F=\mathbb Z/2\mathbb Z$. If
\[
 \beta(\hat x)a=a+x_0\pmod 2,
\]
then
\[
 \hat f_{\mathrm{ext}}(\hat x,a)
 =
 \bigl(\hat f\hat x,a+x_0 \pmod 2\bigr).
\]
The second coordinate records the parity of the number of symbols $1$
encountered along the orbit. The extensions used below are of the same type,
with a general finite set $F$ and an arbitrary permutation-valued cocycle
$\beta$.

The cocycle on the extension is the pullback of $\hat A$ by
$\pi_{\mathrm{ext}}$ given by
\[
 \hat A_{\mathrm{ext}}(\hat x,a) := \hat A(\hat x).
\]
Consequently,
\begin{equation}\label{eq:finite-extension-iterates}
 \hat A_{\mathrm{ext}}^n(\hat x,a) = \hat A^n(\hat x) \; \text{ for every }\; n\ge1.
\end{equation}
In particular, the linear dynamics along an orbit in the extension is exactly
the same as that along its projection to the base.

Let now $\hat p$ be a periodic point of $\hat f$ of period $n$. After
$n$ iterates, a lift $(\hat p,a)$ returns to the same base point, but its
finite coordinate is transformed by the permutation
\[
 \beta^{(n)}(\hat p)
 :=
 \beta(\hat f^{n-1}\hat p)\cdots\beta(\hat p).
\]
Hence $(\hat p,a)$ need not be periodic for
$\hat f_{\mathrm{ext}}^n$. However, since $F$ is finite, there exists
$1\le r\le |F|$ such that
\[
 \bigl(\beta^{(n)}(\hat p)\bigr)^r a=a.
\]
The point $(\hat p,a)$ is therefore periodic for the extension with period
dividing $rn$, and the corresponding return matrix over $rn$ iterates is
\[
 \hat A_{\mathrm{ext}}^{rn}(\hat p,a)
 =
 \bigl(\hat A^n(\hat p)\bigr)^r.
\]
It follows that the normalized Lyapunov exponents of the lifted periodic
orbit are the same as those of the periodic orbit downstairs. In particular,
periodic Lyapunov gaps are unchanged under passage to the finite extension.

We also record how domination behaves under this construction. If
$\hat A$ admits a dominated splitting of index $k$, then its pullback
$\hat A_{\mathrm{ext}}$ clearly admits a dominated splitting of the same
index. Conversely, suppose that the restriction of
$\hat A_{\mathrm{ext}}$ to every transitive component of
$\widehat\Sigma_{\mathrm{ext}}$ is index-$k$ dominated. Since there are only
finitely many such components, the domination constants may be chosen
uniformly. By \eqref{eq:finite-extension-iterates} and the singular-value
characterization of domination given in Theorem \ref{thm:BG-domination}, the same uniform singular-value estimate
holds on the base. Hence $\hat A$ is index-$k$ dominated.

Equivalently, if $\hat A$ is not index-$k$ dominated, then at least one
transitive component of the finite extension is also not index-$k$
dominated. Thus both the periodic-gap hypothesis and the failure of
domination are preserved under the finite extensions used below.

\subsection{Stable and unstable holonomies}\label{sec:holonomies}
One of the most useful consequences of the fiber bunching assumption described in Section \ref{sec:fiber-bunching} is the existence of \emph{stable and unstable holonomies}. They provide canonical linear identifications between fibers over points lying on the same local stable or local unstable set.
More precisely, if $\hat A$ is $\alpha$-H\"older and fiber-bunched, then
for every $\hat y\in W^s_{\rm loc}(\hat x)$ the \emph{stable holonomy} from
$\hat x$ to $\hat y$ is defined by
\begin{equation*}\label{eq:canonical-holonomy}
 H^s_{\hat x,\hat y}
 =\lim_{n\to\infty}\hat A^n(\hat y)^{-1}\hat A^n(\hat x).
\end{equation*}
For $\hat y\in W^u_{\rm loc}(\hat x)$ the \emph{unstable holonomy} is
similarly given by
\begin{equation*}\label{eq:canonical-unstable-holonomy}
 H^u_{\hat x,\hat y}
 =\lim_{n\to\infty}
 \hat A^{-n}(\hat y)^{-1}\hat A^{-n}(\hat x).
\end{equation*}
These maps are H\"older continuous on
local leaves and satisfy the usual composition and equivariance properties
\begin{align*}
 H^*_{\hat x,\hat x}&=\Id,\;\;
 H^*_{\hat y,\hat z}H^*_{\hat x,\hat y}=H^*_{\hat x,\hat z}\; \text{ and }\;
 H^*_{\hat f\hat x,\hat f\hat y}\hat A(\hat x)
 =\hat A(\hat y)H^*_{\hat x,\hat y},
\end{align*}
for $*=s,u$. See, for example, \cite[Proposition~1.2]{BonattiViana2004}. From now on $H^s$ and $H^u$ always denote these canonical holonomies.

We next record the quantitative estimates that will be needed later.
After replacing $\hat f$ by an iterate if necessary, there are
$0<\eta_0\le\theta_s\le\alpha$, constants $C_H,C_B>0$, and $0<\tau<1$ such
that for local stable or unstable pairs $\hat x$ and $\hat y$
\begin{equation}\label{eq:leafwise-package}
 \norm{H^*_{\hat x,\hat y}-\Id}
 \le C_H d_\theta(\hat x,\hat y)^{\alpha}\; \text{ for }\; *=s,u,
\end{equation}
and for every local product rectangle
$(\hat x,\hat y,\hat x',\hat y')$ one has the following estimate
\begin{equation}\label{eq:rectangle-package}
 \norm{H^u_{\hat y,\hat y'}H^s_{\hat x,\hat y}
       -H^s_{\hat x',\hat y'}H^u_{\hat x,\hat x'}}
 \le C_H d_\theta(\hat x,\hat x')^{\theta_s}.
\end{equation}
Here  the term \emph{local product rectangle} means that $\hat y\in W^s_{\rm loc}(\hat x)$, $\hat x'\in W^u_{\rm loc}(\hat x)$, and $\hat y'\in W^u_{\rm loc}(\hat y)\cap W^s_{\rm loc}(\hat x')$.
Moreover, in dimension $d\ge3$,
\begin{equation}\label{eq:eta-bunching-consequence}
 \bol(\hat A^n(\hat z))\,\theta^{\eta_0 n}
 \le C_B\tau^n \; \text{ for }\; n\ge0,\ \hat z\in\widehat\Sigma.
\end{equation}
Since $\eta_0\le\theta_s$, the same estimate holds with $\theta_s$ in place
of $\eta_0$. The leafwise regularity \eqref{eq:leafwise-package} follows already from fiber bunching and can be found, for instance, in  \cite[Proposition~1.2]{BonattiViana2004}. Transverse regularity \eqref{eq:eta-bunching-consequence} also follows from the fiber bunching assumption for some $\theta_s\leq \alpha$ as proved in \cite[Proposition 2.10]{BochiGaribaldi}. Strong bunching is used to obtain an exponent
$\theta_s\ge\eta_0$ for which the transverse estimate
\eqref{eq:rectangle-package} holds and which is therefore compatible with the
exponential estimate \eqref{eq:eta-bunching-consequence} (see
\cite[Lemma~3.6]{BochiGaribaldi}).

\subsection{Stable-holonomy reduction to the one-sided shift}
\label{subsec:stable-reduction}

In this section we describe a standard construction that reduces
$\hat A$ to a cocycle over the one-sided shift. We also explain how
invariant measures and $u$-states behave under this reduction. The
construction uses stable holonomies and goes back to Bonatti--Viana
\cite{BonattiViana2004}. See also
\cite{AvilaViana2007,BackesBrownButler,DeWittMitsutani} for related
formulations.

\subsubsection{Reduction of the cocycle}

For each symbol $a\in\{1,\ldots,\ell\}$, fix a left-infinite admissible
sequence $ x^-(a)\in\Sigma^s$ (called \emph{reference past} of $a$) whose zero-coordinate is $a$. If
$\hat x\in\widehat\Sigma$ has zero-coordinate $x_0=a$, let
\[
 \varphi^u(\hat x)
 \in W^u_{\rm loc}(x^-(a))\cap W^s_{\rm loc}(\hat x)
\]
be the two-sided sequence having past $x^-(a)$ and the same future as
$\hat x$. Thus
\[
 P\varphi^u(\hat x)=P\hat x,
\]
and $\varphi^u$ is constant on each local stable set.

Using the stable holonomies, define
\begin{equation}\label{eq:stable-reduced-cocycle}
 A(P\hat x)
 :=
 H^s_{\hat f\hat x,\,\varphi^u(\hat f\hat x)}
 \,\hat A(\hat x)\,
 H^s_{\varphi^u(\hat x),\,\hat x}.
\end{equation}
The composition and equivariance properties of the stable holonomies imply
that the right-hand side depends only on $P\hat x$. Hence
\eqref{eq:stable-reduced-cocycle} defines a H\"older cocycle
\[
 A:\Sigma\longrightarrow\GL(d,\R)
\]
over the one-sided shift $f:\Sigma\to\Sigma$. In other words, stable
holonomy identifies the fiber over $\hat x$ with the fiber over the preferred
representative $\varphi^u(\hat x)$, and with respect to these
identifications the cocycle is constant on local stable sets. 

It is useful to record the corresponding factor map on the Grassmannian
bundle. Define
\[
 \mathcal Q_q:
 \widehat\Sigma\times\Gr(q,d)
 \longrightarrow
 \Sigma\times\Gr(q,d)
\]
by
\begin{equation}\label{eq:stable-factor-map-prelim}
 \mathcal Q_q(\hat x,V)
 :=
 \bigl(
 P\hat x,\,
 H^s_{\hat x,\,\varphi^u(\hat x)}V
 \bigr).
\end{equation}
If
\[
 \hat F(\hat x,V) = \bigl(\hat f\hat x,\hat A(\hat x)V\bigr) \; \text{ and } \;  F(x,V) =  \bigl(fx,A(x)V\bigr),
\]
then \eqref{eq:stable-reduced-cocycle} implies
\[
 \mathcal Q_q\circ\hat F = F\circ\mathcal Q_q.
\]
Thus $\mathcal Q_q$ intertwines the two-sided and one-sided Grassmannian
dynamics.

\subsubsection{Reduction of measures}

We next recall the measure-theoretic setting in which the reduction will be
used. Let $\hat\mu$ be a fully supported
$\hat f$-invariant probability measure on $\widehat\Sigma$. Set $\mu^s:=(P^s)_*\hat\mu$ and $ \mu:=P_*\hat\mu$.
For each symbol $a$, the map
\[
 [0;a]
 \longrightarrow
 P^s([0;a])\times P([0;a]),
 \qquad
 \hat x\longmapsto
 \bigl(P^s\hat x,P\hat x\bigr),
\]
is a homeomorphism. We say that $\hat\mu$ has
\emph{continuous local product structure} if there exists a positive
continuous function $ \psi:\widehat\Sigma\longrightarrow(0,\infty)$
such that, under this identification,
\begin{equation*}\label{eq:local-product-structure}
 \hat\mu|_{[0;a]}= \psi\, \bigl( \mu^s|_{P^s([0;a])} \times \mu|_{P([0;a])} \bigr)
\end{equation*}
for every symbol $a$. This is the usual local product structure for measures
on a two-sided subshift as used, for instance, in
\cite[Section~2.3]{BackesBrownButler} and
\cite[Section~2]{DeWittMitsutani}.

Disintegrate $\hat\mu$ with respect to the projection
$P:\widehat\Sigma\to\Sigma$
\begin{equation*}\label{eq:stable-disintegration-base}
 \hat\mu
 =
 \int_\Sigma
 \hat\mu_x^s\,d\mu(x) \; \text{ with }\;
 \supp\hat\mu_x^s\subset P^{-1}(x).
\end{equation*}
Since the fibers of $P$ are precisely the local stable sets,
$\hat\mu_x^s$ is the conditional measure of $\hat\mu$ on the stable
fiber over $x$.

Let $\hat m$ be an $\hat F$-invariant probability measure on
$\widehat\Sigma\times\Gr(q,d)$ which projects to $\hat\mu$. Thus
\[
 \hat\pi_*\hat m=\hat\mu,
\]
where $ \hat\pi:
 \widehat\Sigma\times\Gr(q,d)\longrightarrow\widehat\Sigma $
is the canonical projection. Write the fiber disintegration of
$\hat m$ as
\begin{equation}\label{eq:disintegrationhatm}
 \hat m
 =
 \int_{\widehat\Sigma}
 \hat m_{\hat x}\,d\hat\mu(\hat x),
\end{equation}
where $\hat m_{\hat x}$ is, for $\hat\mu$-almost every $\hat x$,
a probability measure on the fiber
$\{\hat x\}\times\Gr(q,d)$, which we identify with $\Gr(q,d)$.

We define the reduced measure by
\[
 m:=(\mathcal Q_q)_*\hat m.
\]
Since $\mathcal Q_q$ intertwines $\hat F$ and $F$, the measure $m$ is
$F$-invariant and projects to $\mu$. Its fiber disintegration is given, for
$\mu$-almost every $x$, by
\begin{equation*}\label{eq:one-sided-conditional-from-two-sided}
 m_x = \int_{P^{-1}(x)}  \bigl( H^s_{\hat x,\,\varphi^u(\hat x)} \bigr)_*
 \hat m_{\hat x}\, d\hat\mu_x^s(\hat x).
\end{equation*}
Thus the one-sided conditional $m_x$ is obtained by transporting the
two-sided fiber conditionals to the preferred stable representative and then
averaging over the stable fiber. Compare with
\cite[Lemma~4.5]{AvilaViana2007},
\cite[Sections~4.2--4.3]{BackesBrownButler}, and
\cite[Lemma~2.14]{DeWittMitsutani}. Note that in these works, $\hat A$ is assumed to be constant along local stable sets after a coordinate change induced by stable holonomies. This explains why stable holonomies do not explicitly appear in their expression for $m_x$.

\subsubsection{$u$-states} An
$\hat F$-invariant probability measure $\hat m$ on
$\widehat\Sigma\times\Gr(q,d)$ which projects to $\hat\mu$ is called a
\emph{$u$-state} if its fiber disintegration
\eqref{eq:disintegrationhatm} can be chosen to be invariant under local
unstable holonomies, that is,
\begin{equation*}\label{eq:ustate-two-sided}
 (H^u_{\hat x,\hat y})_*\hat m_{\hat x}
 =
 \hat m_{\hat y}
\end{equation*}
for almost every pair of points $\hat x,\hat y$ lying on the same local
unstable set.

The same terminology will be used bellow on partial flag bundles: the
Grassmannian fiber $\Gr(q,d)$ is simply replaced by the corresponding
partial flag manifold, and the holonomies act simultaneously on each
subspace of the flag.

One of the main properties of $u$-states used below is that the one-sided
fiber disintegration obtained after stable reduction admits a weak-star
continuous version. More precisely, if $\hat m$ is a $u$-state and
$m=(\mathcal Q_q)_*\hat m$, then the family
\[
 x\longmapsto m_x
\]
may be chosen to depend weak-star continuously on $x\in\Sigma$ as observed in
\cite[Proposition~4.4]{AvilaViana2007} and
\cite[Proposition~2.15]{DeWittMitsutani}.

We finally record the inverse-branch relation satisfied by this continuous
disintegration. Continuous local product structure implies that the
one-sided measure $\mu=P_*\hat\mu$ admits a positive continuous Jacobian
$J_\mu f$ with respect to $f$ (see
\cite[Section~2.3]{BackesBrownButler}). We write $J_\mu f^n$ for the
corresponding Jacobian of $f^n$. Then, for every $x\in\Sigma$ and every
$n\ge1$,
\begin{equation}\label{eq:one-sided-transfer-prelim}
 m_x
 =
 \sum_{z\in f^{-n}(x)}
 \frac{1}{J_\mu f^n(z)}
 (A^n(z))_*m_z.
\end{equation}
In particular, the coefficients
\[
 p_n(z\mid x)
 :=
 \frac{1}{J_\mu f^n(z)}
\]
are positive and satisfy
\[
 \sum_{z\in f^{-n}(x)}p_n(z\mid x)=1.
\]
Formula \eqref{eq:one-sided-transfer-prelim} is the inverse-branch formula
for the continuous disintegration of a $u$-state (see
\cite[Proposition~2.16]{DeWittMitsutani}). When $\mu$ is a fully supported
one-step Markov measure, these inverse-branch weights are locally constant
on the corresponding inverse-branch cylinders.

\section{Preliminaries II: Holonomy paths and periodicization}\label{sec:paths}

The aim of this section is to explain how holonomies and orbit pieces can be
joined into paths between fibers. We also show how a path that almost closes
can be changed into a true periodic orbit while changing the corresponding
linear map only by a small amount. This is the step that will allow us to use
the periodic Lyapunov-gap assumption later in the proof.

A \emph{holonomy path} from $\hat x$ to $\hat y$ is specified by a point
$\hat x'\in W^u_{\rm loc}(\hat x)$, an integer $n\ge0$, and
$\hat y'=\hat f^n\hat x'\in W^s_{\rm loc}(\hat y)$.  Its linear map is
\begin{equation*}\label{eq:path-map}
 {\mathcal B}_{\hat x,\hat y}
 =H^s_{\hat y',\hat y}\hat A^n(\hat x')H^u_{\hat x,\hat x'}:
 E_{\hat x}\longrightarrow E_{\hat y}.
\end{equation*}
Thus, along a holonomy path, we move from one point to another along a stable
or unstable leaf, then apply the cocycle for some number of iterates, and then
move again along a stable or unstable leaf. A finite sequence of such
operations gives a linear map between the corresponding fibers. When the path
starts and ends at the same point $\hat p$, we call it a \emph{based holonomy loop at $\hat p$} or a \emph{based path loop at $\hat{p}$}. 

We shall use two elementary path operations. For this, recall our bracket convention \eqref{eq:bracket-conv} and that hatted symbols denote points in $\widehat \Sigma$.

\begin{lemma}
\label{lem:path-perturb}
Take $\hat c\in W^u_{\rm loc}(\hat y)$ and $\hat b=\hat f^L\hat c\in  W^s_{\rm loc}(\hat z)$ and let
\[
 {\mathcal G}=H^s_{\hat b,\hat z}\hat A^L(\hat c)
 H^u_{\hat y,\hat c}:E_{\hat y}\to E_{\hat z}
\]
be a fixed holonomy path.  If $\hat y_j\to\hat y$, then, for all large $j$,
there are holonomy paths ${\mathcal G}_j:E_{\hat y_j}\to E_{\hat z}$ with the
same orbit length $L$ and ${\mathcal G}_j\to{\mathcal G}$.
\end{lemma}

\begin{proof}
Put $\hat c_j=[\hat y_j,\hat c]$.  Then
$\hat c_j\in W^u_{\rm loc}(\hat y_j)\cap W^s_{\rm loc}(\hat c)$ and
$\hat c_j\to\hat c$.  Since $L$ is fixed,
$\hat b_j:=\hat f^L\hat c_j\to\hat b=\hat f^L\hat c$ and, for all
large $j$, $\hat b_j$ lies in the same local stable set of $\hat z$ as
$\hat b$.  Hence
\[
 {\mathcal G}_j
 =H^s_{\hat b_j,\hat z}\hat A^L(\hat c_j)H^u_{\hat y_j,\hat c_j}
\]
is a holonomy path.  H\"older continuity of $\hat A$ and continuity of the
holonomies give ${\mathcal G}_j\to{\mathcal G}$.
\end{proof}

\begin{lemma}
\label{lem:exact-concat}
Let $\hat a\in W^u_{\rm loc}(\hat x)$ and $\hat y=\hat f^N\hat a$ and consider
\[
 \mathcal D=\hat A^N(\hat a)H^u_{\hat x,\hat a}:
 E_{\hat x}\to E_{\hat y},
\]
so that the terminal stable leg of $\mathcal D$ is the identity.  If
${\mathcal G}:E_{\hat y}\to E_{\hat z}$ is any holonomy path, then the
composition ${\mathcal G}{\mathcal D}$ is itself a holonomy path and its path
map is exactly the matrix product ${\mathcal G}{\mathcal D}$.
\end{lemma}

\begin{proof}
Let us consider
\[
 {\mathcal G}=H^s_{\hat b,\hat z}\hat A^L(\hat c)H^u_{\hat y,\hat c}
\]
with $\hat c\in W^u_{\rm loc}(\hat y)$ and $\hat b=\hat f^L\hat c\in  W^s_{\rm loc}(\hat z)$.
Put $\hat a'=\hat f^{-N}\hat c$. Since $\hat c\in W^u_{\rm loc}(\hat y)$ and
$\hat y=\hat f^N\hat a$, we get that
\[
 \hat a'
 =\hat f^{-N}\hat c
 \in W^u_{\rm loc}(\hat f^{-N}\hat y)
 =W^u_{\rm loc}(\hat a).
\]
Thus, the equivariance of unstable holonomies gives
\[
 H^u_{\hat y,\hat c}\hat A^N(\hat a)
 =\hat A^N(\hat a')H^u_{\hat a,\hat a'}.
\]
Moreover, since $\hat x,\hat a,\hat a'$ lie on the same unstable set, the composition
law gives
$H^u_{\hat a,\hat a'}H^u_{\hat x,\hat a}=H^u_{\hat x,\hat a'}$. Combining these observations we get that
\[
 {\mathcal G}{\mathcal D}
 =H^s_{\hat b,\hat z}\hat A^{N+L}(\hat a')H^u_{\hat x,\hat a'},
\]
which is a holonomy path of the required form.
\end{proof}

\begin{lemma}\label{lem:periodize}
Assume that $\hat p$ is a fixed point of $\hat f$ and take  $\hat a_j\in W^u_{\rm loc}(\hat p)$ and $\hat b_j=\hat f^{M_j}\hat a_j\in W^s_{\rm loc}(\hat p)$. Write $\hat a_j =(a_{j,t})_{t\in \mathbb Z}$ and let
\[
 \mathcal B_j
 =H^s_{\hat b_j,\hat p}\hat A^{M_j}(\hat a_j)
 H^u_{\hat p,\hat a_j}:E_{\hat p}\to E_{\hat p},\\
 \]
be based holonomy loops with $M_j\to\infty$ and $\hat a_j\to\hat p$.  Let
$\hat q_j$ be the periodic point obtained by repeating the admissible starting word of $\hat a_j$ of length $M_j$, that is, $\hat q_j$ is obtained repeating the word $(a_{j,0},a_{j,1},\ldots,a_{j,M_j-1})$.  Then there are invertible maps $H_{1,j},H_{2,j}$ such
that
\begin{equation}\label{eq:periodize-factor}
 \hat A^{M_j}(\hat q_j)=H_{1,j}{\mathcal B}_jH_{2,j} \; \text{ with }\; R_j:=H_{2,j}H_{1,j}\longrightarrow\Id.
\end{equation}
The maps $H_{1,j}^{\pm1},H_{2,j}^{\pm1}$ are uniformly bounded.  In
particular $\hat A^{M_j}(\hat q_j)$ has the same eigenvalues, with
algebraic multiplicity, as $R_j{\mathcal B}_j$.
\end{lemma}

\begin{proof}
We start observing that, since $\hat a_j\in W^u_{\rm loc}(\hat p)$ and
$\hat f^{M_j}\hat a_j\in W^s_{\rm loc}(\hat p)$, the transition from the
last symbol $a_{j,M_j-1}$ of the word $(a_{j,0},a_{j,1},\ldots,a_{j,M_j-1})$ to its first symbol $a_{j,0}$ is admissible and  $\hat q_j$ is well-defined. 

Now, for $0\le i\le M_j$ put
$\hat a_{i,j}=\hat f^i\hat a_j$ and
$\hat q_{i,j}=\hat f^i\hat q_j$. Let $\hat z_{i,j}=[\hat q_{i,j},\hat a_{i,j}]$ and consider $ K_{i,j}:E_{\hat q_{i,j}}\longrightarrow E_{\hat a_{i,j}}
$ given by
\[
 K_{i,j}=H^s_{\hat z_{i,j},\hat a_{i,j}}
 H^u_{\hat q_{i,j},\hat z_{i,j}}.
\]
Note that the bracket is defined because the two points have the same zero-symbol.
Thus, holonomy equivariance gives us that
\[
 K_{i+1,j}\hat A(\hat q_{i,j})
 =\hat A(\hat a_{i,j})K_{i,j},
\]
and therefore, applying this identity inductively, we get that
\begin{equation*}\label{eq:transition-product}
 \hat A^{M_j}(\hat q_j)
 =K_{M_j,j}^{-1}\hat A^{M_j}(\hat a_j)K_{0,j}.
\end{equation*}
Consequently, since
\[
 \hat A^{M_j}(\hat a_j)
 =(H^s_{\hat b_j,\hat p})^{-1}{\mathcal B}_j
 (H^u_{\hat p,\hat a_j})^{-1},
\]
we may take
\[
 H_{1,j}=K_{M_j,j}^{-1}(H^s_{\hat b_j,\hat p})^{-1}
 \; \text{ and }\;
 H_{2,j}=(H^u_{\hat p,\hat a_j})^{-1}K_{0,j}.
\]
Note that all these maps and their inverses are uniformly bounded and, moreover, satisfy the first equality in \eqref{eq:periodize-factor}.

It remains to prove that $H_{2,j}H_{1,j}\to\Id$. We start observing that, since $\hat q_j$ is obtained by repeating the word $(a_{j,0},a_{j,1},\ldots,a_{j,M_j-1})$, the points $\hat q_j$ and
$\hat b_j=\hat f^{M_j}\hat a_j$ agree on the $M_j$ symbols immediately to
the left of the zero coordinate. Their nonnegative coordinates agree for as
many symbols as $\hat a_j$ and the fixed point $\hat p$ agree. Since
$M_j\to\infty$ and $\hat a_j\to\hat p$, we obtain
\begin{equation}\label{eq:q-b-close}
 d_\theta(\hat q_j,\hat b_j)\longrightarrow0.
\end{equation}
At time $M_j$ the transition map $K_{M_j,j}$ is therefore a product of two
holonomies whose endpoints become arbitrarily close, and hence
$K_{M_j,j}\to\Id$.

Now, observe that the periodic point $\hat q_j$ and $\hat a_j$ agree at the coordinates $0,\ldots,M_j-1$. Thus, since
$\hat z_{0,j}=[\hat q_j,\hat a_j]$ has the future of $\hat a_j$ and the past of $\hat q_j$, this gives
\[
 d_\theta(\hat q_j,\hat z_{0,j})\le \theta^{M_j}\longrightarrow0.
\]
Together with \eqref{eq:q-b-close} this also gives
$d_\theta(\hat z_{0,j},\hat b_j)\to0$. Therefore, since $\hat a_j\to\hat p$, uniform
continuity of the holonomies on the compact local stable and unstable sets yields
\[
 H^u_{\hat q_j,\hat z_{0,j}}\longrightarrow\Id \; \text{ and } \;
 \|H^s_{\hat z_{0,j},\hat a_j}
          -H^s_{\hat b_j,\hat p}\|\longrightarrow0.
\]
Consequently,
\begin{equation*}\label{eq:K0-limit}
 \bigl\|K_{0,j}-H^s_{\hat b_j,\hat p}\bigr\|\longrightarrow0.
\end{equation*}
Finally, $H^u_{\hat p,\hat a_j}\to\Id$ because $\hat a_j\to\hat p$ on the
local unstable set. Substituting these three estimates in
\[
 H_{2,j}H_{1,j}
 =(H^u_{\hat p,\hat a_j})^{-1}K_{0,j}K_{M_j,j}^{-1}
 (H^s_{\hat b_j,\hat p})^{-1}
\]
gives $R_j\to\Id$.

The last assertion follows directly by cyclic similarity:
\[
 H_{1,j}{\mathcal B}_jH_{2,j} \sim {\mathcal B}_jH_{2,j}H_{1,j} \sim R_j{\mathcal B}_j.
\]
\end{proof}

\section{A linear-algebraic reduction}\label{sec:reduction-main}
In this section we perform a linear-algebraic reduction of our cocycle. The
purpose is to eliminate the upper-triangular shearing and the finite
permutation of invariant components, while preserving the periodic spectral
data, the failure of domination, and the quantitative estimates required in
the rest of the proof. The construction and its properties are proved in
Appendix~\ref{app:reduction}. We therefore use the reduction here without
interrupting the main dynamical argument with its algebraic details.

Let $\hat p=(p_t)_{t\in \mathbb Z}$ be a periodic point of $\hat f$, which, after passing to a suitable iterate, we assume to be a fixed point.

\begin{proposition}\label{prop:linear-reduction}
If Theorem~\ref{thm:main} fails, then after a bounded H\"older change of
coordinates, passage to a finite permutation extension and a transitive
component, passage to an iterate, and removal of purely upper-triangular
shearing, one obtains another counterexample for which:
\begin{enumerate}[label=\textup{(\roman*)}]
\item the periodic-gap hypothesis is preserved. Moreover, the algebraic
changes preserve the eigenvalues of corresponding periodic return matrices;
\item failure of index-$k$ domination is preserved;
\item the estimates
\eqref{eq:leafwise-package}--\eqref{eq:rectangle-package} and, when $d\ge3$,
\eqref{eq:eta-bunching-consequence}, remain valid with the same exponents;
\item if a finite family of subspaces in an exterior power is permuted by all
based holonomy loops at $\hat p$, then every based holonomy loop fixes each member
of the family individually;
\item if a subspace $W$ of an exterior power is invariant under every based
holonomy loop at $\hat p$, then it has a complementary subspace $W'$ which is also
invariant under every based holonomy loop.
\end{enumerate}
\end{proposition}

We apply Proposition~\ref{prop:linear-reduction} once and then rename the
resulting cocycle and base as $\hat A$ and $\hat f$.  All bad orbit segments below are chosen after this reduction.  This order is important: the final contradiction must use the same bad sequence that produces both the limiting singular covector and the invariant component.

Thus, from now on, let $\hat{p}$ be fixed and assume that the conclusions (i)-(v) of Proposition \ref{prop:linear-reduction} hold for $(\hat A,\hat f)$.

\section{Beginning the proof of Theorem \ref{thm:main}: failure of domination and a singular-value plateau}\label{sec:bad-blocks}

In this section we begin the proof of Theorem \ref{thm:main}. We assume, for a contradiction, that the cocycle is not
index-$k$ dominated. This gives long orbit segments on which the $k$-th and
$(k+1)$-st singular values are too close on the exponential scale. We place
these bad segments inside paths based at the fixed periodic point $\hat p$ and extract a limiting block of singular values with a flat part containing the index $k$.  All the notation introduced here will be fixed for the remainder of the argument

\subsection{Bad orbits and paths}
Assume, toward a contradiction, that $\hat A$ is not index-$k$ dominated.  By the
Bochi--Gourmelon singular-value characterization given in Theorem \ref{thm:BG-domination}, there
are $\hat x_j=(x_{j,t})_{t\in\mathbb Z}\in\widehat\Sigma$ and $n_j\to\infty$ such that
\begin{equation}\label{eq:bad}
 \log\frac{\sigma_k(\hat A^{n_j}(\hat x_j))}
               {\sigma_{k+1}(\hat A^{n_j}(\hat x_j))}
 =o(n_j)
\end{equation}
where $o(n_j)$ means the usual little-o notation. We now construct paths based at $\hat p=(p_t)_{t\in\mathbb Z}$ associated with the bad orbit blocks $(x_{j,0}x_{j,1}\cdots x_{j,n_j})$ coming from $\hat x_j$.  Choose integers
$m_j\to\infty$ with $m_j=o(n_j)$, for example
$m_j=\lfloor\sqrt{n_j}\rfloor$.  Fix also a point
$\hat y=(y_t)_{t\in\mathbb Z}\in\widehat\Sigma$. Its role is only to provide a common limiting endpoint for the extended bad blocks that we build below. Since, after the preliminary reductions (recall Remark \ref{rem:mixing-reduction}), we
may work on a mixing component, there is a uniform constant $L$ such that any
two admissible endpoint symbols can be joined by a bridge containing at most
$L$ inserted symbols.  Passing to a subsequence, if desired, we may also
assume that the endpoint symbols $x_{j,0}$ and $x_{j,n_j}$ are fixed.

In order to build paths based at $\hat p=(p_t)_{t\in\mathbb Z}$, we construct the points $\hat a_j$ that realize the bad orbit blocks. So, choose an admissible bridge
$b_j^-$ from the symbol $p_{m_j}$ to $x_{j,0}$ and an admissible bridge
$b_j^+$ from $x_{j,n_j}$ to $y_{-m_j}$.  Let $\ell_j^\pm\le L$ denote the
numbers of \emph{inserted} symbols in these two bridges.  Define the
bi-infinite sequence $\hat a_j=(a_{j,t})_{t\in\mathbb Z}$ by the following
concatenation:
\[
 \underbrace{\cdots p_{-2}p_{-1}p_0p_1\cdots p_{m_j}}_{\text{past of $\hat p$ and a long future block of $\hat p$}}
 \; b_j^- \;
 \underbrace{x_{j,0}x_{j,1}\cdots x_{j,n_j}}_{\text{copied bad word}}
 \; b_j^+ \;
 \underbrace{y_{-m_j}\cdots y_{-1}y_0y_1\cdots}_{\text{tail of $\hat y$}}.
\]
Equivalently, $a_{j,t}=p_t$ for every $t\le m_j$, the copied bad word starts
at the time
\begin{equation*}\label{eq:tj-definition}
 t_j:=m_j+\ell_j^-+1,
\end{equation*}
and
\begin{equation}\label{eq:copied-word}
 a_{j,t_j+r}=x_{j,r},\qquad 0\le r\le n_j.
\end{equation}
After the second bridge we place the symbol $y_{-m_j}$ and then continue with
the future of $\hat y$.  If $N_j$ denotes the time occupied by the copied
symbol $y_0$, then, with the above convention for the bridge lengths,
\begin{equation*}\label{eq:Nj-definition}
 N_j=t_j+n_j+\ell_j^++1+m_j
     =n_j+2m_j+\ell_j^-+\ell_j^++2.
\end{equation*}
Finally set
\[
 \hat y_j:=\hat f^{N_j}\hat a_j.
\]
These definitions immediately give
\[
 N_j=n_j+o(n_j),\quad
 \hat a_j\in W^u_{\rm loc}(\hat p),\quad \hat a_j\to\hat p \; \text{ and }
 \; \hat y_j\to\hat y.
\]
Indeed, $\hat a_j$ agrees with $\hat p$ at every coordinate $t\le m_j$,
while $\hat y_j$ agrees with $\hat y$ at every coordinate $t\ge -m_j$.

The point at the beginning of the copied bad word is now completely explicit:
\begin{equation}\label{eq:xjprime-definition}
 \hat x'_j:=\hat f^{t_j}\hat a_j.
\end{equation}
By \eqref{eq:copied-word},
\begin{equation}\label{eq:xjprime-word-agreement}
 x'_{j,r}=x_{j,r},\qquad 0\le r\le n_j.
\end{equation}
Thus $\hat x'_j$ and $\hat x_j$ have the same $(n_j+1)$-``starting'' word.  Notice that
copying the endpoint symbol $x_{j,n_j}$ is useful: it guarantees the required
local unstable relation also at the terminal end of the $n_j$-step orbit
segment.

Using the initial unstable holonomy and taking
$\hat y_j=\hat f^{N_j}\hat a_j$ as the actual target (so that no terminal
stable holonomy is needed), define
\begin{equation*}\label{eq:Dj-explicit}
 D_j
 :=\hat A^{N_j}(\hat a_j)H^u_{\hat p,\hat a_j}
 :E_{\hat p}\longrightarrow E_{\hat y_j}.
\end{equation*}
This is the path map of the framed orbit segment of length $N_j$ and will play a fundamental role in our proof.

We next compare $D_j$ explicitly with the original bad matrix
$\hat A^{n_j}(\hat x_j)$.  First we isolate the copied block.  By the cocycle
identity and \eqref{eq:xjprime-definition},
\begin{equation}\label{eq:Dj-prefix-central-suffix}
 D_j=P_j^+\,\hat A^{n_j}(\hat x'_j)\,P_j^-,
\end{equation}
where
\begin{align}
 P_j^-
 &:=\hat A^{t_j}(\hat a_j)H^u_{\hat p,\hat a_j}
   :E_{\hat p}\to E_{\hat x'_j}\label{eq:Pminus}
\end{align}
and
 \begin{align}
P_j^+
 &:=\hat A^{N_j-t_j-n_j}
       (\hat f^{t_j+n_j}\hat a_j)
   :E_{\hat f^{n_j}\hat x'_j}\to E_{\hat y_j}.
   \label{eq:Pplus}
\end{align}
The two outside orbit pieces have lengths
\[
 t_j=m_j+O(1)=o(n_j)\; \text{ and }\;
 N_j-t_j-n_j=m_j+O(1)=o(n_j).
\]

Observe that, since $\hat A$ is H\"older continuous rather than locally constant, the equality of the symbolic words in \eqref{eq:xjprime-word-agreement} does not imply that
$\hat A^{n_j}(\hat x'_j)=\hat A^{n_j}(\hat x_j)$. We use the canonical
holonomies to compare these matrices.  Put
\[
 \hat z_j=[\hat x'_j,\hat x_j].
\]
Then $\hat z_j\in W^u_{\rm loc}(\hat x'_j)\cap W^s_{\rm loc}(\hat x_j)$,
and, because the two points agree also at the endpoint $n_j$,
$\hat f^{n_j}\hat z_j$ is simultaneously unstable-related to
$\hat f^{n_j}\hat x'_j$ and stable-related to
$\hat f^{n_j}\hat x_j$.  Iterating holonomy equivariance gives the exact
factorization
\begin{equation}\label{eq:word-distortion}
 \hat A^{n_j}(\hat x'_j)
 =L_j^{(0)}\hat A^{n_j}(\hat x_j)R_j^{(0)},
\end{equation}
with
\begin{align*}
 R_j^{(0)}
 &=H^s_{\hat z_j,\hat x_j}H^u_{\hat x'_j,\hat z_j}\label{eq:R0-explicit}
\end{align*}
and
 \begin{align*}
 L_j^{(0)}
 &=\bigl(H^u_{\hat f^{n_j}\hat x'_j,
                    \hat f^{n_j}\hat z_j}\bigr)^{-1}
   \bigl(H^s_{\hat f^{n_j}\hat z_j,
                    \hat f^{n_j}\hat x_j}\bigr)^{-1}.
\end{align*}
Note that all four holonomies above, and their inverses, are uniformly bounded.

Combining \eqref{eq:Dj-prefix-central-suffix} and
\eqref{eq:word-distortion}, we obtain the direct comparison
\begin{equation}\label{eq:Dj-direct-comparison}
 D_j=C_j^+\,\hat A^{n_j}(\hat x_j)\,C_j^-
\end{equation}
with $ C_j^+:=P_j^+L_j^{(0)}$ and $C_j^-:=R_j^{(0)}P_j^-$. 
Let
\[
 K_A:=\sup_{\hat x\in\widehat\Sigma}\bol(\hat A(\hat x))<\infty.
\]
Thus, since the lengths of the prefix and suffix are $o(n_j)$ and the canonical
holonomies are uniformly bounded, \eqref{eq:Pminus} and \eqref{eq:Pplus} imply
\begin{equation}\label{eq:Cpm-subexp}
 \log\bol(C_j^+)+\log\bol(C_j^-)=o(n_j).
\end{equation}
Now, for any $L,M,R\in\GL(d,\R)$, standard singular-value inequalities give
\begin{equation}\label{eq:gap-comparison-LMR}
 \left|
 \log\frac{\sigma_k(LMR)}{\sigma_{k+1}(LMR)}
 -\log\frac{\sigma_k(M)}{\sigma_{k+1}(M)}
 \right|
 \le \log\bol(L)+\log\bol(R).
\end{equation}
Applying \eqref{eq:gap-comparison-LMR} to
\eqref{eq:Dj-direct-comparison}, and using \eqref{eq:bad} and
\eqref{eq:Cpm-subexp}, yields
\[
 \log\frac{\sigma_k(D_j)}{\sigma_{k+1}(D_j)}
 =o(n_j).
\]
Thus, since $N_j/n_j\to1$, this is equivalent to
\begin{equation}\label{eq:badD}
 \log\frac{\sigma_k(D_j)}{\sigma_{k+1}(D_j)}=o(N_j).
\end{equation}
This is the precise sense in which $(D_j,N_j)$ is again a bad sequence: its
normalized $k$-th singular-value gap tends to zero, exactly as for the
original sequence
$(\hat A^{n_j}(\hat x_j),n_j)$.

The endpoint estimates needed later are $\hat a_j\to\hat p$ (for
Lemma~\ref{lem:periodize}) and $\hat y_j\to\hat y$ (for perturbing a path beginning at
$\hat y$  -- see Lemma \ref{lem:path-perturb}).

\subsection{Fast and slow spaces}
After passing to a subsequence, we may assume that the limits
\[
 a_i=\lim_{j\to\infty}\frac1{N_j}\log\sigma_i(D_j),
 \]
exist. In particular, we have $a_1\ge\cdots\ge a_d$. Moreover, by \eqref{eq:badD} we have that $a_k=a_{k+1}$.
Let
\begin{equation}\label{eq:plateau}
 a_r>a_{r+1}=\cdots=a_s>a_{s+1}\; \text{ with }\; r<k<s,
\end{equation}
be the \emph{maximal plateau} containing $k$ and $k+1$, with the endpoint inequality
omitted when $r=0$ or $s=d$. In what follows, we use the term \emph{genuine boundary} for an index $q\in\{1,\dots,d-1\}$ at an actual edge of the maximal singular-value plateau, so that $a_q>a_{q+1}$.

As observed above, at every genuine boundary $q\in\{r,s\}$ we have
$a_q>a_{q+1}$.  Hence
\[
 \log\frac{\sigma_q(D_j)}{\sigma_{q+1}(D_j)}
 =(a_q-a_{q+1})N_j+o(N_j),
\]
so for all sufficiently large $j$ there is a strict singular-value gap at
index $q$.  To explain the associated subspaces, choose a singular-value
decomposition 
\[
 D_j=U_j\Sigma_jV_j^*
\]
for $ D_j:E_{\hat p}\longrightarrow E_{\hat y_j}$
or, equivalently, orthonormal right singular vectors
$v_{1,j},\ldots,v_{d,j}\in E_{\hat p}$ and orthonormal left singular vectors
$u_{1,j},\ldots,u_{d,j}\in E_{\hat y_j}$ such that
\[
 D_jv_{i,j}=\sigma_i(D_j)u_{i,j}
 \;\text{ with }\;
 \sigma_1(D_j)\ge\cdots\ge\sigma_d(D_j)>0.
\]
The \emph{left fast $q$-plane} is
\begin{equation*}\label{eq:left-fast-plane}
 F_{q,j}:=\Span\{u_{1,j},\ldots,u_{q,j}\}
 \subset E_{\hat y_j}.
\end{equation*}
It is called ``left'' because it is defined by the left singular vectors,
which live in the target fiber, and ``fast'' because it corresponds to the
$q$ largest singular values.  The \emph{right slow $(d-q)$-plane} is
\begin{equation*}\label{eq:right-slow-plane}
 S_{q,j}:=\Span\{v_{q+1,j},\ldots,v_{d,j}\}
 \subset E_{\hat p}.
\end{equation*}
Similarly, it is called ``right'' because it is defined by the right singular vectors,
which live in the source fiber, and ``slow'' because it corresponds to the
$d-q$ smallest singular values.  Equivalently,
$S_{q,j}^{\perp}=\Span\{v_{1,j},\ldots,v_{q,j}\}$ is the right fast
$q$-plane and
\[
 D_j(S_{q,j}^{\perp})=F_{q,j}.
\]
The strict gap at $q$ makes these two subspaces independent of the choice of
singular bases inside the upper and lower blocks.  Since the relevant
Grassmannians are compact and $\hat y_j\to\hat y$, after passing to a
subsequence we may assume
\[
 F_{q,j}\to F_q \subset E_{\hat y}\;\text{ and }\; S_{q,j}\to S_q\subset E_{\hat p}.
\]
At a formal endpoint we set
\[
 F_0:=\{0\}\; \text{ and }\; F_d:=E_{\hat y},
\]
whenever that endpoint occurs.  Thus $F_r\subset F_s$ in all cases.  For the
slow spaces we use only genuine boundaries: when both are genuine one has
$S_s\subset S_r$.  Put
\begin{equation}\label{eq:J-boundaries}
 J:=\{q\in\{r,s\}:1\le q\le d-1\}.
\end{equation}

\section{Closing the bad blocks and the all-path obstruction}\label{sec:path-obstruction}

The goal of this section is to use the singular-value plateau from the previous section to obtain a
condition that every holonomy path closing the bad blocks must satisfy. If a
closing path failed to satisfy all the boundary relations coming from the
plateau, the comparison between the singular-value sums and the logarithms of the eigenvalue moduli would give periodic orbits whose
$k$-th Lyapunov gap is too small. This would contradict the assumed uniform
gap on periodic orbits. Therefore every closing path must satisfy at least one
of the limiting boundary relations. Moreover, we show that $J\neq \emptyset$. We start with an auxiliary lemma.

\begin{lemma}\label{lem:two-boundary}
Let $B_j,C_j\in\GL(d,\R)$ and $N_j\to\infty$.  Suppose that
\[
 \frac1{N_j}\log\sigma_i(B_j)\to a_i \; \text{ and }\; \log\norm{C_j^{\pm1}}=o(N_j) ,
\]
and assume that the numbers $a_i$ satisfy \eqref{eq:plateau}.
At every $q\in J$ let $u_{q,j}$ be a unit top left singular vector of
$\Lambda^qB_j$ and let $v_{q,j}^*$ be the corresponding unit top right
covector.  Assume that
\begin{equation}\label{eq:boundary-transversality}
 -\log\abs{v_{q,j}^*(\widetilde{\Lambda^q C_j u_{q,j}})}=o(N_j),
\end{equation}
where the tilde denotes unit normalization. Then, ordering eigenvalues by decreasing modulus,
\[
 \log\frac{|\rho_k(C_jB_j)|}{|\rho_{k+1}(C_jB_j)|}=o(N_j).
\]
\end{lemma}
\begin{remark}
Geometrically, $u_{q,j}$ represents the left fast $q$-plane $F_{q,j}$ of $B_j$, while
$\ker v_{q,j}^*$ is the hyperplane section corresponding to $q$-planes that meet the
right slow $(d-q)$-plane $S_{q,j}$ of $B_j$.  Thus \eqref{eq:boundary-transversality} says that $C_j(F_{q,j})$ does not approach this hyperplane at an exponential rate. Equivalently, it remains subexponentially transverse to $S_{q,j}$.
\end{remark}

\begin{proof}
Let
\[
 A_q=a_1+\cdots+a_q.
\]
Note that the quantities $N_j^{-1}\log|\rho_i(C_jB_j)|$ are uniformly bounded.
Indeed, the convergence of the normalized singular values of $B_j$ gives
\[
 \log\norm{B_j}=O(N_j)\; \text{ and }\;
 \log\norm{B_j^{-1}}=O(N_j),
\]
while $\log\norm{C_j^{\pm1}}=o(N_j)$.  Hence
\[
 \log\norm{C_jB_j}=O(N_j)\; \text{ and }\;
 \log\norm{(C_jB_j)^{-1}}=O(N_j).
\]
Every eigenvalue $\rho$ of an invertible matrix $M$ satisfies
$\norm{M^{-1}}^{-1}\le |\rho|\le\norm M$ (apply the upper bound also to the
eigenvalue $\rho^{-1}$ of $M^{-1}$).  Therefore
\[
 \bigl|\log|\rho_i(C_jB_j)|\bigr|\le C N_j
\]
for a constant $C$ independent of $i$ and $j$.  After passing to a subsequence
we may therefore write
\[
 b_i=\lim_j\frac1{N_j}\log|\rho_i(C_jB_j)|\; \text{ and }\;
 B_q=b_1+\cdots+b_q.
\]
We will now observe that $A_q=B_q$. Indeed, for every $q$,
\begin{equation}\label{eq:BqAq}
 B_q\le A_q,
\end{equation}
since
\[
 \rho(\Lambda^q(C_jB_j))
 \le\norm{\Lambda^qC_j}\norm{\Lambda^qB_j}
\]
and $\log\norm{C_j^{\pm1}}=o(N_j)$.

On the other hand, at a genuine boundary $q$, let
\[
 B_j=U_j\operatorname{diag}(\sigma_1(B_j),\ldots,\sigma_d(B_j))V_j^*
\]
be a singular-value decomposition, with left singular vectors $u_{1,j},\ldots,u_{d,j}$ and right
singular vectors $v_{1,j},\ldots,v_{d,j}$.  Taking the $q$-th exterior power
wedges this decomposition: for every multi-index
$I=\{i_1<\cdots<i_q\}$, the vector
$v_{i_1,j}\wedge\cdots\wedge v_{i_q,j}$ is sent to
\[
 \bigl(\sigma_{i_1}(B_j)\cdots\sigma_{i_q}(B_j)\bigr)
 (u_{i_1,j}\wedge\cdots\wedge u_{i_q,j})
\]
by $\Lambda^qB_j$. Thus the largest singular value of $\Lambda^qB_j$ is the product corresponding
to $I=\{1,\ldots,q\}$, while the next largest is obtained by replacing
$\sigma_q(B_j)$ with $\sigma_{q+1}(B_j)$.  Writing the first rank-one singular
term separately gives
\[
 \Lambda^qB_j =  s_{q,j}u_{q,j}\otimes v_{q,j}^*+R_{q,j},
\]
where $u_{q,j}$ denotes the unit decomposable vector spanning
$u_{1,j}\wedge\cdots\wedge u_{q,j}$ which is the top left
singular vector of $\Lambda^qB_j $,  $v_{q,j}^*$ is the unit covector dual
to the right singular line $v_{1,j}\wedge\cdots\wedge v_{q,j}$ and $\otimes $ denotes the tensor product which is defined as follows: given subspaces $U$ and $V$, for $u \in U$ and $v^\ast \in V^*$,  $u \otimes v^\ast$ is the linear map from $V$ to $U$ defined by $(u \otimes v^\ast)(v) = v^\ast(v)u$ for every $v \in V$. Note that, for notational convenience, we use the same notation $u_{q,j}$ and $v_{q,j}$ for the top left and right singular vectors of $\Lambda^q B_j$ and for the $q$-th left and right singular vectors of $B_j$, respectively. The meaning will always be clear from the context. Consequently
\[
 s_{q,j}=\sigma_1(B_j)\cdots\sigma_q(B_j)\; \text{ and }\;
 \frac{\norm{R_{q,j}}}{s_{q,j}} = \frac{\sigma_{q+1}(B_j)}{\sigma_q(B_j)} = e^{-(a_q-a_{q+1})N_j+o(N_j)}
\]
and the Pl\"ucker line $[u_{q,j}]$ represents $F_{q,j}$ and
$\ker v_{q,j}^*$ is the hyperplane section of $q$-planes meeting
$S_{q,j}$.

Put $L_{q,j}=\Lambda^qC_j$.  Then
\begin{align*}
\operatorname{tr}(\Lambda^q(C_jB_j)) 
&= \operatorname{tr}\left( L_{q,j} \left( s_{q,j}u_{q,j}\otimes v_{q,j}^* + R_{q,j} \right) \right) \\
&= \operatorname{tr}\left( s_{q,j} (L_{q,j}u_{q,j})\otimes v_{q,j}^* + L_{q,j}R_{q,j} \right) \\
&= s_{q,j} \operatorname{tr}\left( (L_{q,j}u_{q,j})\otimes v_{q,j}^* \right) + \operatorname{tr}(L_{q,j}R_{q,j}) \\
&= s_{q,j} v_{q,j}^*(L_{q,j}u_{q,j}) + O(\norm{L_{q,j}}\norm{R_{q,j}}).
\end{align*}
Moreover, using \eqref{eq:boundary-transversality} (recall the normalization used in the referred equation) and
$\norm{L_{q,j}^{\pm1}}=e^{o(N_j)}$ we get that
\[
 |v_{q,j}^*(L_{q,j}u_{q,j})|
 \ge m(L_{q,j})\,e^{-o(N_j)}=e^{-o(N_j)}.
\]
Thus, the error term is exponentially smaller than this leading term. Indeed,
\[
 \frac{\norm{L_{q,j}}\norm{R_{q,j}}}
 {s_{q,j}|v_{q,j}^*(L_{q,j}u_{q,j})|}
 \le
 \bol(L_{q,j})e^{-(a_q-a_{q+1})N_j+o(N_j)}\longrightarrow0.
\]
Consequently there is no exponential cancellation and
\[
 |\operatorname{tr}(\Lambda^q(C_jB_j))|
 \ge s_{q,j}e^{-o(N_j)}.
\]
Since in fixed dimension $|\operatorname{tr}M|\le C\rho(M)$,
we obtain that $A_q\leq B_q$. Consequently, $B_q=A_q$ at every transverse genuine boundary.  Equality is
automatic at $q=0$. At $q=d$ it follows from
\[
 \det(C_jB_j)=\det(C_j)\det(B_j)
\]
and $\log|\det C_j|=o(N_j)$, which is a consequence of
$\log\|C_j^{\pm1}\|=o(N_j)$.

Now, the sequence $q\mapsto B_q$ is concave, whereas $q\mapsto A_q$ is affine on
$[r,s]$.  Together with \eqref{eq:BqAq} and equality at the two endpoints,
concavity gives
\[
 B_q=A_q \; \text{ for every }\; r\le q\le s.
\]
Hence
\[
 b_{r+1}=\cdots=b_s,
\]
and in particular $b_k=b_{k+1}$.  The argument applies to every subsequence
along which the normalized logarithms of the eigenvalue moduli converge.
Therefore every subsequential limit of
\[
 \frac1{N_j}\log\frac{|\rho_k(C_jB_j)|}{|\rho_{k+1}(C_jB_j)|}
\]
is zero, and the whole sequence converges to zero.  This is the claimed
$o(N_j)$ estimate.
\end{proof}

\begin{proposition}\label{prop:all-path}
For every holonomy path $G:E_{\hat y}\to E_{\hat p}$, there exists a genuine boundary
$q\in J$ such that
\begin{equation}\label{eq:all-path}
 S_q\cap GF_q\ne\{0\}.
\end{equation}
\end{proposition}

\begin{proof} We use the notation from previous sections.
Fix $G$ and apply Lemma \ref{lem:path-perturb} to obtain paths
$G_j:E_{\hat y_j}\to E_{\hat p}$ with $G_j\to G$.  By
Lemma \ref{lem:exact-concat},
\[
 {\mathcal B}_j:=G_jD_j:E_{\hat p}\to E_{\hat p}
\]
is itself a based holonomy loop.  In the notation of
Lemma \ref{lem:exact-concat}, if
\[G_j=H^s_{\hat b_j,\hat p}\hat A^L(\hat c_j) H^u_{\hat y_j,\hat c_j}\]
then
\[\begin{split}
{\mathcal B}_j&=
H^s_{\hat b_j,\hat p} \hat A^L(\hat c_j) H^u_{\hat y_j,\hat c_j} \hat A^{N_j}(\hat a_j) H^u_{\hat p,\hat a_j}\\
&=H^s_{\hat b_j,\hat p} \hat A^{N_j+L}(\hat a'_j)
H^u_{\hat p,\hat a'_j}
\end{split}
\]
with $\hat a'_j=\hat f^{-N_j}\hat c_j$.
In particular, the concatenated path is represented by an orbit segment beginning at $\hat a'_j$. Now, since
$\hat c_j\in W^u_{\rm loc}(\hat y_j)$ and $\hat y_j=\hat f^{N_j}\hat a_j$, backward contraction shows that $d_\theta(\hat a'_j, \hat a_j)=d_\theta(\hat f^{-N_j}\hat c_j, \hat f^{-N_j} \hat y_j)\to 0$ and, consequently, as $\hat a_j\to \hat p$, $\hat a'_j$ converges to $\hat p$. Moreover, as $\hat a_j\in W^u_{\rm loc}(\hat p)$ we have that $\hat a'_j\in W^u_{\rm loc}(\hat p)$.
We may therefore apply Lemma \ref{lem:periodize} which gives periodic points $\hat q_j$ of periods $M_j=N_j+o(N_j)$ and
corrections $R_j\to\Id$ such that the periodic return $\hat A^{M_j}(\hat q_j)$ has the
same eigenvalues as
\[
 R_j{\mathcal B}_j=(R_jG_j)D_j.
\]
Put $C_j=R_jG_j$.  Then $C_j\to G$ and
$\log\norm{C_j^{\pm1}}=o(N_j)$.  If the conclusion of \eqref{eq:all-path} failed, then for every
$q\in J$ the limiting incidence would be transverse.  Consequently the
Pl\"ucker transversality determinant of
$S_{q,j}$ and $C_jF_{q,j}$ would be bounded below by a positive constant for
all large $j$.  Lemma \ref{lem:two-boundary} applied to $B_j=D_j$ and this
$C_j$ would therefore give
\[
 \log\frac{|\rho_k(\hat A^{M_j}(\hat q_j))|}
               {|\rho_{k+1}(\hat A^{M_j}(\hat q_j))|}=o(N_j)=o(M_j),
\]
contradicting \eqref{eq:periodic-gap}.
\end{proof}

\begin{corollary}\label{cor:J-nonempty}
Recall the set $J$ given in \eqref{eq:J-boundaries}. Then
    \[J\neq \emptyset.\]
\end{corollary}
\begin{proof}
    Proceeding as in the proof of Proposition \ref{prop:all-path}, if $J=\emptyset$ then we could apply Lemma \ref{lem:two-boundary}, because condition \eqref{eq:boundary-transversality} is vacuous, and obtain periodic points contradicting \eqref{eq:periodic-gap}.
\end{proof}
From now on, we assume that at least one of $r$ or $s$ is a genuine boundary.  In
particular, the dimension $d$ of our fiber satisfies $d\ge3$. Indeed, for $d=2$ the only plateau containing $k=1$ and $k+1=2$ is the full plateau just excluded.

\section{From the obstruction to a finite invariant arrangement}\label{sec:arrangement} 

In this section we use the relation obtained in Proposition \ref{prop:all-path} to obtain a finite family of
subspaces that moves continuously with the base point. After passing to the
one-sided model, we use a $u$-state and a maximal-section argument to build
this family in an exterior power. At the fixed periodic point, one member of
the family lies in the kernel of the limiting covector coming from the bad
blocks.

\subsection{Partial flag bundle and the set $\mathscr K$}
Let $\hat y\in \widehat\Sigma$ be the point given in Section \ref{sec:bad-blocks} and consider the (partial) flag space on the fiber $E_{\hat y}$  given by
\[
 \cF_{r,s}(E_{\hat y})
 =\{(V_r,V_s)\in
\Gr(r,E_{\hat y})\times\Gr(s,E_{\hat y}):V_r\subset V_s\}.
\]
Here $r$ and $s$ are exactly the two endpoints of the maximal singular-value
plateau introduced in \eqref{eq:plateau}.  For every genuine boundary
$q\in J$, Section~\ref{sec:bad-blocks} defined $F_q\subset E_{\hat y}$ as the
limit of the left fast $q$-planes $F_{q,j}$ of the bad maps $D_j$. At a formal endpoint we set $F_0=\{0\}$ or $F_d=E_{\hat y}$.  Thus the two endpoint spaces
$F_r$ and $F_s$ satisfy $F_r\subset F_s$.  When $r=0$ we interpret
$V_r=\{0\}$, and when $s=d$ we interpret $V_s=\mathbb R^d$. In particular, these formal endpoints add no data to the flag.

Take
\[
 \xi=(F_r,F_s)\in\cF_{r,s}(E_{\hat y}).
\]
Let $\mathscr K$ be the closure of the set obtained from $\xi$ by finitely many applications of the forward $\hat A$-skew product and local unstable holonomies. Equivalently, $\mathscr K$ is the smallest compact subset of the partial-flag bundle over $\widehat\Sigma$ which contains $\xi$, is forward invariant under the $\hat A$-skew product, and is saturated by local unstable holonomies. For
$\hat x\in\widehat\Sigma$ we write
\[
 \mathscr K_{\hat x}
 :=\{\zeta:(\hat x,\zeta)\in\mathscr K\}
\]
for the fiber of $\mathscr K$ over $\hat x$.

We next put the part of $\mathscr K$ lying over the local stable set of
$\hat p$ into the single fiber $E_{\hat p}$.  This is useful because the stable
reduction to the one-sided model transports precisely such flags to the
preferred fiber over $\hat p$, while Proposition~\ref{prop:all-path} gives an
obstruction for flags obtained by holonomy paths ending at $\hat p$.  Define
\begin{equation*}\label{eq:Khat}
 \hat{\mathscr K}_{\hat p}
 :=\overline{\{H^s_{\hat z,\hat p}\zeta:
       \hat z\in W^s_{\rm loc}(\hat p),\ \zeta\in\mathscr K_{\hat z}\}},
\end{equation*}
where stable holonomy acts componentwise on a flag
\[
 H^s_{\hat z,\hat p}(V_r,V_s)
 =(H^s_{\hat z,\hat p}V_r,H^s_{\hat z,\hat p}V_s).
\]
We claim that
\begin{equation}\label{eq:path-support}
 \hat{\mathscr K}_{\hat p}
 \subset
 \overline{\{G\xi:G:E_{\hat y}\to E_{\hat p}\text{ is a holonomy path}\}}.
\end{equation}
Indeed, before taking the closure in the definition of $\mathscr K$, every
point is obtained from $\xi$ by finitely many forward orbit pieces and unstable
holonomies.  Using
\[
 H^u_{\hat f\hat x,\hat f\hat y}\hat A(\hat x)
 =\hat A(\hat y)H^u_{\hat x,\hat y},
\]
all unstable holonomies can be moved to the beginning of the orbit segment.
The resulting word has the form
$\hat A^n(\hat x')H^u_{\hat y,\hat x'}$.  If its terminal point
$\hat z$ belongs to $W^s_{\rm loc}(\hat p)$, appending
$H^s_{\hat z,\hat p}$ gives a holonomy path from $\hat y$ to $\hat p$.
For a point in the closure, take terminal points
$\hat z_j\to\hat z\in W^s_{\rm loc}(\hat p)$ and put
$\hat w_j=[\hat z_j,\hat p]$.  Appending first
$H^u_{\hat z_j,\hat w_j}$ and then $H^s_{\hat w_j,\hat p}$, and moving the
unstable holonomy backward by equivariance, again gives a genuine holonomy
path.  Since $\hat w_j\to\hat z$ and
$H^u_{\hat z_j,\hat w_j}\to\Id$, this proves \eqref{eq:path-support}.

\subsection{$u$-states on $\mathscr K$}\label{sec:u-states on K}
In what follows, we will need an ergodic $u$-state supported on $\mathscr K$. We will construct such a measure in Lemma \ref{lem:ustate-in-K} but first we record the standard fact that passing to ergodic components does not destroy the $u$-state property in the present symbolic setting.

\begin{lemma}\label{lem:ergodic-components-ustate}
Let $\hat F$ be a continuous bundle skew product over a two-sided subshift of finite type
$\hat f:\widehat\Sigma\to\widehat\Sigma$, with continuous equivariant local
unstable holonomies.  Let $\hat\mu$ be an ergodic $\hat f$-invariant
probability measure with local product structure.  If $\hat m$ is an
$\hat F$-invariant $u$-state projecting to $\hat\mu$, then almost every
ergodic component of $\hat m$ also projects to $\hat\mu$ and is a $u$-state.
Moreover, if $\mathscr C$ is a Borel set and
$\hat m(\mathscr C)=1$, then almost every such ergodic component gives full
mass to $\mathscr C$.
\end{lemma}

\begin{proof}
    The proof is standard and follows directly from the ergodic decomposition theorem, where the ergodicity of $\hat \mu$ forces almost every ergodic component to project to $\hat \mu$, and the uniqueness of fiber disintegration ensures that the holonomy invariance of the conditional measures is inherited componentwise. See \cite[Exercise 5.23]{VianaLectures}.
\end{proof}

\begin{lemma}
\label{lem:ustate-in-K}
Suppose we are in the setting of Theorem \ref{thm:main} and, moreover, that the subshift of finite type $\widehat\Sigma$ is mixing. Let $\hat\mu$ be a fully supported one-step ergodic Markov measure on $\widehat \Sigma$.
Then there exists an ergodic $u$-state
$\hat m$ on the partial-flag bundle which projects to $\hat\mu$ and
satisfies $\hat m(\mathscr K)=1$.
\end{lemma}

\begin{proof}
We first show that every base point in $\widehat\Sigma$ has at least one flag in $\mathscr K$.  Let
$\operatorname{pr}(\mathscr K)\subset\widehat\Sigma$ be the projection of
$\mathscr K$ to the base.  This set is compact, forward invariant, locally
unstable saturated, and contains $\hat y$.  If
$C\subset\widehat\Sigma$ is a cylinder, mixing allows us to choose
$\hat y'\in W^u_{\rm loc}(\hat y)$ and $n\ge0$ with
$\hat f^n\hat y'\in C$.  Hence
\[
 \bigcup_{n\ge0}\hat f^nW^u_{\rm loc}(\hat y)
\]
is dense in $\widehat\Sigma$.  Since this dense set is contained in
$\operatorname{pr}(\mathscr K)$ and the latter is closed, we have
$\operatorname{pr}(\mathscr K)=\widehat\Sigma$.  Thus
$\mathscr K_{\hat x}\ne\varnothing$ for every $\hat x\in \widehat\Sigma$.

We now construct explicitly a probability measure whose conditional measures are
invariant under unstable holonomy.  A local unstable set consists of all
points having the same left-infinite past.  For each admissible past $x^-$,
choose one two-sided sequence $r(x^-)$ having that past.  We may make this choice
Borel, for instance by fixing, for each possible zero-coordinate, one admissible
right-infinite continuation.  Since the compact fiber
$\mathscr K_{r(x^-)}$ is nonempty, choose a flag
\[
 \zeta(x^-)\in\mathscr K_{r(x^-)}
\]
measurably in $x^-$. This can be done because $\mathscr K$ is compact, so the fiber relation
\[ \hat x\longmapsto\mathscr K_{\hat x} \]
has a closed graph and nonempty compact values. After composition with the
Borel map $x^-\mapsto r(x^-)$, we obtain the set-valued map
\[ \Phi(x^-):=\mathscr K_{r(x^-)}.\]
The Kuratowski--Ryll-Nardzewski measurable selection theorem then provides
a Borel selector
\[ \zeta:x^-\longmapsto\zeta(x^-)\in\Phi(x^-)=\mathscr K_{r(x^-)}.\]

For an arbitrary $\hat x\in\widehat\Sigma$, put $x^-=P^s\hat x$ and define
\begin{equation*}\label{eq:section-K-u}
 s(\hat x)
 :=H^u_{r(x^-),\hat x}\,\zeta(x^-).
\end{equation*}
Note that the two points $r(x^-)$ and $\hat x$ lie on the same local unstable set.
Because $\mathscr K$ is unstable-holonomy saturated,
$s(\hat x)\in\mathscr K_{\hat x}$.  Moreover, if
$\hat x$ and $\hat y$ lie on the same local unstable set, then they have the
same past $x^-$, and the composition rule for holonomies gives
\[
 s(\hat y)=H^u_{\hat x,\hat y}s(\hat x).
\]
Thus the Dirac masses $\delta_{s(\hat x)}$ are exactly transported into one
another by unstable holonomy.

Define $\hat m_0$ by putting the Dirac mass
$\delta_{s(\hat x)}$ in the flag fiber over $\hat x$ and then integrating over
$\hat\mu$. Equivalently, for every continuous function $\Phi$ on the
partial-flag bundle, define $\hat m_0$ by
\begin{equation*}\label{eq:m0-explicit}
 \int\Phi\,d\hat m_0
 =\int_{\widehat\Sigma}\Phi(\hat x,s(\hat x))\,d\hat\mu(\hat x).
\end{equation*}
It follows directly from the definition that $\hat m_0$ projects to
$\hat\mu$, is supported on $\mathscr K$, and has fiber conditionals
$\delta_{s(\hat x)}$, which are invariant under local unstable holonomies.
The measure $\hat m_0$ need not yet be invariant under the skew product.

Let $\hat F$ denote the partial-flag skew product.  Holonomy equivariance
implies that if a family of fiber conditionals is invariant under unstable
holonomy, then the same is true after applying $\hat F_*$.  Also,
$\hat F(\mathscr K)\subset\mathscr K$.  Therefore each Ces\`aro average
\[
 \hat m_N
 =\frac1N\sum_{j=0}^{N-1}\hat F_*^j\hat m_0
\]
projects to $\hat\mu$, gives full mass to $\mathscr K$, and has a
$u$-holonomy-invariant fiber disintegration, although the finite average need
not itself be invariant.  Since the
flag bundle is compact, the sequence $(\hat m_N)$ has weak-star
accumulation points.  Any accumulation point $\hat m$ is
$\hat F$-invariant by the usual Ces\`aro argument. Moreover, the unstable-holonomy relation also passes to the limit and $\hat m$ is a $u$-state (see the proof of \cite[Lemma 4.3]{BackesBrownButler}).

Finally, applying Lemma~\ref{lem:ergodic-components-ustate} to the invariant $u$-state $\hat m$ just constructed and to the set $\mathscr K$ we get that almost every ergodic component of $\hat m$ still projects to $\hat\mu$, is a $u$-state, and gives full mass to $\mathscr K$.  Choosing one of these components gives the ergodic $u$-state required in the lemma.
\end{proof}

Choose $\hat m$ as in Lemma~\ref{lem:ustate-in-K}, and choose in the
stable reduction of Subsection~\ref{subsec:stable-reduction} the reference past
for the symbol $p_0$ to be the past of the fixed point $\hat p$, so that
$\varphi^u(\hat p)=\hat p$.  Put $\mu=P_*\hat\mu$.

It is useful first to reduce the entire \emph{partial-flag} measure.  Define the
stable factor map
\begin{equation*}\label{eq:stable-factor-map}
 \mathcal Q_{r,s}(\hat z,(V_r,V_s))
 :=\bigl(P\hat z,
        (H^s_{\hat z,\varphi^u(\hat z)}V_r,
         H^s_{\hat z,\varphi^u(\hat z)}V_s)\bigr)
\end{equation*}
and put
\[
 \bar m:=(\mathcal Q_{r,s})_*\hat m.
\]
The map $\mathcal Q_{r,s}$ is the partial-flag version of the stable reduction
from Subsection~\ref{subsec:stable-reduction}, in particular, it intertwines the two-sided partial-flag skew product with the one-sided one.  The usual continuity argument
for a $u$-state applies without change to this compact partial-flag fiber, so we
write $\bar m_x$ for its weak-star continuous fiber disintegration.  For
reference, these conditionals are given by
\begin{equation*}\label{eq:flag-stable-reduction}
 \bar m_x
 =\int_{P^{-1}(x)}
 \bigl(H^s_{\hat z,\varphi^u(\hat z)}\bigr)_*
 \hat m_{\hat z}\,d\hat\mu_x^s(\hat z),
\end{equation*}
with holonomy acting componentwise on flags, at points where the original fiber
disintegration $\hat m_{\hat z}$ is defined. In what follows, we use the continuous version of this disintegration.

We now explain the reason for having introduced $\hat{\mathscr K}_{\hat p}$.  Since $\hat m(\mathscr K)=1$, the reduced measure $\bar m$ is
supported on the compact set $\mathcal Q_{r,s}(\mathscr K)$.  If
$P\hat z=p:=P\hat p$, then $\hat z\in W^s_{\rm loc}(\hat p)$ and, by our choice
of reference past, $\varphi^u(\hat z)=\hat p$.  Hence the fiber over $p$ of
$\mathcal Q_{r,s}(\mathscr K)$ is contained in
$\hat{\mathscr K}_{\hat p}$.  Because $\mu$ has full support and
$x\mapsto\bar m_x$ is continuous, the same support inclusion holds for the conditional at $p$. Indeed, if a point of $\supp\bar m_p$ lay outside the closed fiber of $\mathcal Q_{r,s}(\mathscr K)$ over $p$, one could choose a product neighborhood disjoint from $\mathcal Q_{r,s}(\mathscr K)$ on which $\bar m_p$ has positive fiber mass. By weak-star continuity of $x\mapsto\bar m_x$, the corresponding fiber neighborhood would have positive conditional mass for all $x$ in a sufficiently small neighborhood of $p$. Since $\mu$ has full support, integrating over that neighborhood would contradict $\bar m(\mathcal Q_{r,s}(\mathscr K))=1$. Therefore
\begin{equation}\label{eq:flag-support-Khat}
 \supp \bar m_p\subset\hat{\mathscr K}_{\hat p}.
\end{equation}
Thus $\bar m_p$ is a measure on the partial-flag fiber over $p$. In particular, it is not a Grassmannian measure.

For $q\in\{r,s\}$, let $\pi_q:\cF_{r,s}(E_{\hat p})\longrightarrow\Gr_q(E_{\hat p})$ given by 
\[
 \pi_q(V_r,V_s)=V_q,
\]
denote the natural projection onto the $q$-plane coordinate.  When
$q\in J$ is a genuine boundary, define the $q$-Grassmannian marginal of the
flag conditional at $p$ by
\begin{equation*}\label{eq:mpq-definition}
 m_p^{(q)}:=(\pi_q)_*\bar m_p.
\end{equation*}

By Corollary~\ref{cor:J-nonempty}, the set $J$ in
\eqref{eq:J-boundaries} is nonempty.  For $q\in J$, let $S_q$ be given in Section \ref{sec:bad-blocks} and define the section
\[
 \cH_q=\{V\in\Gr(q,E_{\hat p}):V\cap S_q\ne\{0\}\}
 =\iota_q^{-1}\bigl(\mathbb P(\ker\ell_q)\bigr)
\]
where $\iota_q:\Gr(q,E_{\hat p})\to \mathbb P(\Lambda^qE_{\hat p})$ is the Pl\"ucker embedding and $\ell_q\in(\Lambda^qE_{\hat p})^*$ is a suitable nonzero covector (recall Section \ref{sec: grassmanian}). Proposition~\ref{prop:all-path} says that every path image $G\xi$ belongs to at
least one of the sets $\pi_q^{-1}\cH_q$, $q\in J$.  This finite union is closed.
Therefore \eqref{eq:path-support} implies
\[
 \hat{\mathscr K}_{\hat p}
 \subset\bigcup_{q\in J}\pi_q^{-1}\cH_q
\]
which together with \eqref{eq:flag-support-Khat},  gives us that
\[
 \supp\bar m_p
 \subset\bigcup_{q\in J}\pi_q^{-1}\cH_q.
\]
Hence for at least one $q\in J$,
\begin{equation}\label{eq:positive-H}
 m_p^{(q)}(\cH_q)>0.
\end{equation}
Fix such a boundary $q$ for the rest of the proof.

We now pass from the flag measure to this chosen Grassmannian coordinate.  Put 
\begin{equation}\label{eq:q-marginal-two-sided}
 \hat m^{(q)}:=(\pi_q)_*\hat m
\end{equation}
where $\pi_q$ here denotes the projection defined on the  whole bundle which is analogous to the projection $\pi_q$ defined above. It will be clear from the context which one we are using so there is no risk of confusion.
The measure $\hat m^{(q)}$ is an ergodic $u$-state on the
$q$-Grassmannian bundle: the projection $\pi_q$ intertwines both the cocycle
action and unstable holonomy, so the $u$-state property passes to the
pushforward, and a factor of an ergodic measure is ergodic.

Let $m$ be the one-sided Grassmannian measure obtained by applying the stable
reduction to $\hat m^{(q)}$, and denote its fiber conditionals by $m_x$.
They have a weak-star continuous version as observed in Section \ref{subsec:stable-reduction}. Recall the stable factor map  $\mathcal Q_q$ on the $q$-Grassmannian bundle given in \eqref{eq:stable-factor-map-prelim}.  Since stable holonomy acts componentwise on flags,
$\mathcal Q_q\circ\pi_q=\pi_q\circ\mathcal Q_{r,s}$.  Thus stable reduction
commutes with $\pi_q$.  Consequently
\begin{equation}\label{eq:q-marginal-commutes-reduction}
 m_x=(\pi_q)_*\bar m_x
 \;\text{ for every }x\in\Sigma,
\end{equation}
and in particular
\begin{equation}\label{eq:mp-equals-mpq}
 m_p=m_p^{(q)}.
\end{equation}
This fixes the notation used below: $\bar m_x$ always denotes a partial-flag
conditional, whereas $m_x$ denotes the conditional of the selected
$q$-Grassmannian marginal.  Since the stable reduction is a factor of the
two-sided Grassmannian skew product (after the stable reduction), the ergodicity of $\hat m^{(q)}$ implies that $m$ is ergodic.

Normalize the Pl\"ucker covector $\ell_q$ defining $\cH_q$ to have norm one.
Recall that $S_{q,j}\to S_q$.  If $v_{q,j}^*$ denotes the unit top right
singular covector of $\Lambda^qD_j$, then
$\ker v_{q,j}^*$ defines, under the Pl\"ucker embedding, the hyperplane
section of $q$-planes meeting $S_{q,j}$. Therefore every accumulation point $v^*$ of the unit covectors $v_{q,j}^*$ defines the limiting hyperplane section
$\cH_q$, so $\ker v^*=\ker\ell_q$.  Since $v^*$ and $\ell_q$ have norm one,
$v^*=\pm\ell_q$.  After changing signs and passing to a subsequence, we may
therefore assume
\begin{equation}\label{eq:vqj converges to lq}
 v_{q,j}^*\longrightarrow\ell_q.
\end{equation}

\subsection{Maximal arrangement} Before we state our main auxiliary result of this section, we make one piece of notation
explicit.  If $W\subset\Lambda^q\mathbb R^d$ is a linear subspace, then
\[
 \mathbb PW
 :=\{[\omega]\in\mathbb P(\Lambda^q\mathbb R^d):
          0\ne\omega\in W\}
\]
is its projectivization.  We define the associated linear section of the
Grassmannian by
\begin{equation}\label{eq:XW-definition}
 X(W):=\iota_q^{-1}(\mathbb PW)
 =\{V\in\Gr(q,d):\iota_q(V)\in\mathbb PW\}.
\end{equation}
Thus $X(W)$ is a Borel subset of $\Gr(q,d)$, so a conditional probability
$m_x$ on the fiber $\{x\}\times\Gr(q,d)$ can be evaluated on it without any
change of ambient space.  If we identify the Grassmannian with its Pl\"ucker
image, then $X(W)$ is exactly the shorthand
$\mathbb PW\cap\Gr(q,d)$.

We now prove the maximal-section construction needed in the argument.  The
broader strategy of extracting finite invariant projective data from
holonomy-invariant conditional measures goes back to the pioneering work of Bonatti--Viana
\cite{BonattiViana2004} and was subsequently used in several different works like \cite{AvilaViana2007, BackesBrownButler, DeWittMitsutani}. 

\begin{lemma}\label{lem:maximal-sections}
Let $f:\Sigma\to\Sigma$ be a mixing one-sided subshift of finite type and let $\mu$ be a
fully supported one-step Markov measure.  Let $A:\Sigma\to\GL(d,\mathbb R)$ be
continuous and let $m$ be an ergodic invariant probability for the skew product
\[
 (x,V)\longmapsto (fx,A(x)V)
 \; \text{ on }\;\Sigma\times\Gr(q,d),
\]
projecting to $\mu$.  Assume that $m$ admits a weak-star continuous
disintegration $x\mapsto m_x$.  We identify the fiber
$\{x\}\times\Gr(q,d)$ with $\Gr(q,d)$, so each $m_x$ is a Borel probability
measure on $\Gr(q,d)$.  Embed $\Gr(q,d)$ in
$\mathbb P(\Lambda^q\mathbb R^d)$ by the Pl\"ucker embedding.

If some $m_x$ gives positive mass to a \emph{proper} linear section, then there exist 
integers $d_0, N_0\ge1$, a number $\gamma_0>0$, and, for every
$x\in \Sigma$, exactly ${N_0}$ subspaces
\[
 W_1(x),\ldots,W_{N_0}(x)\subset\Lambda^q\mathbb R^d \; \text{ with }\; \dim W_i(x)=d_0,
\]
exhibiting the following properties:
\begin{enumerate}[label=\textup{(\roman*)}]
\item If $X_i(x):=X(W_i(x))=\iota_q^{-1}(\mathbb P W_i(x))$, then
      $m_x(X_i(x))=\gamma_0$.
\item Distinct $X_i(x)$ have zero $m_x$-mass intersection.
\item The unordered family $\{W_1(x),\ldots,W_{N_0}(x)\}$ varies continuously
      with $x$ and is permuted by $\Lambda^qA(x)$ meaning that
\begin{equation}\label{eq:max-transfer}
 \Lambda^qA(x)\{W_1(x),\ldots,W_{N_0}(x)\}
 =\{W_1(fx),\ldots,W_{N_0}(fx)\}.
\end{equation}
\item The union has full conditional mass. That is,
\[
 m_x\Bigl(\bigcup_{i=1}^{N_0} X_i(x)\Bigr)=1\; \text{ and  }\; {N_0}\gamma_0=1.
\]
\item Every $W_i(x)$ is spanned by its nonzero decomposable vectors.
\end{enumerate}
\end{lemma}

\begin{proof}
We start recalling some basic facts described in the end of Section \ref{subsec:stable-reduction}.  If $z\in f^{-n}(x)$, let $p_n(z\mid x)$ be the conditional probability, with respect to $\mu$, of the inverse branch ending at $z$ given the future $x$.  Since $\mu$ is a fully supported one-step Markov measure,
\[
 p_n(z\mid x)>0\; \text{ and }\;
 \sum_{z\in f^{-n}(x)}p_n(z\mid x)=1,
\]
and these weights are locally constant in $x$ on the cylinders on which the
corresponding inverse branches are defined.  Invariance of $m$ and Rokhlin
disintegration give, for $\mu$-almost every $x$,
\begin{equation}\label{eq:one-sided-transfer}
 m_x=
 \sum_{z\in f^{-n}(x)}p_n(z\mid x)\,
       (A^n(z))_*m_z.
\end{equation}
Both sides depend continuously on $x$ on each inverse-branch cylinder. In particular, since
$\mu$ has full support, \eqref{eq:one-sided-transfer} holds for every $x\in \Sigma$.
Equivalently, for every Borel set $E\subset\Gr(q,d)$,
\begin{equation}\label{eq:one-sided-transfer-set}
 m_x(E)=
 \sum_{z\in f^{-n}(x)}p_n(z\mid x)\,
 m_z\bigl((A^n(z))^{-1}E\bigr).
\end{equation}

For a linear subspace $W\subset\Lambda^q\mathbb R^d$, let $X(W)$ be the
Grassmannian section defined in \eqref{eq:XW-definition}.  Let $d_0$ be the least dimension for which
$m_x(X(W))>0$ for some $x$ and some $W$ of that dimension.  Such a dimension
exists and under the hypothesis of the lemma it is strictly smaller than
$\dim\Lambda^q\mathbb R^d$, although this strict inequality will not be needed
until the application.

Among all $d_0$-dimensional sections define
\[
 \gamma_0:=\sup\{m_x(X(W)):x\in\Sigma \text{ and } \dim W=d_0\}.
\]
Note that the supremum is attained.  Indeed, take $x_n\to x$ and $W_n\to W$ in the compact
Grassmannian of $d_0$-planes, with the masses tending to $\gamma_0$. Given $\varepsilon>0$, let $N_\varepsilon(X(W))$ denote the open $\varepsilon$-neighborhood of $X(W)$ in
$\Gr(q,d_0)$ for any fixed compatible metric. Then, for all large $n$ we have that $X(W_n)\subset N_\varepsilon(X(W))$. Indeed, otherwise we could choose $V_n\in X(W_n)$ outside $N_\varepsilon(X(W))$. Thus, compactness of $\Gr(q,d)$ would give a subsequence $V_n\to V$, and continuity of the Pl\"ucker embedding together with $W_n\to W$ would imply $V\in X(W)$, a contradiction.
Consequently, weak-star continuity of $m_x$ and the
Portmanteau theorem give
\[
 \gamma_0\le m_x\left(\overline{N_\varepsilon(X(W))}\right).
\]
Letting $\varepsilon\downarrow0$ yields
$m_x(X(W))\ge\gamma_0$, and equality follows from the definition of $\gamma_0$.

Let us call a $d_0$-plane $W$ \emph{maximal at $x$} if
$m_x(X(W))=\gamma_0$, and denote the set of such planes by $\mathcal M(x)$.
If $W,W'\in\mathcal M(x)$ are distinct, then $\dim(W\cap W')<d_0$ and
\[
 X(W)\cap X(W')=X(W\cap W').
\]
By minimality of $d_0$ this intersection has zero $m_x$-mass.  Hence every
$\mathcal M(x)$ is finite and
\begin{equation*}\label{eq:max-number-bound}
 \#\mathcal M(x)\le \gamma_0^{-1}.
\end{equation*}

Now, let us choose $x_*\in\Sigma$ at which a maximal section exists, and then choose $x_*$ so that the number
\[
 {N_0}:=\#\mathcal M(x_*)
\]
is maximal among all fibers.  Write
$\mathcal M(x_*)=\{W_1^*,\ldots,W_{N_0}^*\}$.  Fix $n\ge1$ and an arbitrary
$n$-preimage $w\in f^{-n}(x_*)$. Apply
\eqref{eq:one-sided-transfer-set} with $E=X(W_i^*)$.  Since
$m_{x_*}(X(W_i^*))=\gamma_0$ and every term
\[
 m_w\!\left(X\bigl((\Lambda^qA^n(w))^{-1}W_i^*\bigr)\right)
\]
is at most $\gamma_0$ by definition of $\gamma_0$, while all the weights are
strictly positive and sum to one, equality of the weighted average with
$\gamma_0$ forces every term to equal $\gamma_0$.  Thus, for every
$w\in f^{-n}(x_*)$ and every $i$,
\begin{equation}\label{eq:max-pullback}
 (\Lambda^qA^n(w))^{-1}W_i^*\in\mathcal M(w).
\end{equation}
Moreover, the ${N_0}$ pulled-back planes are distinct because $A^n(w)$ is invertible.
Maximality of ${N_0}$ therefore implies that they are \emph{all} the maximal
planes at $w$.  In particular
\begin{equation}\label{eq:max-transfer-dense}
 \Lambda^qA(w)\,\mathcal M(w)=\mathcal M(fw)
\end{equation}
whenever $w$ belongs to the full backward orbit
$\bigcup_{n\ge0}f^{-n}(x_*)$. For a mixing one-sided subshift of finite type this backward orbit is dense.

Now, given $x\in \Sigma$, let $x_j\to x$ with $x_j$ in the backward orbit of $x_*$.  After passing to a subsequence, label the ${N_0}$ maximal planes at $x_j$ so
that
\[
 W_i(x_j)\longrightarrow W_i \; \text{ for }\; i=1,\ldots,{N_0}.
\]
Thus, as observed above in the argument that the supremum defining $\gamma_0$ is attained, for each fixed $\varepsilon>0$ we have that $X(W_i(x_j)) \subset N_\varepsilon(X(W_i))$ for every large $j$. Thus, weak-star continuity of $m_x$ and Portmanteau,
followed by $\varepsilon\downarrow0$, gives us that 
$m_x(X(W_i))\ge\gamma_0$, hence equality. Moreover, the limits $W_i$ are distinct.  Indeed, if
$W_i=W_k=W$ for $i\ne k$, then, since $X(W_i(x_j))$ and $X(W_k(x_j))$ have zero-mass intersection, their union have mass $2\gamma_0$. Moreover, the union is eventually contained in every fixed neighborhood of $X(W)$, so passage to the limit would give $m_x(X(W))\ge2\gamma_0$, contradicting maximality of $\gamma_0$.
Consequently $x$ has at least ${N_0}$ maximal sections, and by maximality of
${N_0}$ it has exactly ${N_0}$.

The same argument shows continuity of the unordered finite family. Indeed, if
$x_j\to x$, every subsequential limit of the ${N_0}$-tuple of maximal planes
consists of ${N_0}$ distinct maximal planes at $x$, hence is exactly
$\mathcal M(x)$.  In particular, the limiting unordered ${N_0}$-tuple is
unique: two different subsequential limits would both have to equal the same
finite set $\mathcal M(x)$.  Equivalently, $\mathcal M(x_j)\to\mathcal M(x)$
in the Hausdorff metric on finite subsets of
$\Gr(d_0,\Lambda^q\R^d)$.  Since \eqref{eq:max-transfer-dense} holds on a
dense set and both sides vary continuously, it extends to every $x$, proving
\eqref{eq:max-transfer}.

Let
\[
 \mathscr L=
 \{(x,V):V\in\textstyle\bigcup_{W\in\mathcal M(x)}X(W)\}.
\]
Continuity of the finite arrangement makes $\mathscr L$ Borel. Moreover,
\eqref{eq:max-transfer} gives exact invariance, not merely forward invariance.
Indeed, for every inverse branch $z$ of $fx$, the inverse image under
$\Lambda^qA(z)$ of a maximal plane at $fz$ is, by the same weighted-average
argument used in \eqref{eq:max-pullback}, a maximal plane at $z$. Thus
$F_A^{-1}(\mathscr L)=\mathscr L$, where $F_A$ denotes the Grassmannian skew
product. Its conditional mass is the constant ${N_0}\gamma_0>0$ at every $x$,
because
distinct maximal sections have zero-mass intersections.  Hence
$m(\mathscr L)={N_0}\gamma_0>0$.  Ergodicity of $m$ and invariance of
$\mathscr L$ give $m(\mathscr L)=1$, so ${N_0}\gamma_0=1$.  Thus the union has
full $m_x$-mass for every $x$.

Finally fix $W=W_i(x)$ and let $W^{\rm dec}$ be the linear span in $W$ of all
nonzero decomposable vectors contained in $W$.  The Grassmannian section depends
only on those decomposable vectors, so $X(W^{\rm dec})=X(W)$ and has positive
mass.  Minimality of $d_0$ gives $\dim W^{\rm dec}\ge d_0$.  Since
$W^{\rm dec}\subset W$ and $\dim W=d_0$, equality holds.  This proves the final
claim.
\end{proof}

We now go back to the setting of Section \ref{sec:u-states on K}.
The measure $m$ and its continuous disintegration $x\mapsto m_x$ now satisfy
all hypotheses of Lemma~\ref{lem:maximal-sections}.  The transfer formula
\eqref{eq:one-sided-transfer} is precisely the inverse-branch formula for these
conditionals. Moreover, by \eqref{eq:positive-H} and \eqref{eq:mp-equals-mpq}, the conditional $m_p$ gives positive mass to the proper linear section
\[
 \cH_q=\iota_q^{-1}\bigl(\mathbb P(\ker\ell_q)\bigr).
\]
Thus the nontriviality hypothesis of Lemma~\ref{lem:maximal-sections} is
satisfied.  Applying the lemma, we obtain a continuous finite arrangement
\[
 W_1(x),\ldots,W_{N_0}(x)\subset\Lambda^q\mathbb R^d\;
 \text{ and  }\;
 X_i(x)=X(W_i(x))=\iota_q^{-1}(\mathbb PW_i(x)),
\]
with
\begin{equation*}\label{eq:fullmassarr}
 m_x\Bigl(\bigcup_iX_i(x)\Bigr)=1\; \text{ and }\; {N_0}\gamma_0=1,
\end{equation*}
and, moreover,
\begin{equation}\label{eq:generated-decomp}
 W_i(x)=\Span\{\omega:\ 0\ne\omega\in W_i(x)\text{ is decomposable}\}.
\end{equation}

Equation~\eqref{eq:positive-H} and the full-mass property imply that for some
component at $p$, say $W_{i_0}(p)$,
\[
 m_p\bigl(X_{i_0}(p)\cap\cH_q\bigr)>0.
\]
Finally, 
\begin{equation}\label{eq:WinH}
 0<W_{i_0}(p)\subset\ker\ell_q.
\end{equation}
Indeed, if $W_{i_0}(p)$ were not contained in $\ker\ell_q$, then
$W_{i_0}(p)\cap\ker\ell_q$ would have dimension strictly smaller than $d_0$ and
its Grassmannian section would have positive $m_p$-mass, contradicting the
definition of $d_0$.  

\section{Quantitative regularity of the maximal arrangement}\label{sec:quantitative-arrangement}

In the previous section we have built a continuous finite family of subspaces. Here we need more control on how this family changes from one point to another. We show that it is exactly invariant under stable holonomies and that, along unstable leaves, it changes at a controlled H\"older rate. This estimate will be used in the next section to obtain exact unstable invariance.

Before stating the proposition, let us make explicit how the one-sided
arrangement is lifted back to the two-sided shift.  Recall from
Subsection~\ref{subsec:stable-reduction} that, for each
$\hat x\in\widehat\Sigma$, the point $\varphi^u(\hat x)$ has the chosen
reference past corresponding to its zero-coordinate and the same future as
$\hat x$, so that
$P\varphi^u(\hat x)=P\hat x$.  Thus the one-sided component
$W_i(P\hat x)$ may be viewed in the fiber over $\varphi^u(\hat x)$.  We then
transport this subspace to the actual fiber over $\hat x$ by stable holonomy
and define
\begin{equation}\label{eq:lift-arrangement}
 \hat W_i(\hat x)
 =\Lambda^qH^s_{\varphi^u(\hat x),\hat x}\,W_i(P\hat x).
\end{equation}
In other words, the lift simply reverses the stable-holonomy identification
used in the reduction to the one-sided cocycle.  Since the components form an
unordered finite family, the labels $i$ are understood only up to permutation. The intertwining relation in the stable reduction and \eqref{eq:max-transfer} give the cocycle invariance
\begin{equation}\label{eq:lift-arrangement-invariance}
 \Lambda^q\hat A(\hat x)
 \{\hat W_i(\hat x)\}_{i=1}^{N_0}
 =\{\hat W_i(\hat f\hat x)\}_{i=1}^{N_0}.
\end{equation}

We start with an auxiliary result.

\begin{lemma}\label{lem:projective-separation}
For every $\rho>0$ there is $c(\rho)>0$ such that, for any two
$d_0$-planes $S,T\subset\Lambda^q\mathbb R^d$ and
$\Delta=d_{\Gr}(S,T)$, one can find a hyperplane $U<S$ for which
\begin{equation}\label{eq:principal-separation}
 \dist([v],\mathbb PT)\ge c(\rho)\Delta
\end{equation}
whenever $[v]\in\mathbb PS\setminus N_\rho(\mathbb PU)$.
\end{lemma}
Observe that, since $S\neq T$, there are directions in $S$ that are noticeably separated from $T$. However, there can also be some directions in $S$ that are very close to $T$. The lemma says that the latter phenomenon can be localized inside a single hyperplane. More precisely, there is a hyperplane $U<S$ such that every projective direction in $S$ that stays a fixed distance from $\mathbb P U$ is quantitatively separated from $\mathbb P T$, with a lower bound proportional to the Grassmannian distance $d_{\mathrm{Gr}}(S,T)$.

\begin{proof}
 If $S=T$ then $\Delta =0$ and there is nothing to prove. So, suppose $S\neq T$ and consider the orthogonal projection onto $T^\perp$, restricted to $S$ given by
\[L:=P_{T^\perp}|_S:S\longrightarrow T^\perp.\]

For every unit vector $v\in S$,
\[
\|Lv\|=\|P_{T^\perp}v\|=\dist(v,T).
\]
Thus $\|L\|$ measures the largest angle between a direction in $S$ and the
subspace $T$. For the fixed Grassmannian metric $d_{\Gr}$, this quantity is
comparable to the Grassmannian distance between $S$ and $T$. Namely,
\[
\|L\|\asymp d_{\Gr}(S,T)=\Delta,
\]
with constants depending only on the ambient dimension and on the chosen
equivalent metrics. Choose a unit vector $e\in T^\perp$ such that $\|L^*e\|=\|L\|$ where $L^*:T^\perp \to S$ denotes the adjoint of $L$.
Equivalently, $e$ is a left singular vector corresponding to the largest
singular value of $L$. Define a linear functional $\varphi:S\to\mathbb R$ by
\[
\varphi(v):=\langle Lv,e\rangle=\langle P_{T^\perp}v,e\rangle.
\]
Since $\varphi(v)=\langle v,L^*e\rangle$, we have that
\[
\|\varphi\|=\|L^*e\|=\|L\| \asymp\Delta.
\]

Let
\[
U:=\ker\varphi.
\]
Thus, since $S\ne T$, we have $\varphi\ne0$, and hence $U$ is a hyperplane in
$S$. We now show that vectors in $S$ whose projective classes stay away from
$\mathbb PU$ must also stay quantitatively away from $T$. Let
\[w:=\frac{L^*e}{\|L^*e\|}\in S.\]
Then $w$ is a unit vector normal to $U$ inside $S$, and for every unit
$v\in S$,
\[|\varphi(v)|=\|\varphi\|\,|\langle v,w\rangle|.\]
If $\dist([v],\mathbb PU)\ge\rho$, then the angle between $v$ and the hyperplane $U$ is bounded away from zero.
Consequently, there exists $c_0(\rho)>0$, depending only on $\rho$, such that
$|\langle v,w\rangle|\ge c_0(\rho)$. Therefore
\[|\varphi(v)|\ge c_0(\rho)\|\varphi\|\ge c_1(\rho)\Delta.\]
On the other hand,
\[|\varphi(v)|=|\langle P_{T^\perp}v,e\rangle|\le\|P_{T^\perp}v\|=\dist(v,T).\]
Hence
\[\dist(v,T)\ge c_1(\rho)\Delta.\]
Consequently, since for unit vectors the distance to $T$ is uniformly comparable to the projective distance from $[v]$ to $\mathbb PT$, after changing the constant, we obtain
\[\dist([v],\mathbb PT)\ge c(\rho)\Delta,\]
which proves \eqref{eq:principal-separation}.
\end{proof}

We now state our main invariance result of this section.
\begin{proposition}\label{prop:theta-s-arrangement}
After lifting the one-sided maximal arrangement to $\widehat\Sigma$ as in \eqref{eq:lift-arrangement}, it is
exactly stable-holonomy invariant.  Moreover, after a local matching of its
finitely many components,
\begin{equation}\label{eq:ambient-holder}
 d_{\Gr}
 \bigl(\hat W_j(\hat y),
       \Lambda^qH^u_{\hat x,\hat y}\hat W_i(\hat x)\bigr)
 \le C d_\theta(\hat x,\hat y)^{\theta_s}
\end{equation}
for every local unstable pair $\hat x,\hat y$.
\end{proposition}

\begin{proof}
We start observing that the composition rule for
stable holonomies implies immediately that the lifted arrangement constructed in \eqref{eq:lift-arrangement} is exactly stable-holonomy invariant.  Its continuity follows from the continuity of the one-sided arrangement and of the stable holonomies.

In order to prove \eqref{eq:ambient-holder}, we start comparing the corresponding conditional measures. We retain all the notation from Section \ref{sec:arrangement}. So recall that $\hat m^{(q)}_{\hat z}$ denote a $u$-holonomy-invariant fiber disintegration of the selected two-sided Grassmannian $u$-state $\hat m^{(q)}$ from \eqref{eq:q-marginal-two-sided} and $m$ is the one-sided Grassmannian measure obtained by applying the stable reduction to $\hat m^{(q)}$ whose fiber conditionals are denoted by $m_x$.  By the stable-reduction
formula,
\[
 m_x=\int_{P^{-1}(x)}
 (H^s_{\hat z,\varphi^u(\hat z)})_*\hat m^{(q)}_{\hat z}
 \,d\hat\mu_x^s(\hat z).
\]
For $\hat x\in\widehat\Sigma$ define the stable lift of this one-sided
conditional by
\[
 \widetilde m_{\hat x}
 :=(H^s_{\varphi^u(\hat x),\hat x})_*m_{P\hat x}.
\]
Since $\varphi^u$ is constant on stable fibers, the composition rule of the holonomies gives the more useful formula
\begin{equation}\label{eq:lifted-conditional-formula}
 \widetilde m_{\hat x}
 =\int_{P^{-1}(P\hat x)}
 (H^s_{\hat z,\hat x})_*\hat m^{(q)}_{\hat z}
 \,d\hat\mu_{P\hat x}^s(\hat z).
\end{equation}
In particular, $\widetilde m_{\hat x}$ gives full mass to
$\bigcup_iX_i(\hat x)$, where
$X_i(\hat x):=\iota_q^{-1}(\mathbb P\hat W_i(\hat x))$.

We first establish the quantitative transport estimate
\begin{equation}\label{eq:W1}
 W_1\bigl(\widetilde m_{\hat y},
          (H^u_{\hat x,\hat y})_*\widetilde m_{\hat x}\bigr)
 \le C d_\theta(\hat x,\hat y)^{\theta_s}
\end{equation}
for every local unstable pair $\hat x,\hat y$. Put $x=P\hat x$ and
$y=P\hat y$. Since $\hat x$ and $\hat y$ have the same past, we have
$x_0=y_0$. For a one-step Markov measure, the conditional distribution of the
negative coordinates given the nonnegative coordinates depends only on the
symbol at time $0$. Hence the conditional distributions of the past over the
futures $x$ and $y$ coincide. Thus, we may identify $P^{-1}(x)$ and
$P^{-1}(y)$ by keeping the past fixed and replacing the future $x$ by $y$.
More precisely, if
\[
 \hat z=(\ldots,z_{-1},x_0,x_1,\ldots)\in P^{-1}(x),
\]
we set
\[
 \hat z'=(\ldots,z_{-1},y_0,y_1,\ldots)\in P^{-1}(y)
\]
and define $\Phi_{x,y}:P^{-1}(x)\to P^{-1}(y)$ by
$\Phi_{x,y}(\hat z)=\hat z'$. Then $\hat z$ and $\hat z'$ have the same past,
so $\hat z'\in W^u_{\rm loc}(\hat z)$, and
\[
 (\Phi_{x,y})_*\hat\mu_x^s=\hat\mu_y^s.
\]
Equivalently, the probability measure
\[
 \kappa_{x,y}:=(\Id,\Phi_{x,y})_*\hat\mu_x^s
\]
on $P^{-1}(x)\times P^{-1}(y)$ is a coupling of
$\hat\mu_x^s$ and $\hat\mu_y^s$, supported on pairs
$(\hat z,\hat z')$ with the same past.

For $\kappa_{x,y}$-almost every pair $(\hat z,\hat z')$, the $u$-state
property of $\hat m^{(q)}$ gives
\[
 (H^u_{\hat z,\hat z'})_*\hat m^{(q)}_{\hat z}
 =\hat m^{(q)}_{\hat z'}.
\]
Recall from \eqref{eq:lifted-conditional-formula} that the contribution indexed
by $\hat z$ to $\widetilde m_{\hat x}$ is obtained by transporting
$\hat m^{(q)}_{\hat z}$ from $E_{\hat z}$ to $E_{\hat x}$ by the stable
holonomy $H^s_{\hat z,\hat x}$. To send this contribution to the corresponding
one indexed by $\hat z'$ in $\widetilde m_{\hat y}$, we first undo this stable
transport, then move from $\hat z$ to $\hat z'$ by unstable holonomy, and
finally transport from $\hat z'$ to $\hat y$ by stable holonomy. Thus the
resulting fiber map is
\[
 T_{\hat z,\hat z'}
 =H^s_{\hat z',\hat y}
  H^u_{\hat z,\hat z'}
  H^s_{\hat x,\hat z}.
\]

With the ordering $(\hat x,\hat z,\hat y,\hat z')$, these four points form a
local product rectangle as defined in Section \ref{sec:holonomies}. Hence
\eqref{eq:rectangle-package} gives
\[
 \bigl\|H^u_{\hat z,\hat z'}H^s_{\hat x,\hat z}
       -H^s_{\hat y,\hat z'}H^u_{\hat x,\hat y}\bigr\|
 \le C d_\theta(\hat x,\hat y)^{\theta_s}.
\]
Multiplying on the left by $H^s_{\hat z',\hat y}$ and using the composition
rule for stable holonomies, we obtain
\begin{equation}\label{eq:T-Hu}
 \|T_{\hat z,\hat z'}-H^u_{\hat x,\hat y}\|
 \le C d_\theta(\hat x,\hat y)^{\theta_s},
\end{equation}
where the constant $C$ is independent of the coupled pair
$(\hat z,\hat z')$.

We now turn this pointwise estimate into the Wasserstein estimate
\eqref{eq:W1} by explicitly constructing a coupling. For each pair
$(\hat z,\hat z')$ in the support of $\kappa_{x,y}$, define
\[
 \Psi_{\hat z,\hat z'}:\Gr(q,E_{\hat z})
 \longrightarrow
 \Gr(q,E_{\hat y})\times\Gr(q,E_{\hat y})
\]
by
\[
 \Psi_{\hat z,\hat z'}(V_z)
 :=\left(
 H^s_{\hat z',\hat y}H^u_{\hat z,\hat z'}V_z,
 H^u_{\hat x,\hat y}H^s_{\hat z,\hat x}V_z
 \right).
\]
We define a probability measure on
$\Gr(q,E_{\hat y})\times\Gr(q,E_{\hat y})$ by
\begin{equation*}\label{eq:explicit-coupling}
 \Gamma_{\hat x,\hat y}
 :=\int
 (\Psi_{\hat z,\hat z'})_*\hat m^{(q)}_{\hat z}\,
 d\kappa_{x,y}(\hat z,\hat z').
\end{equation*}
We claim that $\Gamma_{\hat x,\hat y}$ is a coupling of
$\widetilde m_{\hat y}$ and
$(H^u_{\hat x,\hat y})_*\widetilde m_{\hat x}$.
Indeed, its second marginal is
\[
\begin{aligned}
 &\int
 \bigl(H^u_{\hat x,\hat y}H^s_{\hat z,\hat x}\bigr)_*
 \hat m^{(q)}_{\hat z}\,
 d\kappa_{x,y}(\hat z,\hat z') \\
 &\qquad =
 (H^u_{\hat x,\hat y})_*
 \int
 (H^s_{\hat z,\hat x})_*\hat m^{(q)}_{\hat z}\,
 d\hat\mu_x^s(\hat z)
 =
 (H^u_{\hat x,\hat y})_*\widetilde m_{\hat x},
\end{aligned}
\]
where we used the fact that the first marginal of $\kappa_{x,y}$ is
$\hat\mu_x^s$ and then \eqref{eq:lifted-conditional-formula}. For the
first marginal, the $u$-state identity and the fact that the second marginal
of $\kappa_{x,y}$ is $\hat\mu_y^s$ give
\[
\begin{aligned}
 &\int
 \bigl(H^s_{\hat z',\hat y}H^u_{\hat z,\hat z'}\bigr)_*
 \hat m^{(q)}_{\hat z}\,
 d\kappa_{x,y}(\hat z,\hat z') \\
 &\qquad =
 \int
 (H^s_{\hat z',\hat y})_*\hat m^{(q)}_{\hat z'}\,
 d\hat\mu_y^s(\hat z')
 =\widetilde m_{\hat y},
\end{aligned}
\]
again by \eqref{eq:lifted-conditional-formula}. Thus
$\Gamma_{\hat x,\hat y}$ has exactly the two required marginals.

To estimate the transportation cost of this coupling, set
\[
 V_x:=H^s_{\hat z,\hat x}V_z.
\]
Then
\[
 H^s_{\hat z',\hat y}H^u_{\hat z,\hat z'}V_z
 =T_{\hat z,\hat z'}V_x,
\]
so the pair produced by $\Psi_{\hat z,\hat z'}$ may be written as
\[
 \bigl(T_{\hat z,\hat z'}V_x,
       H^u_{\hat x,\hat y}V_x\bigr).
\]
Moreover, all the linear maps appearing in \eqref{eq:T-Hu}, together with
their inverses, range in a fixed compact subset of $\GL(d,\mathbb R)$. Their
induced actions on the Grassmannian are therefore uniformly Lipschitz. Thus
\eqref{eq:T-Hu} implies
\[
 d_{\Gr}\!\left(
 T_{\hat z,\hat z'}V_x,
 H^u_{\hat x,\hat y}V_x
 \right)
 \le C d_\theta(\hat x,\hat y)^{\theta_s},
\]
with the same constant for every coupled pair $(\hat z,\hat z')$ and every
$V_z$ in the corresponding Grassmannian. Hence the coupling
$\Gamma_{\hat x,\hat y}$ has average transportation cost bounded by the same
quantity. By the definition of the $1$-Wasserstein distance,
\[
\begin{aligned}
 W_1\bigl(\widetilde m_{\hat y},
          (H^u_{\hat x,\hat y})_*\widetilde m_{\hat x}\bigr)
 &\le
 \int d_{\Gr}(V,W)\,
 d\Gamma_{\hat x,\hat y}(V,W) \\
 &\le C d_\theta(\hat x,\hat y)^{\theta_s}.
\end{aligned}
\]
This proves \eqref{eq:W1} for every local unstable pair for which the chosen
$u$-state disintegration is invariant under unstable holonomies.

Finally, this holonomy invariance holds on a full-measure subset of the local
unstable relation. Because the Markov measure has full support and local
product structure, that subset is dense. The map
$\hat x\mapsto\widetilde m_{\hat x}$ is weak-star continuous and the unstable
holonomies depend continuously on their endpoints. Since the Grassmannian is
compact, the $1$-Wasserstein distance is continuous under weak-star
convergence. Approximating an arbitrary local unstable pair by pairs in the
full-measure set and passing to the limit therefore extends \eqref{eq:W1} to
every local unstable pair.

We next turn the measure estimate into an estimate for the maximal linear
sections. We begin by observing that the minimality of $d_0$ implies the following uniform nonconcentration
property: for every $\varepsilon>0$ there is $\rho(\varepsilon)>0$ such that,
for every $\hat x$, every $i$, and every proper linear subspace
$U<\hat W_i(\hat x)$,
\begin{equation}\label{eq:nonconc}
 \widetilde m_{\hat x}
 \bigl(X_i(\hat x)\cap\iota_q^{-1}(N_{\rho(\varepsilon)}(\mathbb PU))\bigr)
 <\varepsilon.
\end{equation}
Indeed, otherwise there would be $\varepsilon_0>0$, points $\hat x_n$,
components $i_n$, proper subspaces
$U_n<\hat W_{i_n}(\hat x_n)$, and radii $\rho_n\downarrow0$ for which the
left-hand side is at least $\varepsilon_0$. Then, passing to a subsequence, we may assume that the
dimensions of the $U_n$ are constant and
\[
 \hat W_{i_n}(\hat x_n)\to W \; \text{ and }\; U_n\to U<W.
\]
Thus, weak-star continuity and the Portmanteau theorem would then give positive mass to $X(U)$, contradicting the minimality of $d_0$.

In what follows, we will also use the fact that distinct maximal components $\hat W_i(\hat x)$ are  uniformly separated. Indeed, if no positive separation existed, there would be $\hat x_n$ and distinct indices $i_n\ne j_n$ with both maximal planes converging to the same $d_0$-plane $W$.
The corresponding sections have zero-mass intersection and each has mass
$\gamma_0$, so their union has mass $2\gamma_0$. Since this union is eventually
contained in every prescribed neighborhood of $X(W)$, weak-star continuity
would imply $\widetilde m_{\hat x}(X(W))\ge2\gamma_0$, contradicting the definition of
$\gamma_0$. Thus there is $\delta_*>0$ such that distinct components in the
same fiber are at least $\delta_*$ apart. The same remains true, with a smaller
constant, after applying a local holonomy sufficiently close to the identity.

For $\hat x,\hat y$ sufficiently close, continuity of the lifted arrangement,
continuity of the unstable holonomies, and the uniform separation established
above give a unique matching between the components
$\hat W_j(\hat y)$ and the transported components
$\Lambda^qH^u_{\hat x,\hat y}\hat W_i(\hat x)$. More precisely, after
shrinking the neighborhood if necessary, each component at $\hat y$ lies close
to exactly one transported component, while all the remaining transported
components stay a definite distance away. Fix one matched pair indexed by $j$ and $i$ and define
\[
 S=\hat W_j(\hat y),\quad
 T=\Lambda^qH^u_{\hat x,\hat y}\hat W_i(\hat x)\; \text{ and }\;
 \Delta=d_{\Gr}(S,T).
\]
We may assume $\Delta>0$, since otherwise the desired estimate is immediate.
Note that $\Delta$ is small when $\hat x$ and $\hat y$ are close. Moreover, there is a constant $\delta_1>0$, independent of the local pair, such that
\[
 d_{\Gr}\!\left(
 S,\Lambda^qH^u_{\hat x,\hat y}\hat W_h(\hat x)
 \right)
 \ge \delta_1
 \; \text{ for every }h\ne i.
\]
Indeed, this follows from the uniform separation of distinct components and
from the fact that a local unstable holonomy is close to the identity when its
endpoints are close.

Choose $\varepsilon=\frac{\gamma_0}{4{N_0}}$
and let $\rho=\rho(\varepsilon)>0$ be given by
\eqref{eq:nonconc}. For each transported component
\[
 T_h:=\Lambda^qH^u_{\hat x,\hat y}\hat W_h(\hat x)
 \; \text{ with } h=1,\ldots,{N_0},
\]
we have that  $S\neq T_h$ for every $h$ because $\Delta>0$ and we may apply \eqref{eq:principal-separation} to the pair $S$ and $T_h$. Thus, we obtain a
hyperplane $U_h<S$ such that every projective direction in $\mathbb P S$ that
stays at least $\rho$ away from $\mathbb P U_h$ is separated from
$\mathbb P T_h$ by at least
\[
 c(\rho)d_{\Gr}(S,T_h).
\]
For the matched component $h=i$, this lower bound is
$c(\rho)\Delta$. For every $h\ne i$, the preceding uniform separation gives
\[
 c(\rho)d_{\Gr}(S,T_h)\ge c(\rho)\delta_1.
\]
After shrinking the neighborhood once more, we may assume
$\Delta\le\delta_1$. Hence, for every $h$, the lower bound above is at least
$c(\rho)\Delta$.

We now remove from $X_j(\hat y)$ the small neighborhoods of the hyperplane
sections corresponding to the $U_h$. Define
\[
 E:=X_j(\hat y)\setminus
 \bigcup_{h=1}^{{N_0}}
 \iota_q^{-1}\!\left(N_\rho(\mathbb P U_h)\right).
\]
By \eqref{eq:nonconc}, each removed set has
$\widetilde m_{\hat y}$-mass smaller than $\varepsilon$. Since the maximal
section $X_j(\hat y)$ has mass $\gamma_0$, the union bound gives
\[
 \widetilde m_{\hat y}(E)
 \ge \gamma_0-{N_0}\varepsilon
 =\frac{3\gamma_0}{4}.
\]
In particular, after setting, for instance, $\delta_0:=\frac{\gamma_0}{2}$, we have
\[
 \widetilde m_{\hat y}(E)\ge\delta_0.
\]

On the other hand, by the choice of the hyperplanes $U_h$, every point of
$E$ is separated from every transported component. Identifying the
Grassmannian with its Pl\"ucker image, we obtain a constant $c_2>0$,
independent of the local pair, such that
\[
 \dist\!\left(
 E,
 \bigcup_h\mathbb P
 \bigl(\Lambda^qH^u_{\hat x,\hat y}\hat W_h(\hat x)\bigr)
 \right)
 \ge c_2\Delta.
\]
For the matched component this is exactly the conclusion of
\eqref{eq:principal-separation} and for the nonmatched components we use their
fixed separation from $S$ together with $\Delta\le\delta_1$.

It remains to convert this geometric separation into a lower bound for the
Wasserstein distance. Set
\[
 \nu:=(H^u_{\hat x,\hat y})_*\widetilde m_{\hat x}.
\]
The measure $\nu$ gives full mass to the transported arrangement. Let $\pi$ be
an arbitrary transport plan between $\widetilde m_{\hat y}$ and $\nu$. Since
the second marginal of $\pi$ is supported on the transported arrangement,
all the mass of $E$ must be paired with points at distance at least
$c_2\Delta$. More explicitly, if $\mathcal A$ denotes the transported
arrangement, then $\nu(\mathcal A)=1$, and therefore
\[
 \pi(E\times\mathcal A)=\widetilde m_{\hat y}(E)\ge\delta_0.
\]
Consequently,
\[
 \int d_{\Gr}(V,W)\,d\pi(V,W)
 \ge c_2\Delta\,\pi(E\times\mathcal A)
 \ge c_2\delta_0\Delta.
\]
Since this holds for every transport plan $\pi$, taking the infimum over all
such plans yields
\[
 W_1\bigl(\widetilde m_{\hat y},
          (H^u_{\hat x,\hat y})_*\widetilde m_{\hat x}\bigr)
 \ge c_2\delta_0\Delta.
\]
Combining this lower bound with \eqref{eq:W1}, we obtain
\[
 \Delta
 \le C d_\theta(\hat x,\hat y)^{\theta_s},
\]
which is precisely \eqref{eq:ambient-holder} for sufficiently close local unstable pairs. For the remaining local unstable pairs, the distance $d_\theta(\hat x,\hat y)$ is bounded below by the chosen closeness threshold, while the Grassmannian has finite diameter. Thus, after enlarging $C$, the same estimate follows trivially. This completes the proof.
\end{proof}

\section{From approximate to exact unstable invariance}\label{sec:strong-rigidity}

In this section, we use the strong bunching assumption and upgrade the previously obtained approximate unstable invariance of the finite arrangement into exact unstable-holonomy invariance. We start with an auxiliary lemma whose proof is independent of the cocycle and is given in Appendix~\ref{app:orbit-lemma}.

\begin{lemma}\label{lem:orbit-lifting}
Let $G=\GL(d,\R)$ and
\[
 {\mathscr M}=\Gr(d_0,\Lambda^q\R^d).
\]
We let $G$ act on ${\mathscr M}$ by $g\cdot W=\Lambda^qg(W)$.  Let
$\hat f_{\rm ext}:\widehat\Sigma_{\rm ext}\to\widehat\Sigma_{\rm ext}$
be a transitive finite extension of the two-sided subshift of finite type and let
$W:\widehat\Sigma_{\rm ext}\to{\mathscr M}$ be continuous with
\[
 W(\hat f_{\rm ext}\hat z)
   =\Lambda^q\hat A_{\rm ext}(\hat z)W(\hat z),
\]
where $\hat A_{\rm ext}$ is the pullback of $\hat A$ (recall Section \ref{sec:finite extensions}).  For any fully
supported ergodic Markov measure $\hat\mu_{\rm ext}$ on
$\widehat\Sigma_{\rm ext}$, there is a $G$-orbit
${\mathscr O}\subset{\mathscr M}$ such that
\[
 W(\hat z)\in{\mathscr O}
 \quad\text{for }\hat\mu_{\rm ext}\text{-a.e. }\hat z.
\]
Moreover there exist compact sets $K_m\subset{\mathscr O}$ with
$\bigcup_mK_m={\mathscr O}$ and constants $\rho_m,C_m>0$ with the following
relative lifting property: whenever $V\in K_m$ and $V_1,V_2\in{\mathscr O}$
satisfy $d_{\Gr}(V,V_i)<\rho_m$, there exists $g\in G$ such that
\begin{equation}\label{eq:relative-lift}
 gV_1=V_2 \; \text{ and }\; \norm{g-\Id}\le C_m d_{\Gr}(V_1,V_2).
\end{equation}
\end{lemma}

The main consequence of Lemma \ref{lem:orbit-lifting} which we are going to use below is that if two nearby subspaces belong to the same $\GL(d,\R)$-orbit, then there is an element of $\GL(d,\R)$, close to the identity, that sends one subspace to the other. This will help us turn the approximate unstable invariance into full invariance.

\begin{proposition}
\label{prop:strong-rigidity}
For every local unstable pair $\hat x,\hat y\in\widehat\Sigma$, the lifted
maximal arrangement satisfies
\begin{equation}\label{eq:HuW}
 \Lambda^qH^u_{\hat x,\hat y}
 \{\hat W_1(\hat x),\ldots,\hat W_{N_0}(\hat x)\}
 =
 \{\hat W_1(\hat y),\ldots,\hat W_{N_0}(\hat y)\}.
\end{equation}
Together with the exact stable invariance from
Proposition~\ref{prop:theta-s-arrangement}, the lifted arrangement is therefore
invariant under $\hat A$ and both canonical holonomies.
\end{proposition}

\begin{proof}
We begin by introducing a finite extension whose purpose is to replace the
unordered family of components by a single-valued object. Write
\[
 \mathcal W(\hat x):=\{\hat W_1(\hat x),\ldots,\hat W_{N_0}(\hat x)\}
\]
for the unordered arrangement and define intrinsically
\[
 \widehat\Sigma_{\rm ext}
 :=\{(\hat x,W):\hat x\in\widehat\Sigma,\ W\in\mathcal W(\hat x)\}.
\]
By \eqref{eq:lift-arrangement-invariance}, the map
\[
 \hat f_{\rm ext}(\hat x,W)
 :=\bigl(\hat f\hat x,\Lambda^q\hat A(\hat x)W\bigr)
\]
is a well-defined finite extension of $\hat f$. Thus a point of
$\widehat\Sigma_{\rm ext}$ records a base point and one chosen member of the
lifted arrangement, without requiring a global labeling. On a sufficiently
small cylinder, uniform separation and continuity give a continuous labeling of
the components. The permutation of these local labels induced by $\hat A$
is locally constant. Thus, after passing to a sufficiently high block presentation, it depends only on the current symbol. Hence the intrinsic cover above is conjugate to a finite permutation subshift of finite type of the type described in Section \ref{sec:finite extensions}. Consequently, the two-sided extension is a disjoint union of finitely many transitive components. Moreover, every admissible base path has a unique lift from each chosen point in its initial finite fiber. Since the base subshift is mixing and the extension has finite fibers, each transitive
component of $\widehat\Sigma_{\rm ext}$ projects onto the whole base $\widehat\Sigma$. Equivalently, every base point has a lift in every transitive component.

In the next result we show that two lifted points on the finite extension are on the same local unstable set exactly when their selected components are the ones matched along the corresponding unstable leaf by Proposition \ref{prop:theta-s-arrangement}.

\begin{lemma}\label{lem:finite-cover-unstable}
Let us work in a high-block presentation for which the finite permutation carried by
$\hat A$ depends only on the current symbol. Two points
$\hat z=(\hat x,W)$ and $\hat z'=(\hat y,W')$ of
$\widehat\Sigma_{\rm ext}$ belong to the same local unstable set of the finite
subshift of finite type cover if and only if $\hat y\in W^u_{\rm loc}(\hat x)$ and
$W'$ is the unique component over $\hat y$ matched to $W$ by the local matching
furnished by Proposition~\ref{prop:theta-s-arrangement}. Moreover, for every
$n\ge0$,
\[ \hat f_{\rm ext}^{-n}\hat z' \in W^u_{\rm loc}(\hat f_{\rm ext}^{-n}\hat z),
\]
and the selected components at time $-n$ are the corresponding matched pair.
\end{lemma}

\begin{proof}
In the high-block presentation, the finite extension is a permutation extension:
over each base symbol there is a finite set of labels, and the transition from
one time to the next is given by a permutation depending only on the current
base symbol. Thus, once the base bi-infinite sequence and the label at time $0$
are fixed, all labels at times $i\in\mathbb Z$ are uniquely determined.

Recall that two points of the two-sided symbolic cover belong to the same local
unstable set precisely when their coordinates agree for all times $i\le0$.
Therefore, two lifted points
\[ \hat z=(\hat x,W)\; \text{ and }\; \hat z'=(\hat y,W')\]
belong to the same local unstable set if and only if their base points
$\hat x$ and $\hat y$ have the same nonpositive coordinates and their labels determine
the same backward itinerary. In particular, once $W$ is fixed, there is a unique
choice of the component $W'$ over $\hat y$ for which the two lifted points have
the same nonpositive coordinates.

We now identify this symbolic choice with the geometric matching from
Proposition~\ref{prop:theta-s-arrangement}. Restrict first to a sufficiently small
local unstable neighborhood in which the components of the arrangement admit
continuous labels and remain uniformly separated. For a fixed component $W$ over
$\hat x$, Proposition~\ref{prop:theta-s-arrangement} gives a unique component
$W_{\rm geo}(\hat y)$ over $\hat y$ which is matched to $W$ along the unstable
direction. On the other hand, the symbolic construction above gives a component
$W_{\rm sym}(\hat y)$ over $\hat y$, determined by the common backward label
itinerary.

Both $W_{\rm geo}(\hat y)$ and $W_{\rm sym}(\hat y)$ depend continuously on $\hat y$ in this neighborhood. When $\hat y=\hat x$, both coincide with the prescribed component $W$. Since distinct components of the arrangement are uniformly separated, after shrinking the neighborhood if necessary the two continuous branches cannot switch from one component to another. Hence
\[
 W_{\rm sym}(\hat y)=W_{\rm geo}(\hat y).
\]
Thus the symbolic unstable relation in the finite cover is exactly the relation obtained by matching components of the arrangement along unstable leaves.

It remains to consider an arbitrary local unstable pair. If
$\hat y\in W^u_{\rm loc}(\hat x)$, then backward iteration contracts the unstable distance, so for all sufficiently large $n$ the pair $\hat f_{\rm ext}^{-n}\hat z$ and $\hat f_{\rm ext}^{-n}\hat z'$
lies in the small neighborhood considered above. Hence the selected components at time $-n$ form the geometrically matched pair. Since backward iteration preserves the common nonpositive symbolic itinerary, this matching is preserved at every backward time. Consequently, for every $n\ge0$,
\[ \hat f_{\rm ext}^{-n}\hat z' \in W^u_{\rm loc}(\hat f_{\rm ext}^{-n}\hat z),
\]
and the components selected at time $-n$ are precisely the corresponding matched components.
\end{proof}

We now fix one transitive component of $\widehat\Sigma_{\rm ext}$ and denote by
$W(\hat z)$ the component selected at $\hat z$. In this way the previously
unordered family becomes a continuous single-valued map $W$ on the chosen
transitive component. Let $\hat\mu_{\rm ext}$ be a fully supported ergodic
Markov measure on this component, and apply Lemma~\ref{lem:orbit-lifting}. The
lemma provides a $G$-orbit ${\mathscr O}$ containing $W(\hat z)$ for
$\hat\mu_{\rm ext}$-almost every $\hat z$, together with compact sets
$K_m\subset{\mathscr O}$ whose union is ${\mathscr O}$. Choose $m$ such that
\[
 \hat\mu_{\rm ext}\{\hat z:W(\hat z)\in K_m\}>0.
\]

By ergodicity, Birkhoff's theorem applied to the indicator of this positive-measure set shows that, for $\hat\mu_{\rm ext}$-almost every $\hat x$, the negative orbit of $\hat x$ visits the set infinitely many times. Fix such a point $\hat x$ and choose a sequence $n_j\to\infty$ such
that
\[
 W\bigl(\hat f_{\rm ext}^{-n_j}\hat x\bigr)\in K_m
 \;\text{ for every }j.
\]
We also choose $\hat x$ in the full-measure invariant set on which
$W(\hat f_{\rm ext}^{-n}\hat x)\in{\mathscr O}$ for every $n\ge0$.  By the local product
structure of the Markov measure, for almost every $\hat y\in W^u_{\rm loc}(\hat x)$
with respect to the unstable conditional measure,
$\hat f_{\rm ext}^{-n}\hat y$ belongs to this full-measure set for every $n\ge0$.

For such a point $\hat y$, define
\[
 \hat x_j=\hat f_{\rm ext}^{-n_j}\hat x \;\text{ and }\;
 \hat y_j=\hat f_{\rm ext}^{-n_j}\hat y.
\]
By construction,
\[
 W(\hat x_j)\in K_m \;\text{ and }\;
 W(\hat x_j),\,W(\hat y_j)\in{\mathscr O}.
\]
Moreover, since $\hat y\in W^u_{\rm loc}(\hat x)$, backward iteration contracts
the distance along the local unstable set, and therefore
\[
 d_\theta(\hat x_j,\hat y_j)\longrightarrow0
 \; \text{ as }j\to\infty.
\]
Thus, far enough in the past, the two base points are arbitrarily close, while
the corresponding selected components remain in the same $G$-orbit and the
component over $\hat x_j$ lies in the fixed compact set $K_m$. These are exactly
the conditions needed in the next step to apply the relative lifting property
of Lemma~\ref{lem:orbit-lifting} with uniform constants.
  By  \eqref{eq:leafwise-package} we have that
$H^u_{\hat x_j,\hat y_j}\to\Id$. Thus, $\Lambda^qH^u_{\hat x_j,\hat y_j}W(\hat x_j)$ is close to $W(\hat x_j)$. Moreover, Proposition~\ref{prop:theta-s-arrangement} implies that
\[
 d_{\Gr}\bigl(W(\hat y_j),
        \Lambda^qH^u_{\hat x_j,\hat y_j}W(\hat x_j)\bigr)
 \le C d_\theta(\hat x_j,\hat y_j)^{\theta_s}.
\]
Thus, for all large $j$, $W(\hat y_j)$ is also close to
$W(\hat x_j)$ and both subspaces lie in the fixed lifting neighborhood of
$W(\hat x_j)\in K_m$.  Applying \eqref{eq:relative-lift} gives
$R_j\in\GL(d,\R)$ such that
\begin{equation}\label{eq:Rj}
 W(\hat y_j)
 =\Lambda^qR_j\,\Lambda^qH^u_{\hat x_j,\hat y_j}W(\hat x_j)
 \;\text{ and }\;
 \norm{R_j-\Id}
 \le C_m d_\theta(\hat x_j,\hat y_j)^{\theta_s}.
\end{equation}

For a Pl\"ucker subspace $W<\Lambda^q\R^d$ recall that $X(W)=\iota_q^{-1}(\mathbb PW)$. Thus, since the Pl\"ucker embedding is equivariant under the action of $\GL(d,\R)$, equation \eqref{eq:Rj} gives the exact identity
\[
 X(W(\hat y_j))
 =R_jH^u_{\hat x_j,\hat y_j}X(W(\hat x_j)).
\]
Moreover, using that $R_j=\Id+O(d_\theta(\hat x_j,\hat y_j)^{\theta_s})$, the compactness of the Grassmannian gives us that
\[
 d_H\bigl(X(W(\hat y_j)),
          H^u_{\hat x_j,\hat y_j}X(W(\hat x_j))\bigr)
 \le C d_\theta(\hat x_j,\hat y_j)^{\theta_s}.
\]
Thus, far in the past, the two sets of $q$-planes are very close.

We now transport this estimate forward from time $-n_j$ to time $0$ and for this we shall use the Lipschitz estimate \eqref{eq:prelim-grass-lip} which measures how much an invertible linear map can enlarge distances in the
Grassmannian. Applying $\hat A^{n_j}(\hat y_j)$ to the preceding Hausdorff
estimate, and using both the cocycle invariance of $W$ and the equivariance of
unstable holonomies, we obtain
\begin{align*}
 d_H\bigl(X(W(\hat y)),
           H^u_{\hat x,\hat y}X(W(\hat x))\bigr)\le
 C\,\bol(\hat A^{n_j}(\hat y_j))
 d_\theta(\hat x_j,\hat y_j)^{\theta_s}.
\end{align*}
Note that the left-hand side no longer depends on $j$. We will show that the right-hand side tends to zero.

Since $\theta_s\ge\eta_0$ and $d_\theta\le1$,
\[
 d_\theta(\hat x_j,\hat y_j)^{\theta_s}
 \le d_\theta(\hat x_j,\hat y_j)^{\eta_0}.
\]
For a local unstable pair the symbolic metric satisfies
\[
 d_\theta(\hat f^{-n}\hat x,\hat f^{-n}\hat y)
 \le \theta^n d_\theta(\hat x,\hat y)\; \text{ for every }\; n\ge0.
\]
By Lemma~\ref{lem:finite-cover-unstable}, the same estimate applies to the
chosen lifts in the finite extension, so
$d_\theta(\hat x_j,\hat y_j)\le
\theta^{n_j}d_\theta(\hat x,\hat y)$.  Moreover,
\eqref{eq:eta-bunching-consequence} is still available for the reduced
cocycle, with possibly changed multiplicative constants, by Proposition \ref{prop:linear-reduction}.  Combining these two estimates gives
\[
 \bol(\hat A^{n_j}(\hat y_j))
 d_\theta(\hat x_j,\hat y_j)^{\eta_0}
 \le C\tau^{n_j}d_\theta(\hat x,\hat y)^{\eta_0}
 \longrightarrow0.
\]
Hence the fixed quantity on the left-hand side of the previous Hausdorff
estimate must be zero. Therefore
\[
 X(W(\hat y))
 =H^u_{\hat x,\hat y}X(W(\hat x))
\]
for the generic local unstable pairs under consideration.

Finally, by \eqref{eq:generated-decomp}, each subspace $W$ is spanned by the
Pl\"ucker vectors corresponding to the $q$-planes in $X(W)$. Thus equality of
the two Pl\"ucker sets implies equality of their linear spans, and therefore
\[
 \Lambda^qH^u_{\hat x,\hat y}W(\hat x)=W(\hat y)
\]
for those generic pairs.

It remains only to remove the genericity assumption. Both
$W(\hat z)$ and the unstable holonomies depend continuously on their
arguments. The fully supported Markov measure on the finite subshift of finite type extension has
local product structure, and the full-measure set used above is therefore
dense in the local unstable relation. Approximating an arbitrary local
unstable pair by generic ones and passing to the limit extends the equality to
every local unstable pair in the chosen transitive component. Hence the
selected branch is exactly invariant under unstable holonomies on that
component. Repeating the argument on each transitive component of the finite
extension and then forgetting the finite label gives \eqref{eq:HuW} for every
local unstable pair downstairs. The proof is complete.
\end{proof}

The exact stable and unstable invariance shows that every based holonomy loop
at $\hat p$ permutes the finite family
\[
 \mathcal W_{\hat p}
 =\{\hat W_1(\hat p),\ldots,\hat W_{N_0}(\hat p)\}.
\]
By Proposition~\ref{prop:linear-reduction}\textup{(iv)}, the reduction made
before choosing the bad blocks removes this finite monodromy: every based
holonomy loop fixes every member of $\mathcal W_{\hat p}$ individually.

\section{The final rank-one contradiction}\label{sec:final}
In this section we complete the proof of the main theorem. We use again the
bad orbit blocks constructed earlier, together with the invariant subspace
obtained in the previous section. After normalization, the corresponding
maps on the exterior power converge to a nonzero rank-one map. Proposition~\ref{prop:linear-reduction} shows that this limiting map must preserve a
certain invariant splitting. We then show that the rank-one form of the
limit is not compatible with this splitting, which gives the desired
contradiction.

\begin{proof}[Proof of Theorem~\ref{thm:main}]
We now use the \emph{same} bad sequence $D_j$, the same plateau boundary
$q$, and the same Pl\"ucker covector $\ell_q$ that produced the maximal
arrangement (recall the end of Section \ref{sec:arrangement}).  Recall that $p=P\hat p$.  Since the preferred stable
representative of $p$ is $\hat p$, we have
$\hat W_i(\hat p)=W_i(p)$.  Put
\[
 W_{\hat p}:=\hat W_{i_0}(\hat p)=W_{i_0}(p)
\]
where $i_0$ is again as in the end of Section \ref{sec:arrangement}.
By \eqref{eq:WinH},
\begin{equation*}\label{eq:Wp-kernel}
 0<W_{\hat p}\subset\ker\ell_q<\Lambda^qE_{\hat p}.
\end{equation*}
By Proposition~\ref{prop:linear-reduction}\textup{(iv)}, every based holonomy
loop at $\hat p$ fixes $W_{\hat p}$ individually.  This is precisely why the
reduction is made before the bad blocks are chosen: the invariant component
and the limiting singular data come from the same sequence.

We next arrange that the left fast Pl\"ucker direction of the original bad
sequence also lies in $W_{\hat p}$.  Since
$m_p(X_{i_0}(p))=\gamma_0>0$, choose
\[
 V\in\supp(m_p)\cap X_{i_0}(p).
\]
Since $V\in\supp m_p$ and, by \eqref{eq:q-marginal-commutes-reduction},
$m_p=(\pi_q)_*\bar m_p$, compactness of the flag fiber gives a flag
$\zeta\in\supp\bar m_p\subset\hat{\mathscr K}_{\hat p}$ whose
$q$-coordinate is $V$.  Equation \eqref{eq:path-support} therefore gives
holonomy paths
$G_m:E_{\hat y}\to E_{\hat p}$, of orbit lengths $L_m$, such that
\begin{equation*}\label{eq:GmFq}
 G_mF_q\longrightarrow V.
\end{equation*}

For each fixed $m$, Lemma~\ref{lem:path-perturb} gives paths
$G_{m,j}:E_{\hat y_j}\to E_{\hat p}$ with $G_{m,j}\to G_m$ as $j\to\infty$.
Since $F_{q,j}\to F_q$, a diagonal choice $m=m(j)\to\infty$ can be made
sufficiently slowly such that 
\begin{equation*}\label{eq:slow-diagonal}
 \frac{L_{m(j)}+
       \log^+\norm{G_{m(j),j}}+
       \log^+\norm{G_{m(j),j}^{-1}}}{N_j}\rightarrow0
 \; \text{ and }\;
 G_{m(j),j}F_{q,j}\rightarrow V.
\end{equation*}
Indeed, for each fixed $m$ we have $G_{m,j}F_{q,j}\to G_mF_q$, while
$L_m+\log^+\|G_{m,j}\|+\log^+\|G_{m,j}^{-1}\|=O_m(1)$ as $j\to\infty$. Hence, after division by
$N_j\to\infty$, this last quantity tends to zero.  Since $G_mF_q\to V$, we may choose
an increasing sequence $k_m$ such that for all $j\ge k_m$ both the quotient above and
$d_{\Gr}(G_{m,j}F_{q,j},V)$ are smaller than $1/m$.  Defining
$m(j)=\max\{m\le j:k_m\le j\}$ then gives $m(j)\to\infty$ and proves the two asserted
limits.
Moreover, by Lemma~\ref{lem:exact-concat}, the compositions are genuine based holonomy
loops
\[
 B_j'=G_jD_j:E_{\hat p}\to E_{\hat p},
\]
where $G_j:=G_{m(j),j}$.
The diagonal conditions above give
\begin{equation}\label{eq:left-subexp}
 \log\norm{G_j^{\pm1}}=o(N_j)
 \; \text{ and } \;
 G_jF_{q,j}\longrightarrow V.
\end{equation}

At the genuine boundary $q$ there is $\delta=a_q-a_{q+1}>0$ and an singular-value decomposition expansion (recall the proof of Lemma \ref{lem:two-boundary} where a similar construction was used)
\[
 \Lambda^qD_j
 =s_{q,j} u_{q,j}\otimes v_{q,j}^*+R_{q,j}
\; \text{ with }\; \frac{\norm{R_{q,j}}}{s_{q,j}}=e^{-\delta N_j+o(N_j)},
\]
where $s_{q,j}=\sigma_1(D_j)\ldots\sigma_q(D_j)$, $[u_{q,j}]$ represents $F_{q,j}$ and $v_{q,j}^*$ is a covector of $\Lambda^qD_j$ such that $\ker v_{q,j}^*$ defines, under the Pl\"ucker embedding, the hyperplane section of $q$-planes meeting $S_{q,j}$ so that $v_{q,j}^*\to\ell_q$ by the normalization
fixed above (recall \eqref{eq:vqj converges to lq}).  Put $L_{q,j}=\Lambda^qG_j$.  Then
\[
 \Lambda^qB_j'
 =s_{q,j}(L_{q,j}u_{q,j})\otimes v_{q,j}^*+L_{q,j}R_{q,j}.
\]
By \eqref{eq:left-subexp},
\[
 \frac{\norm{L_{q,j}R_{q,j}}}{s_{q,j}\norm{L_{q,j}u_{q,j}}}
 \le
 \bol(L_{q,j})\frac{\norm{R_{q,j}}}{s_{q,j}}
 \le
 \bol(G_j)^q e^{-\delta N_j+o(N_j)}
 \longrightarrow0.
\]
Moreover $[L_{q,j}u_{q,j}]\to\iota_q(V)$.  Hence, after fixing unit representatives and
passing to a subsequence,
\begin{equation}\label{eq:rankone}
 \frac{\Lambda^qB_j'}{\norm{\Lambda^qB_j'}}
 \longrightarrow u\otimes\ell_q,
\end{equation}
where $u$ is a nonzero Pl\"ucker vector with $[u]=\iota_q(V)$.  Therefore
\begin{equation}\label{eq:uW}
 0\ne u\in W_{\hat p} \; \text{ and } \;\ell_q|_{W_{\hat p}}=0.
\end{equation}

Now, every $B_j'$ is a based holonomy loop at $\hat p$, so
\[
 \Lambda^qB_j'(W_{\hat p})=W_{\hat p}
\]
by Proposition~\ref{prop:linear-reduction}\textup{(iv)}.  By
Proposition~\ref{prop:linear-reduction}\textup{(v)}, there is a complementary
subspace $W_{\hat p}'$ invariant under every based holonomy loop such that
\[
 \Lambda^qE_{\hat p}=W_{\hat p}\oplus W_{\hat p}'.
\]
Every operator on the left-hand side of \eqref{eq:rankone} is block diagonal
for this decomposition, so its limit $u\otimes\ell_q$ is block diagonal as
well.  But \eqref{eq:uW} gives us
\[
 (u\otimes\ell_q)(W_{\hat p})=0 \; \text{ and }\;
 (u\otimes\ell_q)(W_{\hat p}')\subset W_{\hat p}.
\]
Block diagonality forces the second image to lie in $W_{\hat p}'$ as well.
Hence it must be zero.  Thus $u\otimes\ell_q=0$, contradicting
\eqref{eq:rankone}.  This contradiction proves Theorem~\ref{thm:main}.
\end{proof}

\begin{remark}\label{rem:role-strong-bunching}
We briefly indicate the precise role of the strong bunching hypothesis in out proof.  Up to the construction of the finite invariant Pl\"ucker arrangement and its transverse regularity, the argument uses only fiber bunching and the corresponding existence and regularity of the canonical stable and unstable holonomies, as expressed in \eqref{eq:leafwise-package} and \eqref{eq:rectangle-package}.  Strong bunching is used only in the final step, where we have to upgrade the transverse H\"older regularity of the arrangement to exact invariance under unstable holonomies. More precisely, the strong bunching hypothesis is used to ensure that the transverse exponent $\theta_s$ can be chosen compatibly with the strong bunching exponent $\eta_0$, namely
\[
\eta_0\le\theta_s\; \text{ and }\;
\bol(\hat A^n(\hat z))\theta^{\eta_0n}\le C_B\tau^n
\quad \forall n\in \mathbb N \text{ and }\hat z\in\widehat\Sigma,
\]
for some $C_B>0$ and $0<\tau<1$. Since $0<\theta<1$ and $\eta_0\le\theta_s$, this gives
\begin{equation}\label{eq:rem-est}
\bol\bigl(\hat A^n(\hat f^{-n}\hat y)\bigr)
\theta^{n\theta_s}
\longrightarrow0.
\end{equation}
Thus \eqref{eq:rem-est} is the precise quantitative estimate from strong bunching that is used in the final unstable-holonomy argument.
In particular, the proof does not use the numerical exponent $\alpha/3$ in the strong bunching assumption directly. Rather, it uses the existence of an exponent $\eta_0\le\theta_s$ for which the preceding exponential estimate holds. Consequently, any weaker hypothesis that directly implies \eqref{eq:rem-est} (and is preserved under the reductions made in the proof) would be sufficient for the argument to go through.
\end{remark}

\begin{remark}
The reduction in Proposition~\ref{prop:linear-reduction} is deliberately kept outside the dynamical argument, but its timing matters: it must be performed before the bad blocks are chosen.  The final contradiction uses the same plateau boundary, the same
limiting right singular covector, and the same invariant component.  Reducing
first guarantees the two structural properties used at the end: finite
monodromy is absent and common invariant subspaces have common invariant
complements.
\end{remark}

\appendix
\section{The linear-algebraic reduction}\label{app:reduction}

This appendix contains the proof of Proposition~\ref{prop:linear-reduction}.
The goal is to justify the algebraic reduction used in the main proof without
putting these details in the main body of the paper where we have focused on the dynamical part of the argument. We aim to provide comprehensive details concerning the algebraic objects and constructions involved in the proof once we have a dynamical systems audience in mind.

The reduction used in this appendix combines ideas concerning Zariski closures and semisimplification of linear cocycles and representations. For instance,  the passage from the Zariski closure to a reductive Levi factor is closely related to the semisimplification procedure used by Kassel--Potrie \cite[Section 2.9]{KasselPotrie} in the setting of eigenvalue gaps for locally constant linear cocycles.

\subsection{Zariski closure and connected components}

A \emph{real algebraic subgroup} of $\GL(V)$ is a subgroup which is the common zero
set, inside $\GL(V)$, of a family of real polynomial functions in the matrix
entries.  If $\Gamma<\GL(V)$ is any subgroup, its \emph{real Zariski closure}
$\overline\Gamma^{\,Z}$ is the smallest real algebraic subgroup containing
$\Gamma$.  Equivalently, every polynomial identity which holds on $\Gamma$ holds
on $\overline\Gamma^{\,Z}$.

Every real algebraic group $H$ has finitely many connected components in the
Zariski topology.  We denote by $H^\circ$ the identity component.  It is a normal
algebraic subgroup of finite index.  Thus the component group $H/H^\circ$ is
finite.  We shall exploit this finiteness by recording the component in a finite
symbolic extension.

\subsection{Unipotent radicals, reductive groups and Levi decomposition}

A matrix $u$ is \emph{unipotent} if all its eigenvalues are equal to $1$.  The
\emph{unipotent radical} $R_u(H)$ of a real algebraic group $H$ is the largest
connected normal algebraic subgroup consisting of unipotent elements.  The group
$H$ is called \emph{reductive} if $R_u(H)$ is trivial.  In characteristic zero,
this is equivalent to the representation-theoretic statement that every
finite-dimensional algebraic representation of $H$ is completely reducible.
In particular, if $H$ is reductive and $W$ is an $H$-invariant subspace of a
finite-dimensional algebraic representation $E$, then there is an $H$-invariant
complement $W'$ with $E=W\oplus W'$.

We use the following standard form of the Levi decomposition theorem (see \cite[Chapter~6, Section~4]{OnishchikVinberg}). See also
~\cite[Theorem~7.1]{Mostow} and \cite[Section 2.9]{KasselPotrie}. 

\begin{theorem}
\label{thm:levi-background}
Let $H<\GL(V)$ be a real algebraic group, not necessarily connected, and let
$U=R_u(H)$ be its unipotent radical.  There is a reductive real algebraic
subgroup $L<H$ for which the multiplication map
\[
 L\ltimes U\longrightarrow H,
 \qquad (\ell,u)\longmapsto \ell u,
\]
is an isomorphism of real algebraic groups.  In particular every $h\in H$ has
a unique factorization $h=\ell u$, and
\[
 \pi_L:H\longrightarrow L,\qquad \pi_L(\ell u)=\ell,
\]
is an algebraic group homomorphism.  Moreover every finite-dimensional
algebraic representation of $L$ is completely reducible.
\end{theorem}

For the reader unfamiliar with this terminology, the useful picture is the
following.  The unipotent radical is the part of the group responsible for
upper-triangular ``shearing''.  Passing from $H$ to its Levi factor $L$ discards
those shears but keeps the diagonal blocks, and therefore keeps all eigenvalues.
For example, for matrices
\[
 h=\begin{pmatrix} a&t\\0&b\end{pmatrix},
 \qquad a,b\ne0,
\]
the shear parameter $t$ is unipotent information, while the Levi projection is
$\operatorname{diag}(a,b)$.  Conjugating by
$D_s=\operatorname{diag}(e^{-s},1)$ sends the upper-right entry to $e^{-s}t$,
so $D_shD_s^{-1}$ converges to the diagonal matrix as $s\to\infty$.  The
semisimplification below is the block version of this elementary operation.

\subsection{Why finite monodromy disappears}

We shall repeatedly use one elementary consequence of connectedness.  If a
connected real algebraic group $H$ acts algebraically on a finite set, then the
action is trivial.  Indeed, the stabilizer of a point, that is, the set of all group elements that leave the point fixed under the group action, is an algebraic subgroup of
finite index, and a connected algebraic group has no proper algebraic subgroup of
finite index.  Equivalently, an algebraic map from a connected algebraic variety
to a finite discrete set is constant.

\subsection{Reduction of a counterexample}\label{sec:early-ss}

Let $\mathscr G_{\hat A}$ be the linear groupoid whose objects are the fibers $E_{\hat x}$ and which is generated by the cocycle arrows $\hat A(\hat x)^{\pm1}$ and by the canonical local stable and unstable holonomies and their inverses. Here an \emph{arrow} simply means an invertible linear map between two fibers, regarded as a morphism in the groupoid. Concretely, this only means that we allow finite compositions of these linear maps whenever the target fiber of one map is the source fiber of the next. Inverses are allowed as well. Thus every holonomy path map is an arrow of $\mathscr G_{\hat A}$.  After replacing $\hat f$ by an iterate we
fix a point $\hat p\in\Fix(\hat f)$ and put
\[
 \Gamma_{\hat p}:=\mathscr G_{\hat A}(\hat p,\hat p)<\GL(E_{\hat p}).
\]
Choose, for each one-cylinder $C_a$, a reference point $\hat x_a\in C_a$ and a
fixed groupoid arrow $P_a:E_{\hat p}\to E_{\hat x_a}$.  Such an arrow exists because the
mixing subshift of finite type admits an admissible orbit word from the symbol of $\hat p$ to $a$. Then, adding the local stable and unstable legs gives a groupoid path from $\hat p$ to $\hat x_a$.  If $\hat x\in C_a$, set $\hat z=[\hat x_a,\hat x]$ and
\begin{equation}\label{eq:path-gauge}
 \Theta_{\hat x}=H^s_{\hat z,\hat x}H^u_{\hat x_a,\hat z}P_a:E_{\hat p}\longrightarrow E_{\hat x}.
\end{equation}
By the leafwise holonomy estimate and the transverse rectangle estimate
\eqref{eq:leafwise-package}--\eqref{eq:rectangle-package}, the family
$\Theta_{\hat x}$ and its inverse are bounded and $\theta_s$-H\"older.  Indeed, if $\hat x$ and $\hat x'$ lie in the same local stable set inside
$C_a$, then $[\hat x_a,\hat x]=[\hat x_a,\hat x']$, and the composition law for $H^s$, together with
\eqref{eq:leafwise-package}, gives
$\norm{\Theta_{\hat x}-\Theta_{\hat x'}}=O(d_\theta(\hat x,\hat x')^{\theta_s})$.  If $\hat x$ and $\hat x'$ lie in the
same local unstable set, then the four points
$[\hat x_a,\hat x],[\hat x_a,\hat x'],\hat x,\hat x'$ form a local product rectangle (recall Section \ref{sec:holonomies}) and thus combining
\eqref{eq:rectangle-package} with \eqref{eq:leafwise-package} gives the same
estimate.  For arbitrary nearby points $\hat x$ and $\hat x'$ in $C_a$, insert the bracket
$[\hat x,\hat x']$ and combine the stable and unstable estimates.  Since there are only
finitely many one-cylinders, the estimate is global.  The inverse family has
the same regularity by the identity
$\Theta_{\hat x}^{-1}-\Theta_{\hat x'}^{-1}=\Theta_{\hat x}^{-1}(\Theta_{\hat x'}-\Theta_{\hat x})\Theta_{\hat x'}^{-1}$.
Using these identifications, we represent the cocycle on the fixed fiber $E_{\hat p}$ by
\begin{equation}\label{eq:gauged-cocycle}
 \hat B(\hat x)=\Theta_{\hat f\hat x}^{-1}\hat A(\hat x)\Theta_{\hat x}\in\Gamma_{\hat p}.
\end{equation}
Likewise, every canonical holonomy, when expressed with respect to these identifications of the fibers with $E_{\hat p}$, belongs to $\Gamma_{\hat p}$. In
particular, if
\[
 H=\overline{\Gamma_{\hat p}}^{\,Z}<\GL(E_{\hat p})
\]
is the real Zariski closure, then the whole range of $\hat B$ is contained in $H$.
Notice that this definition avoids any identification of the based path group with a separately generated matrix semigroup: membership in $\Gamma_{\hat p}$ follows exactly from the composition identity \eqref{eq:gauged-cocycle}.

\begin{lemma}\label{lem:ss}
There are a reductive real algebraic subgroup $L<\GL(E_{\hat p})$, a homomorphism
\[
 \pi_{\rm ss}:H\longrightarrow L,
\]
and constant matrices $D_n\in\GL(E_{\hat p})$ such that, uniformly for $h$ in compact
subsets of $H$,
\begin{equation}\label{eq:ss-limit}
 D_nhD_n^{-1}\longrightarrow\pi_{\rm ss}(h).
\end{equation}
Moreover one may choose a Euclidean norm on $E_{\hat p}$ for which
\begin{equation}\label{eq:ss-bol}
 \bol(\pi_{\rm ss}(h))\le \bol(h),
\end{equation}
and $h$ and $\pi_{\rm ss}(h)$ have the same eigenvalues, with algebraic
multiplicity.
\end{lemma}

\begin{proof}
The construction is the block version of the semisimplification used in
\cite[Section~2.9]{KasselPotrie}.  Choose, by
Theorem~\ref{thm:levi-background}, a Levi decomposition
$H=L\ltimes U$, where $U=R_u(H)$, and let $\pi_L:H\to L$ be the algebraic
projection.

We first construct an $L$-invariant filtration explicitly.  Put $F_0=\{0\}$ and let
$F_1=E_{\hat p}^{U}$ be the set of all vectors in $E_{\hat p}$ which are invariant under every element of $U$.
A connected unipotent
group acting on a nonzero finite-dimensional real vector space has a nonzero
fixed vector after complexification and hence a nonzero real fixed space. Thus
$F_1\ne0$ unless $E_{\hat p}=0$.  Since $U$ is normal in $H$, $F_1$ is
$H$-invariant.  Inductively, having defined $F_{i-1}$, let
$F_i/F_{i-1}$ be the $U$-fixed subspace of $E_{\hat p}/F_{i-1}$ and take
$F_i$ to be its inverse image.  Again normality of $U$ makes $F_i$
$H$-invariant.  The dimension increases at every nonterminal step, so after
finitely many steps
\[
 0=F_0<F_1<\cdots<F_m=E_{\hat p}.
\]
By construction, $U$ acts trivially on each quotient $F_i/F_{i-1}$.

The filtration is $L$-invariant.  Since $L$ is reductive, complete
reducibility allows us to choose $L$-invariant complements
$V_i\subset F_i$ such that
$F_i=F_{i-1}\oplus V_i$.  Thus
$E_{\hat p}=V_1\oplus\cdots\oplus V_m$.  With respect to this decomposition,
every $h\in H$ preserves the filtration and hence is block upper triangular.
If $h=\ell u$ with $\ell\in L$ and $u\in U$, then the action induced by $u$
on $F_i/F_{i-1}$ is the identity.  Therefore the $i$-th diagonal block of
$h$ is exactly the action of $\ell=\pi_L(h)$ on $V_i\simeq F_i/F_{i-1}$.
Consequently the block diagonal part of $h$ is precisely the matrix of
$\pi_L(h)$ on $\bigoplus_iV_i$.  We set
\[
 \pi_{\rm ss}:=\pi_L:H\longrightarrow L
\]
in these coordinates.  Since a block triangular matrix and its block diagonal
part have the same characteristic polynomial, $h$ and $\pi_{\rm ss}(h)$ have
the same eigenvalues with algebraic multiplicity.

Choose the Euclidean norm for which the spaces $V_i$ are mutually orthogonal.
Orthogonal projection onto the diagonal blocks then gives
\[
 \norm{\pi_{\rm ss}(h)}
   =\max_i\norm{h_{ii}}\le\norm h.
\]
The inverse of a block upper triangular matrix is block upper triangular with
diagonal blocks $h_{ii}^{-1}$, so the same argument applied to $h^{-1}$ gives
\[
 \norm{\pi_{\rm ss}(h)^{-1}}
   \le\norm{h^{-1}}.
\]
Multiplying these inequalities proves \eqref{eq:ss-bol}.  Replacing the
original Euclidean norm by this one changes the previously established norm
estimates only by uniform multiplicative constants and therefore does not
affect any exponential bunching rate.

Finally choose real numbers
$a_1<a_2<\cdots<a_m$ and, for $n\ge1$, define the constant diagonal map
\[
 D_n|_{V_i}=e^{a_i n}\Id_{V_i}.
\]
If $h=(h_{ij})$ is block upper triangular, then the $(i,j)$ block of
$D_nhD_n^{-1}$ is
$e^{(a_i-a_j)n}h_{ij}$.  For $i<j$ this tends exponentially to zero, while
for $i=j$ it is $h_{ii}$.  Thus
\[
 D_nhD_n^{-1}\longrightarrow
 \operatorname{diag}(h_{11},\ldots,h_{mm})=\pi_{\rm ss}(h).
\]
If $K\subset H$ is compact, all off-diagonal block norms $\norm{h_{ij}}$ are
bounded uniformly for $h\in K$. Hence the convergence is uniform on $K$.
This proves \eqref{eq:ss-limit} and completes the lemma.
\end{proof}

Set $\hat B_{\rm ss}=\pi_{\rm ss}\circ\hat B$.  Since $\pi_{\rm ss}$ is smooth on the
algebraic group $H$ and $\hat B(\widehat\Sigma)$ is compact, $\hat B_{\rm ss}$ is $\theta_s$-H\"older.  For
a periodic point $\hat x$,
\[
 \hat B_{\rm ss}^n(\hat x)=\pi_{\rm ss}(\hat B^n(\hat x)),
\]
so Lemma \ref{lem:ss} shows that the periodic eigenvalues, and hence the
periodic-gap hypothesis, are unchanged.

The following persistence statement is the key point in the algebraic reduction. 

\begin{lemma}\label{lem:package-persistence}
Suppose a two-sided cocycle $\hat C$ has stable and unstable holonomies satisfying
\eqref{eq:leafwise-package} and \eqref{eq:rectangle-package} with exponent
$\theta_s$.  When $d\ge3$, suppose in addition that it satisfies
\eqref{eq:eta-bunching-consequence} with $\eta_0\le\theta_s$.

\begin{enumerate}[label=(\alph*)]
\item If $\hat C_\Theta(\hat x)=\Theta_{\hat f\hat x}^{-1}\hat C(\hat x)\Theta_{\hat x}$, where $\Theta$ and
$\Theta^{-1}$ are bounded and $\theta_s$-H\"older, then $\hat C_\Theta$ is
$\theta_s$-H\"older, its holonomies are
\[
 H^{*,\Theta}_{\hat x,\hat y}=\Theta_{\hat y}^{-1}H^*_{\hat x,\hat y}\Theta_{\hat x} \; \text{ for }\; *=s,u,
\]
and the estimates \eqref{eq:leafwise-package} and
\eqref{eq:rectangle-package} continue to hold, with changed constants but the
same exponent $\theta_s$.  Moreover
\begin{equation}\label{eq:gauge-bol}
 \bol(\hat C_\Theta^n(\hat x))\le C_\Theta^4\bol(\hat C^n(\hat x))
 \; \text{ with }\;
 C_\Theta=\sup_{\hat x}\max\{\norm{\Theta_{\hat x}},\norm{\Theta_{\hat x}^{-1}}\}.
\end{equation}
Thus \eqref{eq:eta-bunching-consequence} also persists, after changing its
multiplicative constant.

\item Let $H_{\hat C}$ be an algebraic subgroup containing the range of $\hat C$ and all
of its holonomies, and let $\pi:H_{\hat C}\to L$ be a smooth algebraic homomorphism
such that $\bol(\pi(h))\le\bol(h)$.  Then $\hat C_\pi=\pi\circ\hat C$ is
$\theta_s$-H\"older, has holonomies $\pi(H^*)$, and satisfies \eqref{eq:leafwise-package} and
\eqref{eq:rectangle-package} with the same exponent, and, when $d\ge3$,
satisfies \eqref{eq:eta-bunching-consequence} with no worse exponential
constant.

\item Pullback to a finite extension, restriction to a transitive component,
and passage to an iterate preserve these properties (with the evident change
of constants and base contraction rate).
\end{enumerate}
In particular, when $d\ge3$, every two-sided cocycle produced by these operations is
$\theta_s$-H\"older and, after an iterate, $\theta_s$-fiber-bunched by
\eqref{eq:eta-bunching-consequence}.  In that case the transported holonomies
coincide with the canonical holonomies of the resulting fiber-bunched cocycle,
by the usual uniqueness of the canonical holonomy limits.
\end{lemma}

\begin{proof}
For a fiberwise change of coordinates, the formula for the transformed holonomies follows from equivariance. The leafwise estimate follows from that formula, the
$\theta_s$-H\"older regularity of $\Theta$, and boundedness of all factors.
For a local product rectangle, the two transported path maps have the same
initial and terminal fibers, and their difference transforms exactly as
\[
 H^{u,\Theta}_{\hat y,\hat y'}H^{s,\Theta}_{\hat x,\hat y}
 -H^{s,\Theta}_{\hat x',\hat y'}H^{u,\Theta}_{\hat x,\hat x'}
 =\Theta_{\hat y'}^{-1}
 \bigl(H^u_{\hat y,\hat y'}H^s_{\hat x,\hat y}
       -H^s_{\hat x',\hat y'}H^u_{\hat x,\hat x'}\bigr)\Theta_{\hat x}.
\]
This proves persistence of \eqref{eq:rectangle-package}.  The identity
\[
 \hat C_\Theta^n(\hat x)=\Theta_{\hat f^n\hat x}^{-1}\hat C^n(\hat x)\Theta_{\hat x}
\]
gives \eqref{eq:gauge-bol}.

For the homomorphism $\pi$, equivariance and the composition identities are
preserved because $\pi$ is a homomorphism.  On the compact set containing the
relevant holonomy products, $\pi$ is uniformly Lipschitz.  Thus, writing the
two sides of the rectangle as $M_1,M_2\in H_{\hat C}$,
\[
 \norm{\pi(M_1)-\pi(M_2)}\le C_\pi\norm{M_1-M_2},
\]
which preserves the exponent $\theta_s$. The leafwise estimate is obtained in
the same way near the identity.  The bolicity assertion is immediate from the
hypothesis on $\pi$.  We now verify the remaining operations.  For a finite subshift of finite type extension with natural projection
$\rho:\widehat\Sigma'\to\widehat\Sigma$, equip the extension with a distance with the same parameter $\theta$ and consider the pullback cocycle $\hat C'=\hat C\circ\rho$.  Equality of extension symbols implies equality of
their base symbols, so
$d_\theta(\rho\hat x',\rho\hat y')\le d_\theta(\hat x',\hat y')$.  Hence the
pullback holonomies
\[
 H^{*,\prime}_{\hat x',\hat y'}
 :=H^*_{\rho\hat x',\rho\hat y'}\; \text{ for }\; *=s,u,
\]
satisfy \eqref{eq:leafwise-package} and
\eqref{eq:rectangle-package} with the same exponents and no worse
multiplicative constants.  Also
$\hat C'^n(\hat z')=\hat C^n(\rho\hat z')$, so
\eqref{eq:eta-bunching-consequence} is unchanged.  Restriction to an invariant
transitive component only decreases the range of points and therefore changes
none of these estimates.

For the $r$-th iterate, write
$\hat C^{(r)}(\hat x)=\hat C^r(\hat x)$ and view it over $\hat f^r$.  The local
stable and unstable sets are unchanged, while their contraction factor becomes
$\theta^r$.  For $m\ge0$,
\[
 (\hat C^{(r)})^m(\hat x)=\hat C^{rm}(\hat x),
\]
so \eqref{eq:eta-bunching-consequence} becomes
\[
 \bol((\hat C^{(r)})^m(\hat x))(\theta^r)^{\eta_0 m}
 \le C_B(\tau^r)^m.
\]
The leafwise and rectangle estimates retain the exponent $\theta_s$ (after a
harmless change of constants if one recodes $\hat f^r$ as a one-step subshift of finite type).
Moreover the stable holonomy of the iterate is
\[
 \lim_{m\to\infty}\hat C^{rm}(\hat y)^{-1}\hat C^{rm}(\hat x)
 =H^s_{\hat x,\hat y},
\]
and similarly for the unstable holonomy.  The same limit calculation shows
that the pullback holonomies on a finite extension are the canonical
holonomies of the pullback cocycle.  For a homomorphic image, continuity of
$\pi$ gives the canonical limit $\pi(H^*)$, and for a fiberwise conjugacy the
canonical limit is the transported holonomy displayed in part~(a).
Consequently, whenever the resulting cocycle is fiber-bunched, all holonomies
used here are its canonical holonomies, not merely auxiliary equivariant
families.  Finally, if $d\ge3$, $\eta_0\le\theta_s$ and $\nu<1$ imply
$\nu^{\theta_s n}\le\nu^{\eta_0 n}$, so
\eqref{eq:eta-bunching-consequence} yields the required fiber bunching after a
large iterate.
\end{proof}

Apply Lemma \ref{lem:package-persistence} first to the fiberwise change of coordinates given by \eqref{eq:path-gauge} and then to the Levi homomorphism $\pi_{\rm ss}$. Hence $\hat B_{\rm ss}(\hat x)=\pi_{\rm ss}(\hat{B}(\hat x))$ where $\hat B$ is given by \eqref{eq:gauged-cocycle} carries the inherited stable and unstable holonomies, the same transverse exponent $\theta_s$, and, when $d\ge3$, the estimate \eqref{eq:eta-bunching-consequence}. The group of based path maps associated with $\hat B_{\rm ss}$ is $\pi_{\rm ss}(\Gamma_{\hat p})$ and has Zariski closure $L$.

Finally, if $\hat B_{\rm ss}$ were index-$k$ dominated, openness of domination and
\eqref{eq:ss-limit} would imply that $D_n\hat B D_n^{-1}$ is index-$k$ dominated for
all sufficiently large $n$, and hence that $\hat B$, and therefore $\hat A$, is dominated.
Thus a non-dominated cocycle with the periodic gap and the quantitative package semisimplifies to another cocycle with those properties.

We also remove the finite component group at this stage.

\begin{lemma}
\label{lem:connected-cover}
Starting from the counterexample above, equipped with the quantitative package of Lemma \ref{lem:package-persistence}, after finitely many
operations of the following three types,
\begin{enumerate}[label=(\roman*)]
\item bounded $\theta_s$-H\"older fiberwise change of coordinates;
\item the semisimplification of Lemma~\ref{lem:ss};
\item passage to a finite subshift of finite type extension and then to a transitive cyclic
component,
\end{enumerate}
one obtains a counterexample with the same periodic-gap hypothesis whose group of based path maps has connected reductive Zariski closure.
\end{lemma}

\begin{proof}
Apply Lemma \ref{lem:ss}. We are reduced to a counterexample for which the Zariski closure $L$ of the group of based path maps is reductive. Rename the semisimplified cocycle and its representative obtained from the fiber identifications as $\hat A$ and $\hat B$, respectively. Thus, from this point on, the values of $\hat B$ and all canonical holonomies, when expressed with respect to these identifications, lie in $L$. If $L$ is connected, there is nothing to prove. Otherwise, let $L^\circ$ denote the connected component of the identity in $L$ in the Zariski topology. In particular, $L^\circ$ is a normal subgroup of $L$ of finite index. Thus, we may consider
\[ Q=L/L^\circ\]
and let $\chi:L\to Q$ be the quotient map. Using the above fiber identifications, define
\[\beta(\hat x)=\chi(\hat B(\hat x)).\]
Since $Q$ is finite and $\hat B$ is continuous, $\beta$ is locally constant. After passing to a higher-block presentation, it defines a finite-group subshift of finite type extension
\[ \hat f_{\rm ext}(\hat x,a) =(\hat f\hat x,\beta(\hat x)a) \; \text{ for }\; a\in Q.\]

The change in the $Q$-coordinate along every lifted path map is the image under $\chi$ of its corresponding linear map. For orbit maps this follows directly from the definition. For example, if $\hat x,\hat y$ are in the same local stable set, then
\[ H^s_{\hat x,\hat y} =\lim_{n\to\infty}  \hat B^n(\hat y)^{-1}\hat B^n(\hat x).\]
Applying $\chi$ gives
\[ \chi(H^s_{\hat x,\hat y}) = \lim_{n\to\infty} \beta^n(\hat y)^{-1}\beta^n(\hat x),\]
and, since $Q$ is finite, this sequence is eventually constant. Thus two lifts $(\hat x,a)$ and $(\hat y,b)$ belong to the same local stable set precisely when
\[ b=\chi(H^s_{\hat x,\hat y})a.\]
The analogous backward formula gives the same statement for unstable holonomies. Consequently, for any lifted path obtained by composing cocycle maps and holonomies, the change in its $Q$-coordinate is exactly the image under $\chi$ of the associated linear path map. Hence a closed path based at a point of a single lifted component has trivial change in the $Q$-coordinate, and its linear part lies in
\[ \ker\chi=L^\circ.\]

The pullback cocycle on the whole finite extension is non-dominated, since domination there would imply domination on the base. Since the extension has only finitely many transitive components, at least one of them is non-dominated. Restrict to such a component. Choose a periodic point $\hat p'$ in it and pass, if necessary, to a suitable iterate and cyclic component so that the restricted base is mixing and $\hat p'$ is fixed. Let $K$ be the Zariski closure of the group of based path maps at $\hat p'$, using new fiber identifications if necessary.

Changing the base point in a transitive system of cocycle and holonomy paths conjugates the corresponding groups of based path maps by a path map joining the two base points. In the present coordinates, this conjugating map belongs to $L$. Since $L^\circ$ is normal in $L$, conjugation by such a map preserves $L^\circ$. Therefore the conclusion obtained above for closed paths in the chosen lifted component gives
\begin{equation*}\label{eq:K-in-L0}
K\subset L^\circ.
\end{equation*}
Conjugation also preserves algebraic dimension, connectedness, and reductivity. The periodic-gap hypothesis and the quantitative estimates used later survive this restriction by Lemma \ref{lem:package-persistence}.

The following dimension observation makes the termination of the procedure
explicit.

\medskip
\noindent\textbf{Claim.}
Suppose $G$ is the Zariski closure at the beginning of one stage of the
procedure.  After semisimplification let $L_G=\pi_{\rm ss}(G)$, and, if
$L_G$ is disconnected, after passage to the component cover and restriction
to a non-dominated transitive component let $K_G$ be the new based-path
Zariski closure.  Then
\[
 \dim L_G\le\dim G \; \text{ and  }\; K_G\subset L_G^\circ.
\]
Moreover, either $K_G=L_G^\circ$, in which case the new closure is connected
and reductive, or
\[
 \dim K_G<\dim L_G\le\dim G.
\]

\smallskip
\noindent Indeed, $L_G$ is the image of $G$ under the algebraic homomorphism
$\pi_{\rm ss}$, so $\dim L_G\le\dim G$.  The component-cover argument above
gives $K_G\subset L_G^\circ$.  Since
$\dim L_G^\circ=\dim L_G$, equality of dimensions
$\dim K_G=\dim L_G$ forces $K_G=L_G^\circ$. Otherwise the dimension is
strictly smaller.  This proves the claim.

Returning to the present stage, if $\dim K=\dim L$, then the claim gives
$K=L^\circ$, so $K$ is connected and reductive and we are done.  If
$\dim K<\dim L$, repeat the construction with $K$ in place of $L$, first
choosing the corresponding bounded fiberwise identifications at the new base
point.  At every stage that does not terminate with a connected reductive
closure, the claim gives a strict decrease of the nonnegative integer
``dimension of the based-path Zariski closure''.  Hence only finitely many
stages are possible, and the procedure terminates with a counterexample whose
based-path Zariski closure is connected and reductive.
\end{proof}

We now observe that the preceding lemmas prove Proposition~\ref{prop:linear-reduction}. Indeed, the chosen fiber identifications allow us to represent the two-sided cocycle and all holonomies as linear maps on one fixed fiber. Moreover,
the construction in Lemma~\ref{lem:ss} deletes the upper-triangular shear without changing
periodic eigenvalues, and Lemma~\ref{lem:package-persistence} shows that the
dynamical estimates persist.  Lemma~\ref{lem:connected-cover} removes the
finite component group.

For completeness, let $\Gamma$ be the final group of based path maps and let
$L=\overline\Gamma^{\,Z}$, which is connected and reductive. Suppose that
$\Gamma$ permutes a finite family of subspaces. The kernel $\Gamma_0$ of this
finite permutation action has finite index in $\Gamma$. If
$L_0=\overline{\Gamma_0}^{\,Z}$, then $L$ is the union of finitely many cosets of $L_0$.
Thus $L_0$ is an algebraic subgroup of finite index in $L$, and connectedness
forces $L_0=L$. Since the stabilizer of each member of the finite family is
Zariski closed and contains $\Gamma_0$, it contains $L$. Hence every
element of $\Gamma$ fixes every member individually. This proves item
\textup{(iv)}.

Finally, if a subspace $W$ in an exterior power is invariant under $\Gamma$,
then it is invariant under $L$ because its stabilizer is Zariski closed.
Complete reducibility of the exterior-power representation of the reductive
group $L$ gives an $L$-invariant, and therefore $\Gamma$-invariant,
complementary subspace. This proves item \textup{(v)} and completes the proof of
Proposition~\ref{prop:linear-reduction}.
\qed

\section{Proof of Lemma~\ref{lem:orbit-lifting}}\label{app:orbit-lemma}

This appendix proves Lemma~\ref{lem:orbit-lifting}. Its purpose is simple:
if two nearby subspaces belong to the same $\GL(d,\R)$-orbit, we need an
element of $\GL(d,\R)$, close to the identity, that sends one subspace to the
other. This is the only consequence of the algebraic orbit argument that is
used in the main proof.

The proof relies on three standard results. First, Rosenlicht's theorem is
used to show that the measurable family $W(\hat z)$ is contained almost
everywhere in a single complex orbit. Second, the fact that the real points of
this complex orbit split into finitely many real $\GL(d,\R)$-orbits follows from
Borel--Serre \cite[Corollary~6.4]{BorelSerre}. Finally, the quantitative
estimate for two nearby subspaces follows from the local section theorem for
smooth submersions (see \cite[Theorem~4.26]{LeeSmooth}). We include the details
because the main proof requires the explicit estimate
\eqref{eq:relative-lift}.

Let
\[
G=\GL(d,\R),\;
 G_{\C}=\GL(d,\C)
 \;\text{and}\;
 {\mathscr M}_{\C}=\Gr(d_0,\Lambda^q\C^d).
\]
We use Rosenlicht's theorem in its geometric-quotient form
\cite{Rosenlicht}. For an irreducible variety this is also stated explicitly
in \cite[Theorem~1.1]{BellGhiocaReichstein}.  It says that a nonempty
$G_{\C}$-variety contains a nonempty $G_{\C}$-invariant Zariski-open subset
$S$ admitting a morphism
\[
 \pi:S\longrightarrow Z
\]
whose fibers are individual $G_{\C}$-orbits.  Applying this statement by
Noetherian induction gives the finite stratification needed here.  Namely,
start with $Y_0={\mathscr M}_{\C}$.  If $Y_r\ne\varnothing$, choose such an
invariant open subset $S_{r+1}\subset Y_r$ and put
$Y_{r+1}=Y_r\setminus S_{r+1}$.  Then $Y_{r+1}$ is a proper
$G_{\C}$-invariant closed algebraic subset of $Y_r$.  Because the Zariski
topology is Noetherian, the strictly descending chain
$Y_0\supsetneq Y_1\supsetneq\cdots$ terminates after finitely many steps.
Thus
\[
 {\mathscr M}_{\C}=S_1\sqcup\cdots\sqcup S_N
\]
is a finite decomposition into $G_{\C}$-invariant locally closed sets and,
for each $j$, there is a geometric quotient
\[
 \pi_j:S_j\longrightarrow Z_j
\]
whose fibers are exactly the $G_{\C}$-orbits in $S_j$.  It is this last property, rather than merely separation of generic orbits
by rational invariants, that will be used below.

Since
\[
 W(\hat f_{\rm ext}\hat z)
 =\Lambda^q\hat A_{\rm ext}(\hat z)W(\hat z)
\]
and $\hat A_{\rm ext}(\hat z)\in G\subset G_{\C}$, the index $j$ for which
$W(\hat z)\in S_j$ is invariant under $\hat f_{\rm ext}$. By ergodicity,
there is therefore a fixed $j_0$ such that
\[
 W(\hat z)\in S_{j_0}
\]
for almost every $\hat z$.

Let $\pi_{j_0}:S_{j_0}\to Z_{j_0}$ be the corresponding map. Since
$\pi_{j_0}$ is constant on $G_{\C}$-orbits,
\[
 \pi_{j_0}\bigl(W(\hat f_{\rm ext}\hat z)\bigr)
 =
 \pi_{j_0}\bigl(W(\hat z)\bigr).
\]
Ergodicity then implies that $\pi_{j_0}(W(\hat z))$ is constant almost
everywhere. Because the fibers of $\pi_{j_0}$ are precisely the
$G_{\C}$-orbits, it follows that $W(\hat z)$ belongs almost everywhere to a
single complex $G_{\C}$-orbit, which we denote by ${\mathscr O}_{\C}$. Since
$W(\hat z)$ is real,
\[
 W(\hat z)\in{\mathscr O}_{\C}(\R)
\]
for almost every $\hat z$ where ${\mathscr O}_{\C}(\R)$ denotes the set of real points of the complex orbit ${\mathscr O}_{\C}$, that is, those subspaces in ${\mathscr O}_{\C}$ that belong to the real Grassmannian.

Choose a real point $v\in{\mathscr O}_{\C}(\R)$.  Let
$H_{v,\C}=(G_{\C})_v$ and let $H_{v,\R}=G_v$ be its real stabilizer. That is,  $H_{v,\mathbb{C}}$ and $H_{v,\mathbb{R}}$ are the subgroups fixing the point $v$ under the actions of the complex group $G_{\mathbb{C}}$ and the real group $G$, respectively.  The orbit
morphism $ G_{\C}\longrightarrow {\mathscr O}_{\C}$ given by
\[
  g\longmapsto g\cdot v,
\]
is surjective with kernel $H_{v,\C}$ and, in characteristic zero, identifies
the complex orbit with the algebraic homogeneous space
${\mathscr O}_{\C}\simeq G_{\C}/H_{v,\C}$.  Because $v$ is fixed by complex
conjugation and the action is defined over $\R$, the equations defining the
stabilizer are invariant under conjugation. Hence $H_{v,\C}$ is defined over
$\R$ and is the complexification of $H_{v,\R}$.  Thus
${\mathscr O}_{\C}$ is the complexification of the real homogeneous space
$G/H_{v,\R}$, and its set of real points is
$(G/H_{v,\R})(\R)={\mathscr O}_{\C}(\R)$.
Corollary~6.4 of Borel--Serre \cite{BorelSerre}, applied over the locally
compact field $\R$ to the homogeneous space $G/H_{v,\R}$, states that
$(G/H_{v,\R})(\R)$ is a union of only finitely many $G(\R)$-orbits. 
Consequently ${\mathscr O}_{\C}(\R)$ is a finite union of real $G$-orbits.
Each of these real orbits is preserved by the relation
\[
 W(\hat f_{\rm ext}\hat z)
 =\Lambda^q\hat A_{\rm ext}(\hat z)W(\hat z),
\]
because $\hat A_{\rm ext}(\hat z)\in G$. Hence the inverse images under
$W$ of these finitely many real orbits are invariant measurable sets.
Ergodicity therefore selects one real $G$-orbit, denoted by ${\mathscr O}$,
such that
\[
 W(\hat z)\in{\mathscr O}
\]
for almost every $\hat z$.

It remains to prove the estimate in \eqref{eq:relative-lift}. Fix
$v\in{\mathscr O}$. With its natural smooth manifold structure,
${\mathscr O}$ is the homogeneous space $G/G_v$, where $G_v$ is the subgroup
of $G$ that fixes $v$. Consequently, the orbit map
\[
 \alpha_v:G\longrightarrow{\mathscr O},
 \qquad
 \alpha_v(g)=gv,
\]
is a smooth submersion. A real algebraic $G$-orbit is a locally closed embedded smooth submanifold of $\mathscr M$. On compact subsets, an intrinsic smooth metric and the restriction of the ambient Grassmannian metric $d_{\Gr}$ are locally bi-Lipschitz. Thus smooth local sections are locally Lipschitz for the metric used in \eqref{eq:relative-lift}. By the local section theorem \cite[Theorem~4.26]{LeeSmooth}, there exist a neighborhood
$U_v\subset{\mathscr O}$ of $v$ and a smooth map
\[
 s_v:U_v\longrightarrow G
\]
such that
\[
 s_v(w)v=w\; \text{ and }\; s_v(v)=\Id.
\]
In other words, the map $s_v$ smoothly assigns to each neighboring point $w$ a specific group element in $G$ that transforms the base point $v$ precisely into $w$.
After shrinking $U_v$ if necessary, there are constants $L_v,M_v>0$ such that
\[
 \norm{s_v(V_2)-s_v(V_1)}
 \le L_v\,d_{\Gr}(V_1,V_2)
\]
and
\[
 \norm{s_v(V_1)^{-1}}\le M_v
\]
for all $V_1,V_2\in U_v$.

For $V_1,V_2\in U_v$, define
\[
 g=s_v(V_2)s_v(V_1)^{-1}.
\]
Since $s_v(V_i)v=V_i$, we have $gV_1=V_2$. Moreover,
\[
\begin{aligned}
 \norm{g-\Id}
 &=
 \norm{\bigl(s_v(V_2)-s_v(V_1)\bigr)s_v(V_1)^{-1}}\\
 &\le L_vM_v\,d_{\Gr}(V_1,V_2).
\end{aligned}
\]
Thus the required estimate holds whenever $V_1$ and $V_2$ belong to a common
neighborhood of this form.

Finally, choose compact sets $K_m\subset{\mathscr O}$ such that
\[
 \bigcup_m K_m={\mathscr O}.
\]
Fix $m$. By compactness of $K_m$, we can choose finitely many points
$v_1,\ldots,v_r$ and open sets
\[
 U'_{v_j}\subset U_{v_j} \; \text{ for }\; 1\le j\le r,
\]
such that
\[
 K_m\subset\bigcup_{j=1}^r U'_{v_j}
\]
and each $\overline{U'_{v_j}}$ is compact and contained in $U_{v_j}$.
Therefore one may choose $\rho_m>0$ so that, whenever
$V\in U'_{v_j}$ and
\[
 d_{\Gr}(V,V_i)<\rho_m \; \text{ for }\; i=1,2,
\]
both $V_1$ and $V_2$ belong to $U_{v_j}$. Set
\[
 C_m=\max_{1\le j\le r}L_{v_j}M_{v_j}.
\]
If $V\in K_m$ and $d_{\Gr}(V,V_i)<\rho_m$ for $i=1,2$, then
$V\in U'_{v_j}$ for some $j$, and the construction above gives
$g\in G$ such that
\[
 gV_1=V_2 \; \text{ and }\; 
 \norm{g-\Id}\le C_m\,d_{\Gr}(V_1,V_2).
\]
This is exactly \eqref{eq:relative-lift} and completes the proof of
Lemma~\ref{lem:orbit-lifting}.


\medskip{\bf Acknowledgments.}
L. Backes was partially supported by a CNPq-Brazil PQ fellowship under Grant No. 304806/2024-2. This work was also partially supported by FAPERGS - Programa Pesquisador Gaúcho - PqG under Grant No. 25/2551-0002627-0.
ChatGPT 5.6 was used to help draft Appendices \ref{app:reduction} and \ref{app:orbit-lemma} and to revise the manuscript.



\end{document}